\documentclass[10pt,reqno]{amsart}

\usepackage[latin2]{inputenc}
\usepackage{amsmath}
\usepackage{graphicx}
\usepackage{amssymb}
\usepackage{esint}
\usepackage{color}
\usepackage{amsthm}
\usepackage{epsfig}
\usepackage{enumitem}
\usepackage{mathtools}
\usepackage{esvect}
\usepackage{tikz}
\usepackage{dsfont}
\usetikzlibrary{calc,intersections,through,backgrounds,patterns,arrows.meta, math,through}
\usepackage[labelformat = brace, position=top]{subcaption}
\usepackage[english]{babel}
\usepackage{mathrsfs}

\newtheorem{theorem}{Theorem}
\newtheorem{proposition}[theorem]{Proposition}
\newtheorem{lemma}[theorem]{Lemma}
\newtheorem{corollary}[theorem]{Corollary}
\newtheorem*{theorem*}{Theorem}

\theoremstyle{definition}
\newtheorem{definition}[theorem]{Definition}
\newtheorem{remark}[theorem]{Remark}

\def\XXint#1#2#3{{\setbox0=\hbox{$#1{#2#3}{\int}$ }
\vcenter{\hbox{$#2#3$ }}\kern-.6\wd0}}

\definecolor{Yellow}{rgb}{0.95,0.9,0.0} 
\definecolor{Red}{rgb}{0.8,0.1,0.1}
\definecolor{Green}{rgb}{0.1,0.65,0.2}
\definecolor{Blue}{rgb}{0.1,0.1,0.8}
\definecolor{Purple}{rgb}{0.7,0.1,0.7}
\definecolor{Grey}{rgb}{0.6,0.6,0.6}
\definecolor{LightRed}{rgb}{0.8,0.5,0.5}
\definecolor{LightBlue}{rgb}{0.3,0.2,0.7}

\newcommand{\supp}{\operatorname{supp}}

\newcommand{\dist}{\operatorname{dist}}

\newcommand{\Hc}{\mathcal{H}}

\newcommand{\Rd}[1][d]{{\mathbb{R}^{#1}}}
\newcommand{\R}{\mathbb{R}}
\newcommand{\N}{\mathbb{N}}
\allowdisplaybreaks[4]

\newcommand{\dH}[1][d{-}1]{\,\mathrm{d}\mathcal{H}^{#1}}

\newcommand{\ds}{\,\mathrm{d}s}
\newcommand{\dx}{\,\mathrm{d}x}
\newcommand{\dt}{\,\mathrm{d}t}

\newcommand{\eps}{\varepsilon}
\renewcommand{\vec}[1]{{\operatorname{#1}}}

\numberwithin{equation}{section}
\usepackage[linktocpage=true,colorlinks=true,linkcolor=Blue,citecolor=Green]{hyperref}

\begin{document}

\title[Convergence rates for the Allen--Cahn equation]
{Convergence rates for the Allen--Cahn equation with boundary contact energy: The non-perturbative regime}

\author{Sebastian Hensel}
\address{Institute of Science and Technology Austria (IST Austria), Am~Campus~1, 
3400 Klosterneuburg, Austria}
\email{sebastian.hensel@ist.ac.at}
\curraddr{Hausdorff Center for Mathematics, Universit{\"a}t Bonn, Endenicher Allee 62, 53115 Bonn, Germany
(\texttt{sebastian.hensel@hcm.uni-bonn.de})}

\author{Maximilian Moser}
\address{Fakult\"at f\"ur Mathematik, Universit\"at Regensburg, Universit\"atsstrasse 31, 93053 Regensburg, Germany}
\email{maximilian1.moser@mathematik.uni-regensburg.de}
\curraddr{Institute of Science and Technology Austria (IST Austria), Am~Campus~1, 
	3400 Klosterneuburg, Austria (\texttt{maximilian.moser@ist.ac.at})}


\begin{abstract}
We extend the recent rigorous convergence result of Abels and the second author 
(\href{https://arxiv.org/abs/2105.08434}{arXiv preprint 2105.08434})
concerning convergence rates for solutions of the Allen--Cahn equation 
with a nonlinear Robin boundary condition towards evolution by mean curvature flow 
with constant contact angle. More precisely, in the present work we manage 
to remove the perturbative assumption on the contact angle being close to ninety degree.
We establish under usual double-well type assumptions on the potential and for a certain
class of boundary energy densities the sub-optimal convergence rate of order~$\smash{\eps^{\frac{1}{2}}}$
for general contact angles~$\alpha \in (0,\pi)$.
For a very specific form of the boundary energy density, we even obtain from our methods
a sharp convergence rate of order~$\eps$; again for general contact angles~$\alpha \in (0,\pi)$.

Our proof deviates from the popular strategy based on rigorous asymptotic expansions and stability
estimates for the linearized Allen--Cahn operator. Instead, we follow
the recent approach by Fischer, Laux and Simon 
(\href{https://doi.org/10.1137/20M1322182}{SIAM J.\ Math.\ Anal.\ 52, 2020}),
thus relying on a relative entropy technique. We develop a careful adaptation
of their approach in order to encode the constant contact angle condition.
In fact, we perform this task at the level of the notion of gradient flow calibrations.
This concept was recently introduced in the context of weak-strong uniqueness for multiphase
mean curvature flow by Fischer, Laux, Simon and the first author 
(\href{https://arxiv.org/abs/2003.05478}{arXiv preprint 2003.05478}).

\medskip
\noindent \textbf{Keywords:} Mean curvature flow, contact angle,  
boundary contact energy, Allen--Cahn equation, relative entropy method, 
gradient flow calibration

\medskip
\noindent \textbf{Mathematical Subject Classification}:
	53E10, 
	35K57, 
	35K20, 
	35K61 
\end{abstract}

\maketitle
\tableofcontents

\section{Introduction} \label{sec:intro}

\subsection{Context} Curvature driven interface evolution
arises in a broad range of applications, including
for instance liquid-solid interface evolution in solidification processes (e.g., \cite{Langer1980}),
noise removal and feature enhancement in image processing (e.g., \cite{Sapiro2001}),
flame front propagation in combustion processes (e.g., \cite{Markstein1951}), or  
grain coarsening in an annealing polycrystal (e.g., \cite{Mullins1956}).
The present work is concerned with the most basic mathematical
model representing the evolution of an interface (i.e., the common boundary
of a binary system) driven by an extrinsic curvature quantity, namely
evolution by mean curvature flow (MCF). Of course, this is a classical
subject in the literature, see, e.g., the seminal works by Gage and Hamilton~\cite{Gage1986} 
and Grayson~\cite{Grayson1987} for the flow of a smooth and simple closed curve in~$\R^2$.

The main focus of the present work is related to the rigorous treatment of
a certain class of nontrivial boundary effects. More precisely, 
we are concerned with the mean curvature
flow of an interface within a physical domain~$\Omega \subset \Rd$
(e.g., a container holding a binary alloy with a moving internal interface), 
so that the interface intersects the domain boundary~$\partial\Omega$ at a constant contact
angle~$\alpha \in (0,\pi)$; see Figure~\ref{fig_strong_solution}
for an illustration of the geometry. 
The inclusion of such a boundary condition 
poses an interesting and nontrivial mathematical 
problem because the evolving geometry is necessarily singular 
due to the contact set.

Mean curvature flow of an interface with constant contact angle can be generated
as the $L^2$-gradient flow of a suitable energy functional. The total energy
consists of two contributions: \textit{i)} interfacial energy in the interior
of the container, and \textit{ii)} surface energy
along the boundary of the container. 
Given a disjoint partition of the container
into two phases represented by an open subdomain $\mathscr{A}\subset\Omega$ 
and its open complement $\Omega\setminus\overline{\mathscr{A}}$, denote with~$I$ the associated 
interface given by the common boundary of these two sets (cf.\ again 
Figure~\ref{fig_strong_solution}). Expressing the associated
surface tension constants by~$c_0$, $\sigma_+$ and~$\sigma_-$,
respectively, the total energy is then given by
\begin{align*}
E[\mathscr{A}] := c_0 \int_{I} 1\,\mathrm{d}\mathcal{H}^{d-1}
+ \sigma_+ \int_{\partial\mathscr{A} \cap \partial\Omega} 1\,\mathrm{d}\mathcal{H}^{d-1}
+ \sigma_- \int_{\partial(\Omega\setminus\overline{\mathscr{A}}) \cap \partial\Omega} 
1\,\mathrm{d}\mathcal{H}^{d-1},
\end{align*} 
or alternatively by subtracting the constant 
$\sigma_- \int_{\partial\Omega} 1\,\mathrm{d}\mathcal{H}^{d-1}$
\begin{align}
\label{eq:sharpInterfaceEnergy}
E[\mathscr{A}] := c_0 \int_{I} 1\,\mathrm{d}\mathcal{H}^{d-1}
+ (\sigma_+ {-} \sigma_-) \int_{\partial\mathscr{A} \cap \partial\Omega} 1\,\mathrm{d}\mathcal{H}^{d-1}.
\end{align} 
The surface tension constants are assumed to satisfy Young's relation, i.e., 
\begin{align*}
|\sigma_+ {-} \sigma_-| < c_0,
\end{align*} 
so that in particular there exists an angle
$\alpha \in (0,\pi)$ such that 
\begin{align*}
\sigma_+ {-} \sigma_- = c_0\cos\alpha.
\end{align*}
Switching the roles of~$\mathscr{A}$ and~$\Omega\setminus\overline{\mathscr{A}}$,
we may of course assume without loss of generality that~$\alpha \in (0,\frac{\pi}{2}]$.
The geometric interpretation of~$\alpha$ is that it represents the angle
formally formed by the intersection of the interface~$I$ with the boundary of
the container~$\partial\Omega$ through the domain~$\Omega\setminus\overline{\mathscr{A}}$
(cf.\ again Figure~\ref{fig_strong_solution}).

\begin{figure}
	\def\svgwidth{0.5\linewidth}
\begingroup%
  \makeatletter%
  \providecommand\color[2][]{%
    \errmessage{(Inkscape) Color is used for the text in Inkscape, but the package 'color.sty' is not loaded}%
    \renewcommand\color[2][]{}%
  }%
  \providecommand\transparent[1]{%
    \errmessage{(Inkscape) Transparency is used (non-zero) for the text in Inkscape, but the package 'transparent.sty' is not loaded}%
    \renewcommand\transparent[1]{}%
  }%
  \providecommand\rotatebox[2]{#2}%
  \newcommand*\fsize{\dimexpr\f@size pt\relax}%
  \newcommand*\lineheight[1]{\fontsize{\fsize}{#1\fsize}\selectfont}%
  \ifx\svgwidth\undefined%
    \setlength{\unitlength}{751.92462591bp}%
    \ifx\svgscale\undefined%
      \relax%
    \else%
      \setlength{\unitlength}{\unitlength * \real{\svgscale}}%
    \fi%
  \else%
    \setlength{\unitlength}{\svgwidth}%
  \fi%
  \global\let\svgwidth\undefined%
  \global\let\svgscale\undefined%
  \makeatother%
  \begin{picture}(1,0.58085213)%
    \lineheight{1}%
    \setlength\tabcolsep{0pt}%
    \put(0,0){\includegraphics[width=\unitlength,page=1]{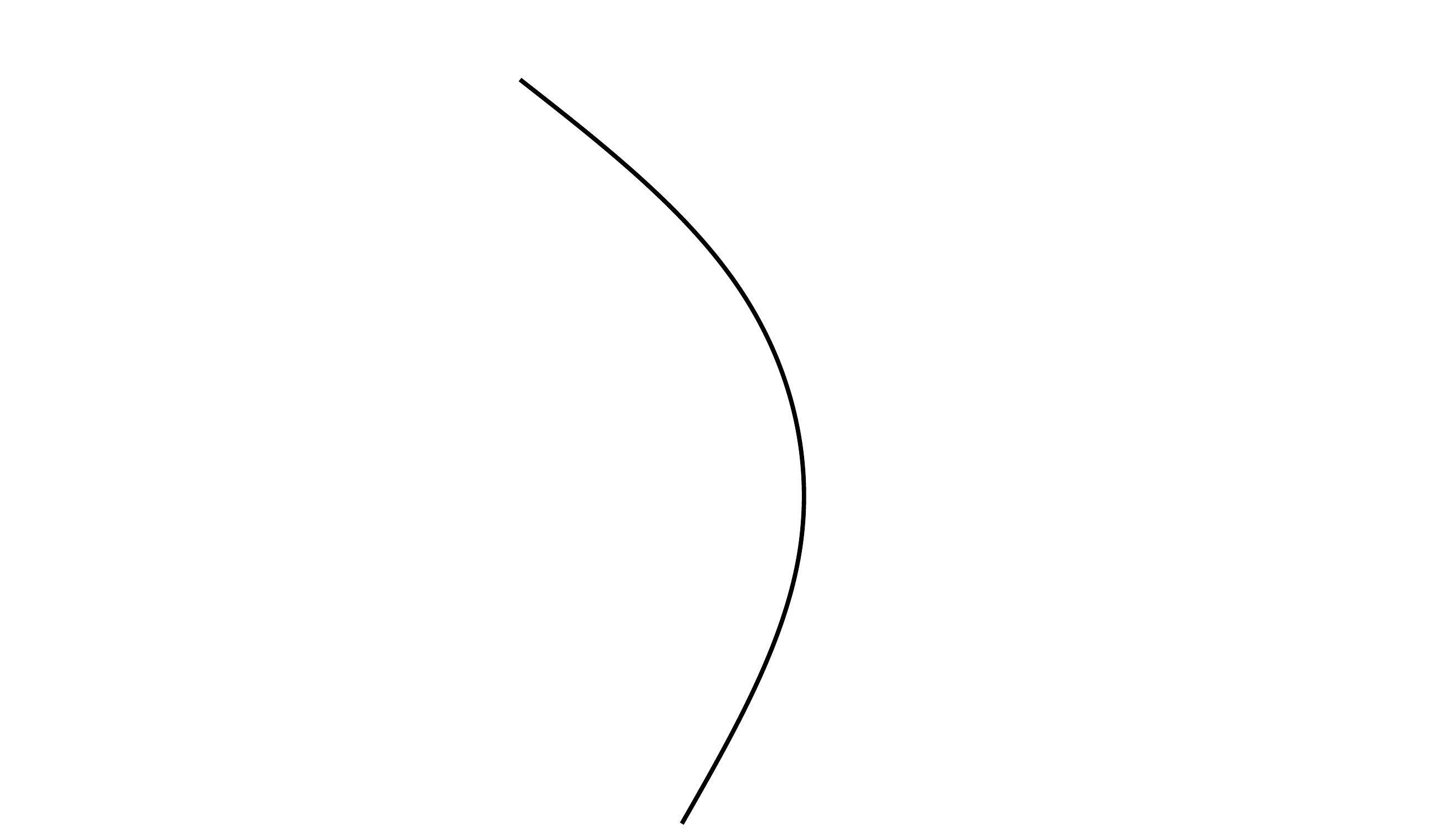}}%
    \put(0.84135978,0.15040361){\color[rgb]{0,0,0}\makebox(0,0)[lt]{\begin{minipage}{0.21801195\unitlength}\raggedright $\Omega$\end{minipage}}}%
    \put(0,0){\includegraphics[width=\unitlength,page=2]{strong_solution_picture2.pdf}}%
    \put(0.44609517,0.49999994){\color[rgb]{0,0,0}\makebox(0,0)[lt]{\lineheight{1.25}\smash{\begin{tabular}[t]{l}$\alpha$\end{tabular}}}}%
    \put(0,0){\includegraphics[width=\unitlength,page=3]{strong_solution_picture2.pdf}}%
    \put(0.53654134,0.07331495){\color[rgb]{0,0,0}\makebox(0,0)[lt]{\lineheight{1.25}\smash{\begin{tabular}[t]{l}$\alpha$\end{tabular}}}}%
    \put(0.092943,0.22997207){\color[rgb]{0,0,0}\makebox(0,0)[lt]{\begin{minipage}{0.39698135\unitlength}\raggedright $\mathscr{A}(t)$\end{minipage}}}%
    \put(0.52949483,0.40528996){\color[rgb]{0,0,0}\makebox(0,0)[lt]{\begin{minipage}{0.14307443\unitlength}\raggedright $I(t)$\end{minipage}}}%
    \put(0,0){\includegraphics[width=\unitlength,page=4]{strong_solution_picture2.pdf}}%
  \end{picture}%
\endgroup%

	\caption{Illustration of a prototypical geometry for interface evolution
	with constant contact angle.}\label{fig_strong_solution}
\end{figure}

As usual in the context of geometric evolution equations, the
corresponding flow in general can not avoid the occurrence
of topology changes and geometric singularities. For an
example specific to the framework of contact angle problems, one may
imagine an initially interior point of the interface
to touch the boundary of the container at a later time;
see~\cite[Figure~2]{Katsoulakis1995} for an illustration
of this scenario. It is for this reason that a global-in-time
representation of the dynamics is in general only possible
in a weaker form than the one provided by solution concepts relying on 
parametrized surfaces with boundary.

One popular approach in this direction consists of phase-field models
which are based on the introduction of a time-dependent order parameter taking values in the 
continuum~$[-1,1]$. Roughly speaking, the regions within the container~$\Omega$ 
in which the order parameter takes values close to~${+}1$ or~${-}1$ 
represent the two underlying evolving phases. The associated evolving interface
is in turn represented by the region in which the order parameter undergoes a transition 
between these two values. The relevant dynamics for the order parameter are again induced
by studying the (in our case $L^2$) gradient flow of an associated energy functional. 

Following the modeling in the sharp-interface regime, this energy also
consists of two contributions. Within the container~$\Omega$, we consider
the standard Cahn--Hilliard energy associated with a double-well type potential~$W$.
For the boundary contribution, we rely on the proposal of Cahn~\cite{Cahn1977}
and include a boundary contact energy in terms of a boundary energy density~$\sigma$.
Both contributions together then result in the following ansatz for the total energy
functional of the order parameter:
\begin{align}
\label{eq_energyPhaseField}
E_\eps[ u] :=  \int_{\Omega} \frac{\eps}{2}|\nabla u|^2
+ \frac{1}{\eps}W(u) \dx + \int_{\partial \Omega} \sigma(u) \dH.
\end{align}
The associated ($\smash{\frac{1}{\eps}}$-accelerated) $L^2$-gradient flow 
leads to the standard Allen--Cahn equation within the container~$\Omega$. 
Boundary effects along~$\partial\Omega$ are captured by a nonlinear 
Robin boundary condition; cf.\ \eqref{eq_AllenCahn}--\eqref{eq_initialData} below
for the full PDE~problem.

Concerning the static case, Modica~\cite{Modica1987} shows that
phase-field energies of the form~\eqref{eq_energyPhaseField} $\Gamma$-converge
to sharp-interface energies of the form~\eqref{eq:sharpInterfaceEnergy},
and thus relates the associated minimizers of these energy functionals in the limit~$\eps\searrow 0$.
The main goal of the present work is instead concerned with the corresponding dynamics.
It consists of a rigorous justification of the relation
of the $L^2$-gradient flows associated with the energies~\eqref{eq:sharpInterfaceEnergy}
and~\eqref{eq_energyPhaseField} in the limit $\eps\searrow 0$.
Computations based on formal asymptotic expansions by Owen and Sternberg~\cite{Owen1992}
suggest that solutions of the phase-field model based on~\eqref{eq_energyPhaseField}
converge to solutions of the sharp-interface model related with~\eqref{eq:sharpInterfaceEnergy}, 
i.e., mean curvature flow with constant contact angle.
The main result of the present work establishes this connection in a rigorous fashion 
for a certain class of double-well type potentials~$W$ and boundary energy densities~$\sigma$;
cf.\ Subsection~\ref{subsec:assumptions} below for precise assumptions.
To the best of our knowledge, our result is the first for which this is achieved
without any restriction on the value of the contact angle~$\alpha$.
Apart from the qualitative statement of convergence, we also establish
convergence rates as a consequence of a general quantitative stability estimate
between solutions of the phase-field model and solutions of its sharp-interface limit.
Within the full generality of our assumptions, these are suboptimal with
respect to the scaling in the parameter~$\eps$. However, for a specific choice
of the boundary energy density~$\sigma$, we even obtain optimal convergence rates.
We finally remark that our results hold true on a time horizon on which a sufficiently regular solution
to mean curvature flow with constant contact angle exists, i.e.,
prior to the occurrence of geometric singularities. We refer to Theorem~\ref{theo_mainResult} 
below for a complete mathematical statement.

Before we proceed in Subsection~\ref{subsec:assumptions} 
with a precise description of the mathematical setting and assumptions, 
let us first put our main result into the context of the existing literature.
In the most basic setting of the full-space problem $\Omega=\Rd$,
rigorous proofs for the convergence of solutions of the Allen--Cahn equation towards
solutions of MCF were already established several decades ago.
Evans, Soner and Souganidis~\cite{Evans1992} provide a
global-in-time convergence result based on the notion
of viscosity solutions for MCF, thus
relying in an essential way on the comparison principle.
The global-in-time convergence result of Ilmanen~\cite{ilmanen} instead
makes use of the notion of Brakke flows.
Only recently, Laux and Simon~\cite{LauxSimon18} succeeded in
deriving a conditional convergence result for the vectorial Allen--Cahn problem
whose sharp-interface limit is represented by multiphase~MCF.
Their result is phrased in terms of so-called BV~solutions
and is conditional due to a required energy convergence assumption
in the spirit of the seminal work by Luckhaus and Sturzenhecker~\cite{Luckhaus1995}.
Based on a natural varifold generalization of BV~solutions, 
even an unconditional convergence result holds true at least
in the two-phase regime as shown by Laux and the first author~\cite{Hensel2021e}. 
Finally, local-in-time convergence of solutions of the 
Allen--Cahn equation towards classical solutions of MCF
in the full-space setting~$\Omega=\Rd$ goes back to the seminal work of De~Mottoni and 
Schatzman~\cite{DeMottoni1995}. Their method is based on
rigorous asymptotic expansions as well as
stability estimates for the linearized Allen--Cahn operator.
%

When including boundary effects in form of constant contact angles,
the majority of the results in the existing literature treats the case
of vanishing boundary energy density~$\sigma=0$. In other words, 
a fixed-in-time ninety degree angle condition is prescribed for
the intersection of the interface with the boundary of the container.
In terms of the phase-field approximation, this modeling assumption 
leads to a homogeneous Neumann boundary condition for the order parameter.
Global-in-time convergence in this setting towards weak solutions
of~MCF interpreted in a viscosity sense is due to Katsoulakis, Kossioris and Reitich~\cite{Katsoulakis1995}.
A corresponding result with respect to a suitably generalized notion of
Brakke flows is derived by Mizuno and Tonegawa~\cite{Mizuno2015} (for strictly
convex and smooth containers) and Kagaya~\cite{Kagaya2018a} (for general
smooth containers). 

Local-in-time convergence results
in terms of smooth solutions to~MCF with constant ninety degree
angle condition were in turn established in a work of Chen~\cite{Chen1992}
and a recent work of Abels and the second author~\cite{Abels2019}.
The former relies on the construction of super- and subsolutions
of the Allen--Cahn equation as well as comparison principle arguments,
whereas the latter extends the method of De~Mottoni and 
Schatzman~\cite{DeMottoni1995} to the ninety degree contact angle setting;
see in this context also the work of the second author~\cite{Moser2021} for 
extensions of~\cite{Abels2019} in several directions.

We next comment on the literature in the
regime of general boundary energy densities~$\sigma$
modeling the case of general contact angles~$\alpha \in (0,\frac{\pi}{2}]$
in the sharp-interface limit. To the best of our knowledge,  
up to the present work no rigorous convergence result
allowing for arbitrary values of the contact angle has been established.
The only two results we are aware of consist of the
non-rigorous derivation of the sharp-interface limit by Owen and Sternberg~\cite{Owen1992}
as well as the recent work by Abels and the second author~\cite{Abels2021}, which
constitutes the first rigorous version of the formal arguments given
by Owen and Sternberg~\cite{Owen1992}. However, the results of~\cite{Abels2021} are restricted 
to a perturbative regime in the sense that the contact angle
is assumed to be close to ninety degrees. The present work does not rely on this requirement 
and therefore establishes for the first time a
local-in-time convergence proof for general contact angles~$\alpha \in (0,\frac{\pi}{2}]$,
which, similar to~\cite{Abels2021}, even provides convergence rates.
Note that in a companion article,
Laux and the first author~\cite{Hensel2021d} 
also prove a (purely qualitative) global-in-time convergence result
towards a novel notion of BV~solutions to~MCF
with general constant contact angle~$\alpha \in (0,\frac{\pi}{2}]$.

We conclude the discussion with some context on our methods.
In contrast to the work of Abels and the second author~\cite{Abels2021},
which makes use of rigorous asymptotic expansions and 
stability of the linearized Allen--Cahn operator in the spirit
of De~Mottoni and Schatzman~\cite{DeMottoni1995}, our proof
is directly inspired by the recent approach of 
Fischer, Laux and Simon~\cite{Fischer2020}. They employ a
novel relative entropy technique to prove, even in an optimally quantified way, 
local-in-time convergence of solutions of the 
full-space Allen--Cahn equation towards smooth solutions of~MCF.
Their technique is based on a natural phase-field analogue
of an error functional which has been extensively used
throughout recent years to study stability and weak-strong uniqueness
properties of weak solution concepts in interface evolution
problems on the sharp interface level. 

One version of this error functional,
which is supposed to measure the difference between two solutions
in a sufficiently strong sense, appeared for the first time
in the work of Jerrard and Smets~\cite{Jerrard2015} dealing
with binormal curvature flow of curves in~$\Rd[3]$. In a structurally analogous 
but slightly adapted form more suited for interface evolution, it was used
by Fischer and the first author~\cite{Fischer2020b} to establish
weak-strong uniqueness for a two-phase Navier--Stokes system with surface tension.
It was afterwards extended by Fischer, Laux, Simon and the first author~\cite{Fischer2020a}
to treat the case of planar multiphase~MCF (see also~\cite{Hensel2021} and~\cite{Hensel2021e}).
In the present work, we develop a careful adaptation of the approach by
Fischer, Laux and Simon~\cite{Fischer2020} to incorporate the 
contact angle condition. This is a nontrivial task due to the necessarily
singular nature of the geometry associated with a solution of~MCF
with constant contact angle. For a more detailed description
of our strategy, we refer to the discussion in Subsections~\ref{subsec:quantStab}
and~\ref{subsec:existenceCalibration} below.

\subsection{Assumptions and setting} \label{subsec:assumptions} 
In the present work, we study the convergence of solutions to the
Allen--Cahn equation with a nonlinear Robin boundary condition.
In its strong PDE formulation, the problem is given as follows:
\begin{subequations}
\begin{align}
\label{eq_AllenCahn}
\tag{AC1}
\partial_t u_\eps &= \Delta u_\eps - \frac{1}{\eps^2} W'( u_\eps)
&& \text{in } \Omega{\times}(0,T),
\\
\label{eq_RobinBoundaryCondition}
\tag{AC2}
(\mathrm{n}_{\partial \Omega}\cdot\nabla) u_\eps &= \frac{1}{\eps} \sigma'( u_\eps)
&& \text{on } \partial \Omega{\times} (0,T),
\\
\label{eq_initialData}
\tag{AC3}
u_\varepsilon|_{t=0} &= u_{\eps,0}
&& \text{in } \Omega.
\end{align}
\end{subequations}
Here, $\Omega\subset\Rd$ denotes a bounded (not necessarily convex) domain
with orientable and sufficiently regular boundary~$\partial\Omega$, the vector field
$\mathrm{n}_{\partial\Omega}$ denotes the associated inward pointing
unit normal, $T \in (0,\infty)$ is a finite time horizon, and 
$W\colon\Rd[]\to[0,\infty)$ is a standard free energy density (per unit volume) 
of double-well type whereas $\sigma\colon\Rd[]\to[0,\infty)$
denotes a boundary contact energy density (per unit surface area).
The latter two are assumed to be at least differentiable; more 
assumptions on~$W$ and~$\sigma$ will be imposed below.

As already mentioned previously, the Allen--Cahn problem~\eqref{eq_AllenCahn}--\eqref{eq_initialData}
can in fact be derived as the ($\smash{\frac{1}{\eps}}$-accelerated)
$L^2$-gradient flow of the free energy functional~\eqref{eq_energyPhaseField}.
In particular, sufficiently regular solutions to~\eqref{eq_AllenCahn}--\eqref{eq_initialData}
satisfy an energy dissipation equality of the form
\begin{align}
\label{eq_energyDissipationPhaseField}
E_\eps[ u_\eps(\cdot,T')] = E_\eps[ u_{\eps,0}]
- \int_0^{T'}\int_{\Omega} \frac{1}{\eps} H_\eps^2 \dx\dt
\end{align}
for all $T' \in [0,T]$, where the map~$H_\eps$ is defined by 
\begin{align}
\label{eq_curvaturePhaseField}
H_\eps := -\eps\Delta u_\eps + \frac{1}{\eps}W'( u_\eps).
\end{align}

We now specify our assumptions with respect to
the nonlinearities~$W$ and~$\sigma$. 
For the potential $W$, we impose $W\in C^2(\R)$ and the following conditions:
\begin{enumerate}[leftmargin=0.5cm]
\begin{subequations}
    \item[1.] $W$ has a double-well shape in the following sense:
    \begin{align}\label{eq_double_well1}
    W(\pm 1)=0, \quad W'(\pm1)=0, \quad W''(\pm1)>0, \quad  
    W>0\text{ in } \R\setminus\{\pm 1\}.
    \end{align}
	\item[2.] There exist $p\in[2,\infty)$ and constants $c,C,R>0$ such that 
	\begin{align}\label{eq_double_well2}
	c|u|^p\leq W(u)\leq C|u|^p \quad\text{ and }\quad |W'(u)|\leq C|u|^{p-1}\quad\text{ for all }|u|\geq R.
	\end{align}
	\item[3.] The decomposition $W=W_1+W_2$ holds with $W_1, W_2\in C^2(\R)$,
	\begin{align}\label{eq_double_well3}
	W_1\geq0\text{ convex }\quad\text{ and }\quad |W_2''|\leq C.
	\end{align}\end{subequations}
\end{enumerate}
Note that~\eqref{eq_double_well2} and~\eqref{eq_double_well3} represent
analogous assumptions as in~\cite{LauxSimon18}, 
where the vector-valued Allen-Cahn equation was considered
(see~\cite[Lemma~2.3]{LauxSimon18} for the existence of weak solutions in this case).
The standard choice satisfying the conditions~\eqref{eq_double_well1}--\eqref{eq_double_well3}
consists of course of $W(u) \sim (1 - u^2)^2$.

We next define
\begin{align}
\label{eq_auxIndicatorPhaseField}
\psi(r) &:= \int_{-1}^{r} \sqrt{2W(s)} \ds,
\quad r\in\Rd[],
\end{align}
as well as the interfacial surface tension constant
\begin{align}
\label{def_surfaceTension}
 c_0 := \int_{-1}^{1} \sqrt{2W(s)} \ds.
\end{align}
In view of the Modica--Mortola~\cite{Modica1977}/Bogomol'nyi~\cite{bogomolnyi} trick, 
the motivation behind this definition is
that the map~$\psi_\eps:=\psi( u_\eps)$ represents an approximation for
(a suitable multiple of) the indicator function of a phase with sharp interface
evolving by mean curvature flow. The boundary energy density is then assumed to satisfy
\begin{subequations}
\begin{align}
\label{eq_propBoundaryEnergyDensity}
\sigma\in C^{1,1}(\Rd[];[0,\infty)), \quad
\sigma' \geq 0 \text{ in } \Rd[], \quad
\supp\sigma'\subset [-1,1],
\end{align}
as well as
\begin{align}
\label{eq_propBoundaryEnergyDensity2}
\sigma(-1) = 0, 
\quad \sigma \geq \psi\cos\alpha \text{ on } [-1,1],
\quad \sigma(1) = \psi(1)\cos\alpha =  c_0\cos\alpha.
\end{align}
\end{subequations}
Due to~$\sigma(-1)=0$, the third item of~\eqref{eq_propBoundaryEnergyDensity2} 
in fact reads $\sigma(1) - \sigma(-1) =  c_0\cos\alpha$ and thus may be identified
with Young's law.

Under these assumptions on the potential~$W$ and the boundary energy density~$\sigma$,
we derive in the present work suboptimal convergence rates for solutions of
the Allen--Cahn problem~\eqref{eq_AllenCahn}--\eqref{eq_initialData} towards
smooth solutions of mean curvature flow with constant contact angle~$\alpha$
(cf.\ Theorem~\ref{theo_mainResult} below for a precise statement).
In order to achieve an optimal rate of convergence, our approach relies on a 
more restrictive assumption on the boundary energy density:
\begin{align}
\label{eq_choiceBoundaryEnergyDensity}
\sigma(r) := 
\begin{cases}
0 & r\in (-\infty,-1),
\\
\psi(r)\cos\alpha & r\in [-1,1],
\\
 c_0\cos\alpha & r \in (1,\infty).
\end{cases}
\end{align}
Note that~\eqref{eq_choiceBoundaryEnergyDensity} is obviously
consistent with~\eqref{eq_propBoundaryEnergyDensity} and~\eqref{eq_propBoundaryEnergyDensity2}.
%

\section{Main results \& definitions}
As already announced in the introduction, our main result
concerns the rigorous derivation of convergence rates for the Allen--Cahn
problem~\eqref{eq_AllenCahn}--\eqref{eq_initialData}
with well-prepared initial data towards the sharp interface limit
given by evolution by mean curvature flow 
with a constant contact angle $\alpha\in (0,\smash{\frac{\pi}{2}}]$. 
The precise statement reads as follows.

\begin{theorem}[Convergence rates for the Allen--Cahn problem~\eqref{eq_AllenCahn}--\eqref{eq_initialData}
towards strong solutions of mean curvature flow with constant contact 
angle~$0<\alpha\leq\smash{\frac{\pi}{2}}$]
\label{theo_mainResult} 
Consider a finite time horizon~$T\in (0,\infty)$ and a bounded $C^3$-domain $\Omega\subset\Rd[2]$,
and let $\mathscr{A}=\smash{\bigcup_{t\in [0,T]}\mathscr{A}(t){\times}\{t\}}$
be a strong solution to evolution by mean curvature flow in~$\Omega$ 
with constant contact angle~$\alpha\in (0,\smash{\frac{\pi}{2}}]$ in
the sense of Definition~\ref{def_strongSolution}. Denote for every $t\in[0,T]$ 
by~$\chi_{\mathscr{A}(t)}$ the characteristic function associated with~$\mathscr{A}(t)$.

Moreover, let a potential~$W$ and boundary energy density~$\sigma$ be
given such that the assumptions~\emph{\eqref{eq_double_well1}--\eqref{eq_double_well3}}
and~\emph{\eqref{eq_propBoundaryEnergyDensity}--\eqref{eq_propBoundaryEnergyDensity2}}  
are satisfied, respectively, and consider an initial phase 
field~$ u_{\eps,0}$ with finite energy~$E_\eps[ u_{\eps,0}]<\infty$
which moreover satisfies 
\begin{align}
\label{eq_boundInitialPhaseField}
& u_{\eps,0} \in [-1,1] \text{ almost everywhere in } \Omega.
\end{align}
Denote by~$u_\eps$ the associated weak solution
of the Allen--Cahn problem~\emph{\eqref{eq_AllenCahn}--\eqref{eq_initialData}}
in the sense of Definition~\ref{def_weakSolutionAllenCahn} (on a time horizon $> T$).

Then, there exists a constant $C=C(\mathscr{A},T) > 0$ such that it holds
\begin{align}
\label{eq_optimalConvergenceRates}
\big\|\psi\big( u_\eps(\cdot,T')\big) {-}  c_0\chi_{\mathscr{A}(T')}\big\|_{L^1(\Omega)}
\leq e^{CT'}\sqrt{E_{\mathrm{relEn}}\big[ u_{\eps,0}|\mathscr{A}(0)\big]
{+} E_{\mathrm{bulk}}[ u_{\eps,0}|\mathscr{A}(0)]}
\end{align}
for all $T'\in [0,T]$, where we recall 
from~\eqref{eq_auxIndicatorPhaseField} and~\eqref{def_surfaceTension}
the definition of~$\psi$ and~$c_0$, respectively.
For the definition of the relative energy functional~$E_{\mathrm{relEn}}$ 
and the bulk error functional~$E_{\mathrm{bulk}}$, we refer to~\eqref{eq_relEnergy} 
and~\eqref{eq_bulkError} below, respectively.

Furthermore, the class of finite energy initial phase fields 
satisfying~\eqref{eq_boundInitialPhaseField} and
\begin{align}
\label{eq_wellPreparedInitialRelativeEnergy}
&E_{\mathrm{relEn}}\big[ u_{\eps,0}|\mathscr{A}(0)\big]
+ E_{\mathrm{bulk}}[ u_{\eps,0}|\mathscr{A}(0)]
\lesssim \eps
\end{align}
is non-empty. In particular, for such well-prepared initial data
one obtains from the quantitative stability estimate~\eqref{eq_optimalConvergenceRates}
a suboptimal convergence rate of order $\eps^{\frac{1}{2}}$.
Finally, in case of the specific choice~\eqref{eq_choiceBoundaryEnergyDensity},
one may upgrade the requirement~\eqref{eq_wellPreparedInitialRelativeEnergy} from~$\eps$ to~$\eps^2$, and thus as a consequence the suboptimal convergence rate~$\eps^{\frac{1}{2}}$ 
to an optimal convergence rate of order~$\eps$.
\end{theorem}

\begin{proof}
First note that due to Theorem~\ref{theo_existenceBoundaryAdaptedCalibration} there exists a boundary adapted gradient flow calibration~$(\xi,B,\vartheta)$ with respect to a strong solution~$\mathscr{A}$ 
evolving by mean curvature flow in~$\Omega$ with constant contact 
angle~$\alpha\in (0,\smash{\frac{\pi}{2}}]$ in the sense of Definition~\ref{def_strongSolution}. 
Hence the estimate~\eqref{eq_optimalConvergenceRates} follows directly from a combination 
of the quantitative stability estimates
relative to a calibrated evolution, see Theorem~\ref{theo_convergenceRatesCalibratedPartition},
a post-processing of the latter based on Lemma~\ref{lem_coercivityBulkError},
and Gronwall's inequality.

The assertions with respect to the existence of well-prepared initial phase fields
are part of Lemma~\ref{lem_wellPreparedInitialData}.
\end{proof}

\subsection{Quantitative stability with respect to calibrated evolutions in $d\geq 2$}
\label{subsec:quantStab}
Our approach to the proof of Theorem~\ref{theo_mainResult} 
is directly inspired by the recent work~\cite{Fischer2020} of Fischer, Laux and Simon,
who establish the same result in a full space setting. In contrast to other approaches (cf.\ Section~\ref{sec:intro}), 
they capitalize on a novel relative
entropy technique. Their strategy can be interpreted as a diffuse interface
analogue of the relative entropy approach to weak-strong uniqueness
for certain mean curvature driven sharp interface evolution problems 
as introduced in~\cite{Fischer2020b} by Fischer and the first author
(cf.\ also the earlier work~\cite{Jerrard2015} of Jerrard and Smets for a similar approach in the 
setting of a codimension two evolution problem).

However, in comparison to the work~\cite{Fischer2020} of Fischer, Laux and Simon,
we will employ a conceptually more general viewpoint by splitting the task into
two separate steps. This two-step procedure is directly inspired by the
recent work~\cite{Fischer2020a} of Fischer, Laux, Simon and the first author
on weak-strong uniqueness for planar multiphase mean curvature flow
(cf.\ also the work~\cite{Hensel2021} of Laux and the first author).
The first step concerns the notion of a calibrated evolution along the
gradient flow of an interfacial energy, which in a sense generalizes the classical notion
of calibrations from minimal surface theory to an evolutionary setting,
and to prove quantitative stability of solutions to~\eqref{eq_AllenCahn}--\eqref{eq_initialData}
with respect to such calibrated evolutions. The second step then consists of
showing that sufficiently regular solutions to mean curvature flow with
constant contact angle are in fact calibrated, so that the stability estimates
from the first step can be used to yield the asserted convergence rate. 

The following definition represents a generalization of the two-phase versions 
of~\cite[Definition~2 and Definition~4]{Fischer2020a} 
in order to encode the correct constant contact angle condition for the
intersection of the evolving interface with the boundary of the container.

\begin{definition}[Calibrated evolutions and boundary adapted gradient flow calibrations]
\label{def_boundaryAdaptedCalibration}
Let $T\in (0,\infty)$ be a finite time horizon and let~$\Omega$
be a bounded $C^2$-domain in~$\R^d$. 
Consider $\mathscr{A}=\smash{\bigcup_{t\in [0,T]}\mathscr{A}(t){\times}\{t\}}$
such that for each $t\in [0,T]$ the set~$\mathscr{A}(t)$
is an open subset of~$\Omega$ with finite perimeter in~$\Rd$ 
and the closure of $\partial^*\mathscr{A}(t) \subset \overline{\Omega}$
is given by $\partial\mathscr{A}(t)$. Denote for all $t\in [0,T]$ 
by~$\mathrm{n}_{\partial^*\mathscr{A}(t)}$ the measure-theoretic
unit normal along~$\partial^*\mathscr{A}(t)$ pointing inside~$\mathscr{A}(t)$.
Writing~$\chi(\cdot,t)$ for the characteristic function associated with~$\mathscr{A}(t)$,
we assume that $\chi\in BV(\Rd{\times}(0,T)) \cap C([0,T];L^1(\Rd))$
and that the measure $\partial_t\chi$ is absolutely continuous 
with respect to the measure~$|\nabla\chi|$ restricted 
to~$\smash{\bigcup_{t\in (0,T)}(\partial^*\mathscr{A}(t) \cap \Omega){\times}\{t\}}$ 
(i.e., the associated Radon--Nikod\'{y}m derivative yields a normal speed).
Let $\alpha\in (0,\smash{\frac{\pi}{2}}]$ and~$ c_0>0$ be two constants. 

We then call~$\mathscr{A}=\smash{\bigcup_{t\in [0,T]}\mathscr{A}(t){\times}\{t\}}$ 
a \emph{calibrated evolution} for the $L^2$-gradient flow of the sharp interface energy functional
\begin{align}
E[\mathscr{A}(t)] :=  c_0 \int_{\partial^*\mathscr{A}(t) \cap \Omega} 1 \, \dH
+  c_0 \int_{\partial^*\mathscr{A}(t) \cap \partial \Omega} \cos\alpha \, \dH
\end{align}
if there exists a triple $(\xi,B,\vartheta)$ of maps as well as 
constants $c\in (0,1)$ and $C>0$ subject to the following conditions.
First, concerning regularity it is required that
\begin{subequations}
\begin{align}
\label{eq_regularityXi}
\xi &\in C^1\big(\overline{\Omega}{\times}[0,T];\Rd\big)
\cap C\big([0,T];C^2_{\mathrm{b}}(\Omega;\Rd)\big),
\\ \label{eq_regularityB}
B &\in C\big([0,T];C^1(\overline{\Omega};\Rd)\cap C^2_{\mathrm{b}}(\Omega;\Rd)\big),
\\ \label{eq_regularityWeight}
\vartheta &\in C^1_{\mathrm{b}}\big(\Omega{\times}[0,T]\big) 
\cap C\big(\overline{\Omega}{\times}[0,T];[-1,1]\big).
\end{align}
\end{subequations}
Second, for each~$t\in [0,T]$ the vector field~$\xi(\cdot,t)$ models an extension of the 
unit normal of~$\partial^*\mathscr{A}(t) \cap \Omega$ and the vector field~$B(\cdot,t)$ models 
an extension of a velocity vector field of~$\partial^*\mathscr{A}(t) \cap \Omega$ in the precise
sense of the conditions
\begin{subequations}
\begin{align}
\label{eq_consistencyProperty}
\xi(\cdot,t) &= \mathrm{n}_{\partial^*\mathscr{A}(t)}
\text{ and } \big(\nabla\xi(\cdot,t)\big)^\mathsf{T}\mathrm{n}_{\partial^*\mathscr{A}(t)} = 0
&& \text{along } \partial^*\mathscr{A}(t)\cap \Omega, 
\\
\label{eq_quadraticLengthConstraint}
|\xi|(\cdot,t) &\leq 1 {-} c\min\big\{1,\dist^2\big(\cdot,\overline{\partial^*\mathscr{A}(t)\cap \Omega}\big)\big\}
&& \text{in } \Omega,
\end{align}
as well as
\begin{align}\label{eq_calibration1}
|\partial_t\xi + (B\cdot\nabla)\xi + (\nabla B)^\mathsf{T}\xi|(\cdot,t)
&\leq C\min\big\{1,\dist\big(\cdot,\overline{\partial^*\mathscr{A}(t)\cap \Omega}\big)\big\}
&& \text{in } \Omega,
\\\label{eq_calibration2}
|\xi\cdot(\partial_t\xi+(B\cdot\nabla)\xi)|(\cdot,t)
&\leq C\min\big\{1,\dist^2\big(\cdot,\overline{\partial^*\mathscr{A}(t)\cap \Omega}\big)\big\}
&& \text{in } \Omega,
\\ \label{eq_calibrationEvolByMCF}
|\xi\cdot B + \nabla\cdot\xi|(\cdot,t) 
&\leq C\min\big\{1,\dist\big(\cdot,\overline{\partial^*\mathscr{A}(t)\cap \Omega}\big)\big\}
&& \text{in } \Omega,
\\
|\xi\cdot(\xi\cdot\nabla) B|(\cdot,t) 
&\leq C\min\big\{1,\dist\big(\cdot,\overline{\partial^*\mathscr{A}(t)\cap \Omega}\big)\big\}
&& \text{in } \Omega,\label{eq_calibration4}
\end{align}
which are accompanied by the (natural) boundary conditions
\begin{align}
\label{eq_boundaryCondXi}
\xi(\cdot,t)\cdot\mathrm{n}_{\partial \Omega}(\cdot) &= \cos \alpha
&&\text{along } \partial \Omega,
\\\label{eq_boundaryCondVelocity}
B(\cdot,t)\cdot\mathrm{n}_{\partial \Omega}(\cdot) &= 0
&&\text{along } \partial \Omega.
\end{align}
\end{subequations}
Third, for all~$t\in [0,T]$ the weight~$\vartheta(\cdot,t)$ models 
a truncated and sufficiently regular ``signed distance'' of~$\partial^*\mathscr{A}(t)\cap \Omega$ 
in the sense that 
\begin{subequations}
\begin{align}
\label{eq_weightNegativeInterior}
\vartheta(\cdot,t) &< 0 
&&\text{in the essential interior of } \mathscr{A}(t) \text{ within } \Omega,
\\ \label{eq_weightPositiveExterior}
\vartheta(\cdot,t) &> 0 
&&\text{in the essential exterior of } \mathscr{A}(t),
\\ \label{eq_weightZeroInterface}
\vartheta(\cdot,t) &= 0
&&\text{along } \partial^*\mathscr{A}(t) \cap \Omega,
\end{align}
as well as
\begin{align}
\label{eq_weightCoercivity}
\min\{\dist(\cdot,\partial\Omega),
\dist\big(\cdot,\overline{\partial^*\mathscr{A}(t) \cap \Omega}\big), 1
\big\} &\leq C|\vartheta|(\cdot,t)
&&\text{in } \Omega,
\\
\label{eq_weightEvol}
|\partial_t\vartheta + (B\cdot\nabla)\vartheta|(\cdot,t)
&\leq C|\vartheta|(\cdot,t)
&&\text{in } \Omega. 
\end{align}
\end{subequations}

Given a calibrated evolution~$\mathscr{A}=\smash{\bigcup_{t\in [0,T]}\mathscr{A}(t){\times}\{t\}}$,
an associated triple~$(\xi,B,\vartheta)$ subject to these requirements
is called a \emph{boundary adapted gradient flow calibration}.
\end{definition}		

We remark that the second property in~\eqref{eq_consistencyProperty} only
enters the proof of Lemma \ref{lem_wellPreparedInitialData} on the existence 
of well-prepared initial data in the sense of Theorem \ref{theo_mainResult},
and thus is in principle only needed at the initial time $t=0$. 
Note also that for sufficiently small $c$ in~\eqref{eq_quadraticLengthConstraint} 
there is no contradiction with~\eqref{eq_boundaryCondXi}.

Keeping in mind that the vector field~$\xi$ models an extension of the
unit normal vector field of the evolving interface whereas~$B$ models
an extension of an associated velocity vector field, the boundary
conditions~\eqref{eq_boundaryCondXi} and~\eqref{eq_boundaryCondVelocity}
are natural. Indeed, the former simply encodes the constant contact
angle condition	along the evolving contact set whereas the latter
is directly motivated by the fact that the evolution of the contact set
occurs within the domain boundary. Note also that condition~\eqref{eq_calibrationEvolByMCF}
is then the only requirement in the previous definition which formally
makes a connection to evolution by mean curvature flow.

The merit of Definition~\ref{def_boundaryAdaptedCalibration} consists of the fact
that it already implies a rigorous justification of the heuristic that
solutions to the Allen--Cahn problem~\eqref{eq_AllenCahn}--\eqref{eq_initialData} 
with well-prepared initial data
represent an approximation to evolution by mean curvature flow with
constant contact angle (for a non-rigorous derivation based on formally
matched asymptotic expansions, see Owen and Sternberg~\cite{Owen1992}).
More precisely, we show that solutions to the Allen--Cahn problem~\eqref{eq_AllenCahn}--\eqref{eq_initialData}
can in a way be interpreted as stable perturbations of a calibrated evolution
(as measured in the sense of a relative energy).
						
\begin{theorem}[Quantitative stability for the 
Allen--Cahn problem~\eqref{eq_AllenCahn}--\eqref{eq_initialData} 
with respect to a calibrated evolution]
\label{theo_convergenceRatesCalibratedPartition}
Consider a finite time horizon~$T\in (0,\infty)$ and a bounded $C^2$-domain $\Omega\subset\Rd$,
fix a contact angle $\alpha\in (0,\smash{\frac{\pi}{2}}]$, and 
let $\mathscr{A}=\smash{\bigcup_{t\in [0,T]}\mathscr{A}(t){\times}\{t\}}$ be a calibrated
evolution with respect to this data in the sense
of Definition~\ref{def_boundaryAdaptedCalibration}.
Furthermore, let a potential~$W$ as well as a boundary energy density~$\sigma$ be
given such that the assumptions~\emph{\eqref{eq_double_well1}--\eqref{eq_double_well3}}
and~\emph{\eqref{eq_propBoundaryEnergyDensity}--\eqref{eq_propBoundaryEnergyDensity2}}  
are satisfied, respectively. Consider finally an initial 
phase field~$ u_{\eps,0}\in H^1(\Omega)$ with finite energy~$E_\eps[ u_{\eps,0}]<\infty$
such that $ u_{\eps,0} \in [-1,1]$ almost everywhere in~$\Omega$.

Then, denoting by~$ u_\eps$ the associated weak solution
of the Allen--Cahn problem~\emph{\eqref{eq_AllenCahn}--\eqref{eq_initialData}}
in the sense of Definition~\ref{def_weakSolutionAllenCahn}, by~$\chi$
the time-dependent characteristic function associated with~$\mathscr{A}$,
as well as by~$E_{\mathrm{relEn}}[ u_\eps|\mathscr{A}]$ and $E_{\mathrm{bulk}}[ u_\eps|\mathscr{A}]$
the relative energy functional and bulk error functional defined by~\eqref{eq_relEnergy}
and~\eqref{eq_bulkError}, respectively, there exists a constant~$C = C(\mathscr{A},T) > 0$ such that
for all $T'\in [0,T]$
\begin{align*}
E_{\mathrm{relEn}}[ u_\eps|\mathscr{A}](T')
&\leq E_{\mathrm{relEn}}[ u_\eps|\mathscr{A}](0)
+ C\int_0^{T'} E_{\mathrm{relEn}}[ u_\eps|\mathscr{A}](t) \,\dt,
\\
E_{\mathrm{bulk}}[ u_\eps|\mathscr{A}](T')
&\leq (E_{\mathrm{relEn}}{+}E_{\mathrm{bulk}})[ u_\eps|\mathscr{A}](0)
+ C\int_0^{T'} (E_{\mathrm{relEn}}{+}E_{\mathrm{bulk}})[ u_\eps|\mathscr{A}](t) \,\dt.
\end{align*}
\end{theorem}

Apart from the above quantitative stability result in terms of
the phase-field approximation, we remark that
a calibrated evolution in the sense of Definition~\ref{def_boundaryAdaptedCalibration}
also gives rise to a weak-strong uniqueness principle for a notion of BV~solutions
to evolution by mean curvature flow with constant contact angle.
This is made precise in a paper by Laux and the first author~\cite{Hensel2021d}
(for a major part of the argument, one may already consult Subsection~2.3.3 of the 
PhD~thesis~\cite{Hensel2021a} of the first author).

\subsection{Existence of boundary adapted gradient flow calibrations in $d=2$}
\label{subsec:existenceCalibration}
In view of Theorem~\ref{theo_convergenceRatesCalibratedPartition},
it essentially remains to show in a second step that sufficiently regular solutions to
evolution by mean curvature flow with constant contact
angle admit a boundary adapted gradient flow calibration.
This is the content of the following result, which is stated in the planar setting
for simplicity only. We expect an extension to the $d=3$ case (i.e.,
an evolving contact line) to be rather straightforward; definitely
less involved than the triple line construction in the recent
work~\cite{Hensel2021} of Laux and the first author.
For a related (yet again planar) construction in the case of two-phase Navier--Stokes
flow with constant ninety degree contact angle, we refer to the recent 
work~\cite{Fischer2021} of Marveggio and the first author.

\begin{theorem}[Strong solutions of planar mean curvature flow with constant contact 
angle~$0<\alpha\leq\smash{\frac{\pi}{2}}$ are calibrated]
\label{theo_existenceBoundaryAdaptedCalibration}
Fix a finite time horizon~$T\in (0,\infty)$ and a bounded $C^3$-domain $\Omega\subset\Rd[2]$,
and let $\mathscr{A}=\smash{\bigcup_{t\in [0,T]}\mathscr{A}(t){\times}\{t\}}$
be a strong solution to evolution by mean curvature flow in~$\Omega$ 
with constant contact angle~$\alpha\in (0,\smash{\frac{\pi}{2}}]$ in
the sense of Definition~\ref{def_strongSolution}. Then, the evolution given by~$\mathscr{A}$
is calibrated in the sense of Definition~\ref{def_boundaryAdaptedCalibration}.
\end{theorem}

Even though not needed for the goals of the present work, we remark
that our construction of the pair of vector fields~$(\xi,B)$ 
satisfies the following additional conditions, which may become handy 
for potential future purposes:
\begin{align}
\label{eq:GradVelXiXi}
(\xi \cdot \nabla^\mathrm{sym}B)(\cdot,t) &= 0
&& \text{along } \partial\Omega,
\\ \label{eq:GradVelTangent}
(\mathrm{n}_{\partial\Omega} \cdot \nabla^\mathrm{sym}B)(\cdot,t) &= 0
&& \text{along } \partial\Omega,
\\ \label{eq:skewSymmetryGradVelInt}
|\xi \cdot \nabla^\mathrm{sym} B|(\cdot,t) &\leq C 
\min\big\{1,\dist\big(\cdot,\overline{\partial^*\mathscr{A}(t)\cap \Omega}\big)\big\}
&& \text{in } \Omega
\end{align}
for all $t \in [0,T]$. 
A proof of these three conditions is contained in the proof of 
Theorem~\ref{theo_existenceBoundaryAdaptedCalibration}.

\subsection{Weak solutions to the Allen--Cahn problem~\eqref{eq_AllenCahn}--\eqref{eq_initialData}}
\label{subsec_weakSolAllenCahn}
In this subsection, we introduce the definition of a weak solution concept for
the Allen--Cahn problem~\eqref{eq_AllenCahn}--\eqref{eq_initialData}.

\begin{definition}[Weak solutions of the Allen--Cahn problem~\eqref{eq_AllenCahn}--\eqref{eq_initialData}]
\label{def_weakSolutionAllenCahn}
We consider a finite time horizon $T\in (0,\infty)$, a potential $W$ that satisfies 
\eqref{eq_double_well2}--\eqref{eq_double_well3}, 
a boundary energy density $\sigma$ subject to the properties~\eqref{eq_propBoundaryEnergyDensity},
and an initial phase field $ u_{\eps,0}\in H^1(\Omega)$ with finite energy~$E_\eps[ u_{\eps,0}]<\infty$.

We call a measurable map $ u_\eps\colon \Omega{\times}[0,T]\to\Rd[]$
an associated weak solution of the Allen--Cahn problem~\eqref{eq_AllenCahn}--\eqref{eq_initialData}
if it satisfies the following conditions.
First, in terms of regularity we require that
\begin{subequations}
\begin{align}
\label{eq_regularityWeakSolution}
 u_\eps &\in H^1\big(0,T;L^2(\Omega)\big) 
\cap L^\infty\big(0,T;H^1(\Omega)\cap L^p(\Omega)\big).
\end{align}
Second, the evolution problem~\eqref{eq_AllenCahn}--\eqref{eq_RobinBoundaryCondition}
is satisfied in weak form of
\begin{align}
\label{eq_evolEquationWeakSolution}
&\int_{0}^{T'}\int_{\Omega} \zeta \partial_t u_\eps \dx \dt 
+ \int_{0}^{T'}\int_{\Omega} \nabla \zeta \cdot \nabla u_\eps  \dx \dt
\\& \nonumber
= - \int_{0}^{T'}\int_{\partial \Omega} \zeta \frac{1}{\eps}\sigma'( u_\eps) \dH \dt
- \int_{0}^{T'}\int_{\Omega} \zeta \frac{1}{\eps^2} W'( u_\eps) \dx \dt
\end{align}
for all $T'\in (0,T)$ and all $\zeta\in C^\infty_{\mathrm{cpt}}\big([0,T);C^\infty(\overline{\Omega})\big)$,
whereas the initial condition~\eqref{eq_initialData} is achieved in form of
\begin{align}
\label{eq_initialDataWeakSolution}
 u_\eps(\cdot,0) =  u_{\eps,0} \quad\text{almost everywhere in } \Omega.
\end{align}
\end{subequations}
\end{definition}

Existence of weak solutions in the sense
of the previous definition will be established
by means of a minimizing movements scheme. More
precisely, we obtain

\begin{lemma}[Existence of weak solutions]
\label{lem_existenceWeakSolutionAllenCahn}
Let $T\in (0,\infty)$ be a finite time horizon, let $W$ be a potential 
with \emph{\eqref{eq_double_well2}--\eqref{eq_double_well3}}, let~$\sigma$ be a 
boundary energy density with the properties~\eqref{eq_propBoundaryEnergyDensity},
and let $ u_{\eps,0}\in H^1(\Omega)$ be an initial phase field with finite energy~$E_\eps[ u_{\eps,0}]<\infty$.
Then there exists an associated unique weak solution
of the Allen--Cahn problem~\emph{\eqref{eq_AllenCahn}--\eqref{eq_initialData}}
in the sense of Definition~\ref{def_weakSolutionAllenCahn}. 

If the initial phase field in addition satisfies $ u_{\eps,0}\in [-1,1]$
a.e.~in~$\Omega$, then the associated weak solution~$ u_\eps$
of the Allen--Cahn problem~\emph{\eqref{eq_AllenCahn}--\eqref{eq_initialData}}
is subject to
\begin{align}
\label{eq_comparisonPrincipleWeakSolution}
 u_\eps(\cdot,T') \in [-1,1]
\text{ a.e. in } \Omega
\end{align}
for all $T'\in [0,T]$.
\end{lemma}

As usual in the context of a minimizing movements scheme,
the associated energy estimate is short by a factor of~$2$
with respect to the sharp energy dissipation principle,
which is crucial for our purposes. If one does not want
to make use of De~Giorgi's variational interpolation and the concept of metric slope,
an alternative way to proceed is by means of higher regularity of weak solutions
(which we anyway rely on in the derivation of the estimate of the time evolution of the
relative energy). For our purposes, it suffices to prove the following result.

\begin{lemma}[Higher regularity for bounded weak solutions]
\label{lem_higherRegularityWeakSolutionAllenCahn}
In the situation of Lemma~\ref{lem_existenceWeakSolutionAllenCahn},
assume in addition that the initial phase field satisfies
$ u_{\eps,0}\in [-1,1]$ almost everywhere in~$\Omega$.
Then, the associated weak solution~$ u_\eps$
of the Allen--Cahn problem~\emph{\eqref{eq_AllenCahn}--\eqref{eq_initialData}}
satisfies the higher regularity
\begin{align}
\label{eq_higherRegularityPhaseField}
 u_\eps \in L^2\big(0,T;H^2(\Omega)\big)\cap C([0,T];H^1(\Omega)), 
\quad
\nabla\partial_t u_\eps \in L^2_{\mathrm{loc}}\big(0,T;L^2(\Omega)\big).
\end{align}
In particular, it holds
\begin{align}
\label{eq_weakSolutionsAreStrongSolutions}
\partial_t  u_\eps = \Delta u_\eps
- \frac{1}{\eps^2}W'( u_\eps) 
\quad\text{almost everywhere in } \Omega{\times}(0,T),
\end{align}
as well as
\begin{align}
\label{eq_weakFormRobinCondition}
\int_{\Omega} \zeta\,\Delta u_\eps(\cdot,T') \dx &=
- \int_{\Omega} \nabla\zeta\cdot\nabla  u_\eps(\cdot,T') \dx
- \int_{\partial \Omega} \zeta\,\frac{1}{\eps}\sigma'\big( u_\eps(\cdot,T')\big) \dH
\end{align}
for all~$\zeta\in C^\infty(\overline{\Omega})$ and almost every $T'\in (0,T)$.
\end{lemma}

With the previous regularity statement in place, we may then
establish the required sharp energy dissipation principle
which, as a consequence of the higher regularity, even occurs
as an identity.

\begin{lemma}[Energy dissipation equality for bounded weak solutions]
\label{lem_energyDissipationEquality}
In the situation of Lemma~\ref{lem_existenceWeakSolutionAllenCahn},
assume in addition that $ u_{\eps,0}\in [-1,1]$ almost everywhere in~$\Omega$.
Then, for the associated weak solution~$ u_\eps$
of the Allen--Cahn problem~\emph{\eqref{eq_AllenCahn}--\eqref{eq_initialData}},
the energy dissipation principle~\eqref{eq_energyDissipationPhaseField}
holds true in form of the following equality
\begin{align}
\label{eq_energyDissipationWeakSolution}
E_\eps[ u_\eps(\cdot,T')] 
+ \int_{0}^{T'} \int_{\Omega} \eps\big|\partial_t  u_\eps\big|^2 \dx \dt
= E_\eps[ u_{\eps,0}]
\end{align}
for all $T'\in (0,T)$.
\end{lemma}

Proofs for the previous three results can be found in Appendix~\ref{sec_appA}.
We conclude this subsection on weak solutions for the Allen--Cahn problem~\eqref{eq_AllenCahn}--\eqref{eq_initialData}
by mentioning that the set of well-prepared initial data 
as formalized in the statement of Theorem~\ref{theo_mainResult} is indeed non-empty.
The construction of a well-prepared initial phase field
is deferred until Appendix~\ref{sec_appB}.

\begin{lemma}[Existence of well-prepared initial data]
\label{lem_wellPreparedInitialData}
Consider a finite time horizon~$T\in (0,\infty)$ and a bounded $C^{2}$-domain $\Omega\subset\Rd[2]$, 
and let $\mathscr{A}=\smash{\bigcup_{t\in [0,T]}\mathscr{A}(t){\times}\{t\}}$
be a strong solution to evolution by mean curvature flow in~$\Omega$ 
with constant contact angle~$\alpha\in (0,\smash{\frac{\pi}{2}}]$ in
the sense of Definition~\ref{def_strongSolution}. Let a boundary energy density~$\sigma$ be
given such that~\emph{\eqref{eq_propBoundaryEnergyDensity}--\eqref{eq_propBoundaryEnergyDensity2}} hold true.
 
Then there exists an initial phase field~$ u_{\eps,0}$ with finite energy~$E_\eps[ u_{\eps,0}]<\infty$ which is well-prepared with respect to~$\mathscr{A}(0)$ in the
precise sense of~\emph{\eqref{eq_boundInitialPhaseField} 
and~\eqref{eq_wellPreparedInitialRelativeEnergy}}. In case of the specific choice~\eqref{eq_choiceBoundaryEnergyDensity},
one may upgrade the requirement~\eqref{eq_wellPreparedInitialRelativeEnergy} to~$\eps^2$.
\end{lemma}

\subsection{Definition of strong solutions to planar MCF with contact angle}
For completeness, we make precise what we mean by a sufficiently regular
solution to evolution by mean curvature flow with a constant contact angle.
We model the evolving geometry by the space-time track 
$\mathscr{A}=\bigcup_{t\in [0,T]} \mathscr{A}(t)\times\{t\}$
of a time-dependent family~$(\mathscr{A}(t))_{t \in [0,T]}$ of
sufficiently regular open sets in~$\Omega$.
For simplicity only, we will reduce ourselves to the most basic topological setup:
the phase~$\mathscr{A}(t)$ consists of only one connected component and
the associated interface $I(t):=\overline{\partial^\ast\mathscr{A}(t)\cap\Omega}$
is a sufficiently regular connected curve with exactly two distinct boundary points
which in turn are located on~$\partial\Omega$;
recall Figure~\ref{fig_strong_solution} for a sketch.
We emphasize that the chosen setup already involves all the major difficulties.

\begin{definition}[Strong solutions of planar mean curvature flow with constant contact angle~$0<\alpha\leq\smash{\frac{\pi}{2}}$] \label{def_strongSolution}
Let $\Omega\subset\R^2$ be a bounded domain with $C^3$-boundary, $T>0$ and $\alpha\in(0,\frac{\pi}{2}]$. We call $\mathscr{A}=\bigcup_{t\in [0,T]} \mathscr{A}(t)\times\{t\}$ a \textup{strong solution to mean curvature flow with constant contact angle $\alpha$} if the following conditions are satisfied:
\begin{enumerate}[leftmargin=0.7cm]
	\item[1.] \emph{Evolving regular partition in $\Omega$:} For all $t\in[0,T]$ the set $\mathscr{A}(t)\subset\Omega$ is open and connected with finite perimeter in $\R^2$ such that $\overline{\partial^\ast\mathscr{A}(t)}=\partial\mathscr{A}(t)$. The interface $I(t):=\overline{\partial^\ast\mathscr{A}(t)\cap\Omega}$ is a compact, connected, one-dimensional embedded $C^5$-manifold with boundary such that its interior~$I(t)^\circ$ lies in~$\Omega$ and its boundary~$\partial I(t)$ consists of exactly two distinct points which are located on the boundary of the domain, i.e., $\partial I(t)\subset\partial\Omega$.
	
	Moreover, there are diffeomorphisms $\Phi(\cdot,t)\colon\R^2\rightarrow\R^2$, $t\in[0,T]$, with
	$\Phi(\cdot,0)=\mathrm{Id}$ as well as
	\[
	\Phi(\mathscr{A}(0),t) = \mathscr{A}(t), \quad \Phi(I(0),t)= I(t)
	\quad\text{and}\quad \Phi(\partial I(0),t) = \partial I(t)
	\]
	for all $t\in[0,T]$, such that $\Phi\colon I(0)\times[0,T]\rightarrow I:=\bigcup_{t\in [0,T]} I(t)\times\{t\}$ is a diffeomorphism of class $C_t^0C_x^5\cap C_t^1C_x^3$.
	\item[2.] \emph{Mean curvature flow:} the interface~$I$ evolves by~MCF in the classical sense.
	\item[3.] \emph{Contact angle condition:} Let $\mathrm{n}_I(\cdot,t)$ denote the inner unit normal of $I(t)$ with respect to $\mathscr{A}(t)$ and let $\mathrm{n}_{\partial\Omega}$ be the inner unit normal of $\partial\Omega$ with respect to $\Omega$. Let $p_0 \in \partial I(0)$ be a boundary point, and let $p(t):=\Phi(p_0,t) \in \partial I(t)$. Then 
	\begin{align}\label{eq_angle}
	\mathrm{n}_I|_{(p(t),t)} \cdot \mathrm{n}_{\partial\Omega}|_{p(t)} = \cos\alpha
	\end{align}
	for all $t\in[0,T]$ encodes the contact angle condition. 
\end{enumerate}
\end{definition}

We emphasize that the required regularity of a strong solution implies
necessary (higher-order) compatibility conditions at the contact points
for the initial data. For the purposes of this work, we only rely on the one which
one obtains from differentiating in time the contact angle
condition~\eqref{eq_angle} and sending~$t \searrow 0$. To formulate it, let~$J$ denote 
the constant counter-clockwise rotation by $90^\circ$, and define the tangent vector fields 
$\tau_{\partial\Omega}:=J^\mathsf{T}\mathrm{n}_{\partial\Omega}$ as well as
$\tau_I(\cdot,0):=J^\mathsf{T}\mathrm{n}_{I}(\cdot,0)$. Denoting by~$H^{\partial\Omega}$ and~$H^{I}(\cdot,0)$
the scalar mean curvature of~$\partial\Omega$ and~$I(0)$ oriented with respect
to~$\mathrm{n}_{\partial\Omega}$ and~$\mathrm{n}_I(\cdot,0)$, respectively, we then have
as a necessary condition for the initial data the identity
(for a derivation, see Remark~\ref{th_strongsol_comp_cond})
\begin{align*}
&- H^I|_{(p_0,0)} H^{\partial\Omega}|_{p_0} 
+ (H^I)^2 \tau_I|_{(p_0,0)}\cdot \tau_{\partial\Omega}|_{p_0}  
- \mathrm{n}_{\partial\Omega}|_{p_0} \cdot
((\tau_I \cdot \nabla) H^I) \tau_I|_{(p_0,0)}=0
\end{align*}
for each of the two contact points~$p_0 \in \partial I(0)$. 

\subsection{Structure of the paper}
The remaining parts of the paper are structured as follows.
In Section~\ref{sec_stabilityEstimateRelEnergy},
we define the relative energy functional, cf.\ \eqref{eq_relEnergy}, encoding a
distance measure between solutions of~\eqref{eq_AllenCahn}--\eqref{eq_initialData}
and a calibrated evolution, discuss its coercivity properties,
and finally derive the associated stability estimate
from Theorem~\ref{theo_convergenceRatesCalibratedPartition}.
We then proceed in Section~\ref{sec_quantStabilityCalibrated}
to derive, based on the stability estimate for the relative energy,
a stability estimate in terms of a phase field version of a 
Luckhaus--Sturzenhecker type error functional, cf.\ \eqref{eq_bulkError}, 
which in turn controls the square of the $L^1$-error appearing on the
left hand side of the main quantitative convergence
estimate~\eqref{eq_optimalConvergenceRates}. Section~\ref{sec_constructionCalibration}
is devoted to the construction of a boundary adapted gradient flow
calibration with respect to a sufficiently regular evolution by
mean curvature flow with constant contact angle, thus providing
a proof of Theorem~\ref{theo_existenceBoundaryAdaptedCalibration}.
We conclude with two appendices, Appendix~\ref{sec_appA} and Appendix~\ref{sec_appB},
providing the proofs for the auxiliary results on weak solutions
of~\eqref{eq_AllenCahn}--\eqref{eq_initialData} as stated in Subsection~\ref{subsec_weakSolAllenCahn} 
and the existence of well-prepared initial data, Lemma~\ref{lem_wellPreparedInitialData}. 

\section{The stability estimate for the relative energy}
\label{sec_stabilityEstimateRelEnergy}
The aim of this section is to derive the 
first stability estimate from Theorem~\ref{theo_convergenceRatesCalibratedPartition},
which is phrased in terms of a suitable relative energy.
With respect to the definition and the coercivity
properties of the relative energy functional, we follow 
closely~\cite[Subsection~2.2 and Subsection~2.3]{Fischer2020}.

\subsection{Definition of the relative energy}
Let $\mathscr{A}=\smash{\bigcup_{t\in [0,T]}\mathscr{A}(t){\times}\{t\}}$
be a calibrated evolution in~$\Omega\subset\Rd$ with associated boundary adapted
gradient flow calibration~$(\xi,B,\vartheta)$ in the sense of Definition~\ref{def_boundaryAdaptedCalibration}.
Let~$ u_\eps$ be a solution to the Allen--Cahn problem~\eqref{eq_AllenCahn}--\eqref{eq_initialData}
in the sense of Definition~\ref{def_weakSolutionAllenCahn}
with finite energy initial data satisfying~$ u_{\eps,0}\in [-1,1]$. 
To be precise, we assume that the boundary energy density~$\sigma$
satisfies both~\eqref{eq_propBoundaryEnergyDensity} and~\eqref{eq_propBoundaryEnergyDensity2}.
Recalling~\eqref{eq_auxIndicatorPhaseField}, define
\begin{align}
\label{eq_indicatorPhaseField}
\psi_\eps(x,t) &:= \psi\big( u_\eps(x,t)\big)=\int_{-1}^{ u_\varepsilon(x,t)} \sqrt{2W(s)} \ds,
&& (x,t) \in \overline{\Omega}{\times} [0,T],
\end{align} 
and by fixing an arbitrary unit vector $\vec{s}\in\mathbb{S}^{d-1}$
\begin{align}
\label{eq_phaseFieldNormal}
\mathrm{n}_\eps :=
\begin{cases}
\frac{\nabla u_\eps}{|\nabla u_\eps|} &\text{if } \nabla u_\eps \neq 0,
\\
\vec{s} &\text{else.}
\end{cases}
\end{align}
Due to the regularity properties of the weak solution $ u_\varepsilon$ from Definition \ref{def_weakSolutionAllenCahn} as well as Lemma \ref{lem_existenceWeakSolutionAllenCahn} and Lemma \ref{lem_higherRegularityWeakSolutionAllenCahn}, it holds $\psi_\varepsilon\in C([0,T],H^1(\Omega))$, $\mathrm{n}_\varepsilon\in L^\infty((0,T)\times \Omega)$ and together with the features of $\xi, B$ from Definition \ref{def_boundaryAdaptedCalibration} the following computations are rigorous. First, note that the definitions~\eqref{eq_indicatorPhaseField} and~\eqref{eq_phaseFieldNormal} 
imply the relations
\begin{align}
\label{eq_normalAuxIdentity}
\nabla\psi_\varepsilon=\sqrt{2W( u_\varepsilon)}\nabla u_\varepsilon, \quad
\mathrm{n}_\eps |\nabla u_\eps| = \nabla u_\eps,
\quad
\mathrm{n}_\eps |\nabla\psi_\eps| = \nabla\psi_\eps.
\end{align}
Given this data, we define a relative energy as follows
\begin{align}
\label{eq_relEnergy}
E_{\mathrm{relEn}}[ u_\eps|\mathscr{A}](t) 
&:= \int_{\Omega} \frac{\eps}{2}\big|\nabla u_\eps(\cdot,t)\big|^2
+ \frac{1}{\eps}W\big( u_\eps(\cdot,t)\big) 
- \nabla\psi_\eps(\cdot,t)\cdot\xi(\cdot,t) \dx
\\&~~~\nonumber
+ \int_{\partial \Omega} \sigma\big( u_\eps(\cdot,t)\big)
- \psi\big( u_\eps(\cdot,t)\big)\cos\alpha \dH,
\quad t\in [0,T].
\end{align}
\subsection{Coercivity properties of the relative energy}
Using $\nabla\psi_\varepsilon=\sqrt{2W( u_\varepsilon)}\nabla u_\varepsilon$ and completing the square yields the alternative representation
\begin{equation}
\label{eq_alternativeRepRelEnergy}
\begin{aligned}
E_{\mathrm{relEn}}[ u_\eps|\mathscr{A}](t) &= \int_{\Omega} \frac{1}{2}\bigg(\sqrt{\eps}
|\nabla u_\eps(\cdot,t)| - \frac{\sqrt{2W( u_\eps(\cdot,t))}}{\sqrt{\eps}}\bigg)^2 \dx
\\&~~~
+ \int_{\Omega} (1 - \mathrm{n}_\eps\cdot\xi)(\cdot,t) |\nabla\psi_\eps(\cdot,t)| \dx
\\&~~~
+ \int_{\partial \Omega} \sigma\big( u_\eps(\cdot,t)\big)
- \psi\big( u_\eps(\cdot,t)\big)\cos\alpha \dH,
\quad t\in [0,T].
\end{aligned}
\end{equation}
It follows immediately from the representation~\eqref{eq_alternativeRepRelEnergy}
and the first item of~\eqref{eq_propBoundaryEnergyDensity2} that for all $t\in [0,T]$
\begin{align}
\label{eq_coercivityProp1}
0 \leq \int_{\Omega} \frac{1}{2}\bigg(\sqrt{\eps} |\nabla u_\eps(\cdot,t)| 
- \frac{\sqrt{2W( u_\eps(\cdot,t))}}{\sqrt{\eps}}\bigg)^2 \dx 
&\leq E_{\mathrm{relEn}}[ u_\eps|\mathscr{A}](t),
\\
\label{eq_coercivityProp2}
0 \leq \int_{\Omega} (1 - \mathrm{n}_\eps\cdot\xi)(\cdot,t) |\nabla\psi_\eps(\cdot,t)| \dx
&\leq E_{\mathrm{relEn}}[ u_\eps|\mathscr{A}](t),
\\
\label{eq_coercivityProp0}
0 \leq \int_{\partial \Omega} \sigma\big( u_\eps(\cdot,t)\big)
- \psi\big( u_\eps(\cdot,t)\big)\cos\alpha \dH
&\leq E_{\mathrm{relEn}}[ u_\eps|\mathscr{A}](t).
\end{align}
Moreover, it is a consequence of the length constraint~\eqref{eq_quadraticLengthConstraint}
and the coercivity property~\eqref{eq_coercivityProp2} that
\begin{align}
\label{eq_coercivityProp3}
\int_{\Omega} \min\big\{1,\dist^2\big(\cdot,\overline{\partial^*\mathscr{A}(t)\cap \Omega}\big)\big\} 
|\nabla\psi_\eps(\cdot,t)| \dx &\leq \frac{1}{c}E_{\mathrm{relEn}}[ u_\eps|\mathscr{A}](t), \quad t\in [0,T],
\\\label{eq_coercivityProp4}
\int_{\Omega} |(\mathrm{n}_\eps {-} \xi)(\cdot,t)|^2
|\nabla\psi_\eps(\cdot,t)| \dx &\leq 2E_{\mathrm{relEn}}[ u_\eps|\mathscr{A}](t), \quad t\in [0,T].
\end{align}

Finally, as it turns out in the sequel, we have to control analogous terms with 
the diffuse surface measure $\varepsilon|\nabla u_\varepsilon|^2$ instead of $|\nabla\psi_\varepsilon|$. Therefore adding zero as well as using that $\nabla\psi_\eps=\sqrt{2W( u_\eps)}\nabla u_\eps$
in a first step, and then applying Young's inequality together with~\eqref{eq_quadraticLengthConstraint}
in form of~$|\xi|\leq 1$ in a second step yields an auxiliary estimate for all~$t\in[0,T]$
\begin{align*}
&\int_{\Omega} |(\mathrm{n}_\eps {-} \xi)(\cdot,t)|^2 \eps|\nabla u_\eps(\cdot,t)|^2 \dx
\\&
= \int_{\Omega} |(\mathrm{n}_\eps {-} \xi)(\cdot,t)|^2 \sqrt{\eps}|\nabla u_\eps(\cdot,t)|
\bigg(\sqrt{\eps} |\nabla u_\eps(\cdot,t)| - \frac{\sqrt{2W( u_\eps(\cdot,t))}}{\sqrt{\eps}}\bigg) \dx 
\\&~~~
+ \int_{\Omega} |(\mathrm{n}_\eps {-} \xi)(\cdot,t)|^2 |\nabla\psi_\eps(\cdot,t)| \dx
\\&
\leq \frac{1}{2} \int_{\Omega} |(\mathrm{n}_\eps {-} \xi)(\cdot,t)|^2 \eps|\nabla u_\eps(\cdot,t)|^2 \dx
+ 2\int_{\Omega} \bigg(\sqrt{\eps} |\nabla u_\eps(\cdot,t)| 
- \frac{\sqrt{2W( u_\eps(\cdot,t))}}{\sqrt{\eps}}\bigg)^2 \dx 
\\&~~~
+ \int_{\Omega} |(\mathrm{n}_\eps {-} \xi)(\cdot,t)|^2 |\nabla\psi_\eps(\cdot,t)| \dx.
\end{align*}
Hence, absorbing the first right hand side term of this inequality
into the corresponding left hand side, and recalling the coercivity 
properties~\eqref{eq_coercivityProp1} and~\eqref{eq_coercivityProp4}, respectively, 
we thus obtain the bound
\begin{align}
\label{eq_coercivityProp5}
\int_{\Omega} |(\mathrm{n}_\eps {-} \xi)(\cdot,t)|^2 \eps|\nabla u_\eps(\cdot,t)|^2 \dx
&\leq 12 E_{\mathrm{relEn}}[ u_\eps|\mathscr{A}](t), \quad t\in [0,T].
\end{align}
Along similar lines using also~\eqref{eq_coercivityProp3}, one establishes that
for all $t\in [0,T]$
\begin{align}
\label{eq_coercivityProp6}
\int_{\Omega} \min\big\{1,\dist^2\big(\cdot,\overline{\partial^*\mathscr{A}(t)\cap \Omega}\big)\big\}  
\eps|\nabla u_\eps(\cdot,t)|^2 \dx
&\leq (1{+}2c^{-1})E_{\mathrm{relEn}}[ u_\eps|\mathscr{A}](t).
\end{align}

\subsection{Time evolution of the relative energy}
We proceed with the derivation of the stability estimate
for the relative energy from Theorem~\ref{theo_convergenceRatesCalibratedPartition}.
The basis is given by the following relative energy inequality.

\begin{lemma}
\label{lem_relativeEnergyInequality}
In the setting of Theorem~\ref{theo_convergenceRatesCalibratedPartition},
the following estimate on the time evolution of
the relative energy~$E_{\mathrm{relEn}}[ u_\eps|\mathscr{A}]$ 
defined by~\eqref{eq_relEnergy} holds true:
\begin{align}
\nonumber
&E_{\mathrm{relEn}}[ u_\eps|\mathscr{A}](T')
\\& \nonumber
+ \int_0^{T'}\int_{\Omega} \frac{1}{2\eps} \Big(H_\eps {+}(\nabla\cdot\xi)\sqrt{2W( u_\eps)}\Big)^2
{+} \frac{1}{2\eps} \Big(H_\eps {-}(B\cdot\xi)\eps|\nabla u_\eps|\Big)^2 \dx\dt 
\\& \label{eq_timeEvolutionRelEnergy}
\leq E_{\mathrm{relEn}}[ u_\eps|\mathscr{A}](0) 
\\&~~~\nonumber
+ \int_0^{T'}\int_{\Omega} \frac{1}{\sqrt{\eps}} \Big(H_\eps {+} (\nabla\cdot\xi) \sqrt{2W( u_\eps)}\Big)
 \big(B\cdot(\mathrm{n}_\eps{-}\xi)\big) \,\sqrt{\eps}|\nabla u_\eps| \dx \dt
\\&~~~\nonumber
+ \int_{0}^{T'} \int_{\partial \Omega} \big(\sigma( u_\eps) - \psi_\eps\cos\alpha\big)
(\mathrm{Id} {-} \mathrm{n}_{\partial \Omega} \otimes \mathrm{n}_{\partial \Omega}) : \nabla B \dH \dt
\\&~~~\nonumber
+ \int_0^{T'}\int_{\Omega} 2 \big|(B\cdot\xi){+}(\nabla\cdot\xi)\big|^2\,\eps|\nabla u_\eps|^2 \dx \dt
\\&~~~\nonumber
+ \int_0^{T'}\int_{\Omega} 2 |\nabla\cdot\xi|^2 \bigg(\sqrt{\eps}|\nabla u_\eps|
{-}\frac{\sqrt{2W( u_\eps)}}{\sqrt{\eps}}\bigg)^2 \dx \dt
\\&~~~\nonumber
- \int_0^{T'}\int_{\Omega} (\mathrm{n}_\eps {-} \xi) \cdot
\big(\partial_t\xi {+} (B\cdot\nabla)\xi {+} (\nabla B)^\mathsf{T}\xi\big) \,|\nabla\psi_\eps| \dx \dt
\\&~~~\nonumber
- \int_0^{T'}\int_{\Omega} \xi \cdot \big(\partial_t\xi {+} (B\cdot\nabla)\xi\big) \,|\nabla\psi_\eps| \dx \dt
\\&~~~\nonumber
- \int_0^{T'}\int_{\Omega} (\mathrm{n}_\eps {-} \xi) \otimes (\mathrm{n}_\eps {-} \xi) : \nabla B  \,|\nabla\psi_\eps| \dx \dt
\\&~~~\nonumber
+ \int_0^{T'}\int_{\Omega} (\nabla\cdot B) (1 - \mathrm{n}_\eps\cdot\xi ) \,|\nabla\psi_\eps| \dx \dt
\\&~~~\nonumber
+ \int_0^{T'}\int_{\Omega} (\nabla\cdot B)\frac{1}{2}\bigg(\sqrt{\eps}
|\nabla u_\eps| - \frac{\sqrt{2W( u_\eps)}}{\sqrt{\eps}}\bigg)^2  \dx \dt
\\&~~~\nonumber
- \int_0^{T'}\int_{\Omega} (\mathrm{n}_\eps\otimes\mathrm{n}_\eps {-} \xi\otimes\xi) : \nabla B 
\,\big(\eps|\nabla u_\eps|^2 {-} |\nabla\psi_\eps|\big) \dx \dt
\\&~~~\nonumber
- \int_0^{T'}\int_{\Omega} \xi\otimes\xi : \nabla B 
\,\big(\eps|\nabla u_\eps|^2 {-} |\nabla\psi_\eps|\big) \dx \dt.
\end{align}
for all $T'\in (0,T]$.
\end{lemma}

\begin{proof}
Fix $T'\in (0,T]$. Based on the definitions~\eqref{eq_energyPhaseField}
and~\eqref{eq_relEnergy} of the energy functional and the relative
energy, respectively, and the boundary condition~\eqref{eq_boundaryCondXi} for~$\xi$,
we may write
\begin{align*}
E_{\mathrm{relEn}}[ u_\eps|\mathscr{A}](T') &= E_\eps[ u_\eps(\cdot,T')]
-\int_{\Omega} \nabla\psi_\eps(\cdot,T')\cdot\xi(\cdot,T') \dx
\\&~~~
-\int_{\partial \Omega} \psi_\eps(\cdot,T')
\big(\mathrm{n}_{\partial \Omega}\cdot\xi(\cdot,T')\big) \dH.
\end{align*}
Hence, by means of the energy dissipation equality~\eqref{eq_energyDissipationWeakSolution}
(which can be equivalently expressed in form of~\eqref{eq_energyDissipationPhaseField}
thanks to~\eqref{eq_weakSolutionsAreStrongSolutions}),
the analogous representation of the relative energy at the initial time,
the fundamental theorem of calculus facilitated by a standard
mollification argument in the time variable, the 
definitions~\eqref{eq_auxIndicatorPhaseField} and~\eqref{eq_indicatorPhaseField}
together with an application of the chain rule, 
the boundary condition~\eqref{eq_boundaryCondXi} for~$\xi$,
as well as finally an integration by parts, 
we then obtain the estimate
\begin{align}
\nonumber
&E_{\mathrm{relEn}}[ u_\eps|\mathscr{A}](T') 
\\& \nonumber
= E_{\mathrm{relEn}}[ u_\eps|\mathscr{A}](0)
- \int_0^{T'}\int_{\Omega} \frac{1}{\eps} H_\eps^2 \dx\dt
\\&~~~\nonumber
- \bigg(\int_{\Omega} \nabla\psi_\eps(\cdot,T')\cdot\xi(\cdot,T') \dx
- \int_{\Omega} \nabla\psi_\eps(\cdot,0)\cdot\xi(\cdot,0) \dx \bigg)
\\&~~~\nonumber
- \bigg(\int_{\partial \Omega} \psi_\eps(\cdot,T')
\big(\mathrm{n}_{\partial \Omega}\cdot\xi(\cdot,T')\big) \dH
- \int_{\partial \Omega} \psi_\eps(\cdot,0)
\big(\mathrm{n}_{\partial \Omega}\cdot\xi(\cdot,0)\big) \dH\bigg)
\\& \label{eq_relEnergyInequalityStep1}
= E_{\mathrm{relEn}}[ u_\eps|\mathscr{A}](0) - \int_0^{T'}\int_{\Omega} \frac{1}{\eps} H_\eps^2 \dx\dt
\\&~~~\nonumber
+  \int_0^{T'}\int_{\Omega} (\nabla\cdot\xi) \sqrt{2W( u_\eps)} \partial_t u_\eps \dx\dt
\\&~~~\nonumber
- \int_0^{T'}\int_{\Omega} \mathrm{n}_\eps \cdot\partial_t\xi \,|\nabla\psi_\eps| \dx \dt.
\end{align}
Adding zero twice implies
\begin{align*}
- \int_0^{T'}\int_{\Omega} \mathrm{n}_\eps \cdot\partial_t\xi \,|\nabla\psi_\eps| \dx \dt
&= - \int_0^{T'}\int_{\Omega} \mathrm{n}_\eps \cdot
\big(\partial_t\xi {+} (B\cdot\nabla)\xi {+} (\nabla B)^\mathsf{T}\xi\big) \,|\nabla\psi_\eps| \dx \dt
\\&~~~
+ \int_0^{T'}\int_{\Omega} \xi\otimes\mathrm{n}_\eps : \nabla B \,|\nabla\psi_\eps| \dx \dt
\\&~~~
+ \int_0^{T'}\int_{\Omega} \mathrm{n}_\eps\cdot (B\cdot\nabla)\xi \,|\nabla\psi_\eps| \dx \dt
\\&
= - \int_0^{T'}\int_{\Omega} (\mathrm{n}_\eps {-} \xi) \cdot
\big(\partial_t\xi {+} (B\cdot\nabla)\xi {+} (\nabla B)^\mathsf{T}\xi\big) \,|\nabla\psi_\eps| \dx \dt
\\&~~~
- \int_0^{T'}\int_{\Omega} \xi \cdot \big(\partial_t\xi {+} (B\cdot\nabla)\xi\big) \,|\nabla\psi_\eps| \dx \dt
\\&~~~
+ \int_0^{T'}\int_{\Omega} \xi\otimes(\mathrm{n}_\eps {-} \xi) : \nabla B \,|\nabla\psi_\eps| \dx \dt	
\\&~~~
+ \int_0^{T'}\int_{\Omega} \mathrm{n}_\eps\cdot (B\cdot\nabla)\xi \,|\nabla\psi_\eps| \dx \dt.
\end{align*}
Moreover, we may compute by means of $\mathrm{n}_\eps |\nabla\psi_\eps| = \nabla\psi_\eps$,
the product rule, and adding zero twice
\begin{align*}
\int_0^{T'}\int_{\Omega} \mathrm{n}_\eps\cdot (B\cdot\nabla)\xi \,|\nabla\psi_\eps| \dx \dt 
&= \int_0^{T'}\int_{\Omega} \nabla\psi_\eps \cdot (B\cdot\nabla)\xi \dx \dt
\\&
= \int_0^{T'}\int_{\Omega} \nabla\psi_\eps \cdot \big(\nabla\cdot (\xi\otimes B)\big) \dx \dt
\\&~~~
- \int_0^{T'}\int_{\Omega} (\mathrm{n}_\eps\cdot\xi - 1)(\nabla\cdot B) \,|\nabla\psi_\eps| \dx \dt
\\&~~~
- \int_0^{T'}\int_{\Omega} (\mathrm{Id}{-}\mathrm{n}_\eps\otimes\mathrm{n}_\eps) : \nabla B \,|\nabla\psi_\eps| \dx \dt
\\&~~~
- \int_0^{T'}\int_{\Omega} \mathrm{n}_\eps\otimes\mathrm{n}_\eps : \nabla B \,|\nabla\psi_\eps| \dx \dt.
\end{align*}
By an integration by parts based on the 
regularity~\eqref{eq_regularityXi}--\eqref{eq_regularityB}
of~$(\xi,B)$, an application of the product rule, 
the symmetry relation $\nabla\cdot \big(\nabla\cdot(\xi\otimes B)\big)=\nabla\cdot \big(\nabla\cdot(B\otimes\xi)\big)$,
and an application of the boundary condition~\eqref{eq_boundaryCondVelocity} for the velocity field~$B$,
we also get
\begin{align*}
\int_0^{T'}\int_{\Omega} \nabla\psi_\eps \cdot \big(\nabla\cdot (\xi\otimes B)\big) \dx \dt
&= - \int_0^{T'}\int_{\Omega} \psi_\eps \nabla\cdot \big(\nabla\cdot(\xi\otimes B)\big) \dx \dt
\\&~~~
- \int_0^{T'}\int_{\partial \Omega} \psi_\eps \mathrm{n}_{\partial \Omega}
\cdot \big(\nabla\cdot(\xi\otimes B)\big) \dH \dt
\\&
= \int_0^{T'}\int_{\Omega} \nabla\psi_\eps \cdot 
\big(\nabla\cdot(B\otimes\xi)\big) \dx \dt
\\&~~~
+ \int_0^{T'}\int_{\partial \Omega} \psi_\eps \mathrm{n}_{\partial \Omega} \cdot (\xi\cdot\nabla)B \dH \dt
\\&~~~
- \int_0^{T'}\int_{\partial \Omega} \psi_\eps \mathrm{n}_{\partial \Omega}
\cdot \big((B\cdot\nabla)\xi {+} (\nabla\cdot B)\xi\big) \dH \dt
\\&
= \int_0^{T'}\int_{\Omega} (\nabla\cdot\xi) (B\cdot\mathrm{n}_\eps) \,|\nabla\psi_\eps| \dx \dt
\\&~~~
+ \int_0^{T'}\int_{\Omega} \mathrm{n}_\eps\otimes\xi : \nabla B \,|\nabla\psi_\eps| \dx \dt
\\&~~~
- \int_0^{T'}\int_{\partial \Omega} \psi_\eps (\mathrm{n}_{\partial \Omega}\cdot\xi) (\nabla\cdot B) \dH \dt
\\&~~~
- \int_0^{T'}\int_{\partial \Omega} \psi_\eps \mathrm{n}_{\partial \Omega} \cdot (B\cdot\nabla)\xi \dH \dt
\\&~~~
+ \int_0^{T'}\int_{\partial \Omega} \psi_\eps \mathrm{n}_{\partial \Omega} \cdot (\xi\cdot\nabla)B \dH \dt.
\end{align*}
Splitting the vector field~$\xi$ into tangential and normal components in the form of
$\xi=(\mathrm{n}_{\partial \Omega}\cdot\xi)\mathrm{n}_{\partial \Omega} + (\mathrm{Id}{-}\mathrm{n}_{\partial \Omega}\otimes\mathrm{n}_{\partial \Omega})\xi$,
and making use of $\nabla^\mathrm{tan}(B\cdot \mathrm{n}_{\partial \Omega})=0$ due to the boundary condition~\eqref{eq_boundaryCondVelocity} for the velocity field~$B$ as well as the product rule, we may further equivalently express
\begin{align*}
&\int_0^{T'}\int_{\partial \Omega} \psi_\eps \mathrm{n}_{\partial \Omega} \cdot (\xi\cdot\nabla)B \dH \dt
\\&
= \int_0^{T'}\int_{\partial \Omega} \psi_\eps (\mathrm{n}_{\partial \Omega} \cdot \xi) 
\mathrm{n}_{\partial \Omega} \cdot (\mathrm{n}_{\partial \Omega}\cdot\nabla)B \dH \dt
\\&~~~
- \int_0^{T'}\int_{\partial \Omega} \psi_\eps B\cdot\big(
(\mathrm{Id}{-}\mathrm{n}_{\partial \Omega}\otimes\mathrm{n}_{\partial \Omega})\xi
\cdot\nabla\big)\mathrm{n}_{\partial \Omega} \dH \dt.
\end{align*}
Exploiting the boundary condition~\eqref{eq_boundaryCondVelocity} for~$B$,
applying the product rule, splitting again the vector field~$\xi$ into tangential
and normal components as before, and finally relying on the classical facts that
$(\nabla^\mathrm{tan}\mathrm{n}_{\partial \Omega})^\mathsf{T}\mathrm{n}_{\partial \Omega}=0$
and $(\nabla^\mathrm{tan}\mathrm{n}_{\partial \Omega})^\mathsf{T}=\nabla^\mathrm{tan}\mathrm{n}_{\partial \Omega}$ 
along~$\partial \Omega$, we also have
\begin{align*}
&- \int_0^{T'}\int_{\partial \Omega} \psi_\eps \mathrm{n}_{\partial \Omega} \cdot (B\cdot\nabla)\xi \dH \dt
\\&
= \int_0^{T'}\int_{\partial \Omega} \psi_\eps \xi \cdot (B\cdot\nabla)\mathrm{n}_{\partial \Omega} \dH \dt
- \int_0^{T'}\int_{\partial \Omega} \psi_\eps (B\cdot\nabla)(\xi\cdot\mathrm{n}_{\partial \Omega}) \dH \dt
\\&
=\int_0^{T'}\int_{\partial \Omega} \psi_\eps
(\mathrm{Id}{-}\mathrm{n}_{\partial \Omega}\otimes\mathrm{n}_{\partial \Omega}) \xi 
\cdot (B\cdot\nabla)\mathrm{n}_{\partial \Omega} \dH \dt
\\&~~~
- \int_0^{T'}\int_{\partial \Omega} \psi_\eps (B\cdot\nabla)(\xi\cdot\mathrm{n}_{\partial \Omega}) \dH \dt
\\&
=\int_0^{T'}\int_{\partial \Omega} \psi_\eps B\cdot\big(
(\mathrm{Id}{-}\mathrm{n}_{\partial \Omega}\otimes\mathrm{n}_{\partial \Omega})\xi
\cdot\nabla\big)\mathrm{n}_{\partial \Omega} \dH \dt
\\&~~~
- \int_0^{T'}\int_{\partial \Omega} \psi_\eps (B\cdot\nabla)(\xi\cdot\mathrm{n}_{\partial \Omega}) \dH \dt.
\end{align*}
The previous three displays in total imply
\begin{align*}
\int_0^{T'}\int_{\Omega} \nabla\psi_\eps \cdot \big(\nabla\cdot (\xi\otimes B)\big) \dx \dt
&= \int_0^{T'}\int_{\Omega} (\nabla\cdot\xi) (B\cdot\mathrm{n}_\eps) \,|\nabla\psi_\eps| \dx \dt
\\&~~~
+ \int_0^{T'}\int_{\Omega} \mathrm{n}_\eps\otimes\xi : \nabla B \,|\nabla\psi_\eps| \dx \dt
\\&~~~
- \int_0^{T'}\int_{\partial \Omega} \psi_\eps (\mathrm{n}_{\partial \Omega}\cdot\xi) (\nabla^{\mathrm{tan}}\cdot B) \dH \dt
\\&~~~
- \int_0^{T'}\int_{\partial \Omega} \psi_\eps (B\cdot\nabla)(\xi\cdot\mathrm{n}_{\partial \Omega}) \dH \dt,
\end{align*}
so that the combination of the previous six displays culminates into
\begin{align}
\label{eq_relEnergyInequalityStep2}
- \int_0^{T'}\int_{\Omega} \mathrm{n}_\eps \cdot\partial_t\xi \,|\nabla\psi_\eps| \dx \dt
&= \int_0^{T'}\int_{\Omega} (\nabla\cdot\xi) (B\cdot\mathrm{n}_\eps) \,|\nabla\psi_\eps| \dx \dt
\\&~~~\nonumber
- \int_0^{T'}\int_{\Omega} (\mathrm{Id}{-}\mathrm{n}_\eps\otimes\mathrm{n}_\eps) : \nabla B \,|\nabla\psi_\eps| \dx \dt
\\&~~~\nonumber
- \int_0^{T'}\int_{\partial \Omega} \psi_\eps (\mathrm{n}_{\partial \Omega}\cdot\xi) (\nabla^{\mathrm{tan}}\cdot B) \dH \dt
\\&~~~\nonumber
- \int_0^{T'}\int_{\partial \Omega} \psi_\eps (B\cdot\nabla)(\xi\cdot\mathrm{n}_{\partial \Omega}) \dH \dt
\\&~~~\nonumber
- \int_0^{T'}\int_{\Omega} (\mathrm{n}_\eps {-} \xi) \cdot
\big(\partial_t\xi {+} (B\cdot\nabla)\xi {+} (\nabla B)^\mathsf{T}\xi\big) \,|\nabla\psi_\eps| \dx \dt
\\&~~~\nonumber
- \int_0^{T'}\int_{\Omega} \xi \cdot \big(\partial_t\xi {+} (B\cdot\nabla)\xi\big) \,|\nabla\psi_\eps| \dx \dt
\\&~~~\nonumber
- \int_0^{T'}\int_{\Omega} (\mathrm{n}_\eps {-} \xi) \otimes (\mathrm{n}_\eps {-} \xi) : \nabla B  \,|\nabla\psi_\eps| \dx \dt
\\&~~~\nonumber
- \int_0^{T'}\int_{\Omega} (\mathrm{n}_\eps\cdot\xi - 1)(\nabla\cdot B) \,|\nabla\psi_\eps| \dx \dt.
\end{align}
Inserting~\eqref{eq_relEnergyInequalityStep2} back into~\eqref{eq_relEnergyInequalityStep1}
and inspecting the structure of the right hand side of the desired estimate~\eqref{eq_timeEvolutionRelEnergy},
we still have to post-process the terms $\mathrm{Res}:=\mathrm{Res}^{(1)} + \mathrm{Res}^{(2)}$,
where
\begin{align}
\label{eq_relEnergyInequalityStep3}
\mathrm{Res}^{(1)} &:= - \int_0^{T'}\int_{\Omega} \frac{1}{\eps} H_\eps^2 \dx\dt
+  \int_0^{T'}\int_{\Omega} (\nabla\cdot\xi) \sqrt{2W( u_\eps)} \partial_t u_\eps \dx\dt
\\&~~~~\nonumber
+ \int_0^{T'}\int_{\Omega} (\nabla\cdot\xi) (B\cdot\mathrm{n}_\eps) \,|\nabla\psi_\eps| \dx \dt
\\&~~~~\nonumber
- \int_0^{T'}\int_{\Omega} (\mathrm{Id}{-}\mathrm{n}_\eps\otimes\mathrm{n}_\eps) : \nabla B \,|\nabla\psi_\eps| \dx \dt,
\\\label{eq_relEnergyInequalityStep4}
\mathrm{Res}^{(2)} &:= 
- \int_0^{T'}\int_{\partial \Omega} \psi_\eps (\mathrm{n}_{\partial \Omega}\cdot\xi) (\nabla^{\mathrm{tan}}\cdot B) \dH \dt
\\&~~~~\nonumber
- \int_0^{T'}\int_{\partial \Omega} \psi_\eps (B\cdot\nabla)(\xi\cdot\mathrm{n}_{\partial \Omega}) \dH \dt.
\end{align}
We start with the first residual term~$\mathrm{Res}^{(1)}$, and rewrite the last term for $\varepsilon|\nabla u_\varepsilon|^2$ instead of~$|\nabla\psi_\varepsilon|$ using the boundary condition~\eqref{eq_boundaryCondVelocity} for~$B$ and the definition \eqref{eq_curvaturePhaseField} for~$H_\varepsilon$. It turns out later that the difference can be controlled. Recalling $\mathrm{n}_\eps |\nabla u_\eps| = \nabla u_\eps$
and integrating by parts in the sense of the identity~\eqref{eq_weakFormRobinCondition}
based on the higher regularity of~$ u_\eps$ provided by
the first item of~\eqref{eq_higherRegularityPhaseField}
shows that
\begin{align}
\nonumber
\int_0^{T'}\int_{\Omega} \mathrm{n}_\eps \otimes \mathrm{n}_\eps : \nabla B \,\eps|\nabla u_\eps|^2 \dx \dt
&= \int_0^{T'}\int_{\Omega} \eps \nabla u_\eps \otimes \nabla u_\eps  : \nabla B  \dx \dt
\\&\label{eq_relEnergyInequalityStep5}
= - \int_0^{T'}\int_{\Omega} \eps \Delta u_\eps (B\cdot\mathrm{n}_\eps)\,|\nabla u_\eps|  \dx \dt
\\&~~~\nonumber
- \int_0^{T'}\int_{\Omega} \eps \nabla u_\eps \otimes B : \nabla^2 u_\eps  \dx \dt
\\&~~~\nonumber
- \int_0^{T'}\int_{\partial \Omega} \sigma'( u_\eps) (B\cdot\nabla) u_\eps  \dH \dt.
\end{align}
Moreover, another integration by parts in combination with
the boundary condition~\eqref{eq_boundaryCondVelocity} for the velocity field~$B$ entails
\begin{align*}
- \int_0^{T'}\int_{\Omega} \eps \nabla u_\eps \otimes B : \nabla^2 u_\eps \dx \dt
&= \int_0^{T'}\int_{\Omega} \eps \nabla u_\eps \otimes B : \nabla^2 u_\eps \dx \dt
\\&~~~
+ \int_0^{T'}\int_{\Omega} (\nabla\cdot B)\,\eps|\nabla u_\eps|^2 \dx \dt.
\end{align*}
In other words,
\begin{align*}
- \int_0^{T'}\int_{\Omega} \eps \nabla u_\eps \otimes B : \nabla^2 u_\eps \dx \dt
= \int_0^{T'}\int_{\Omega} (\nabla\cdot B)\,\frac{1}{2}\eps|\nabla u_\eps|^2 \dx \dt,
\end{align*}
so that completing the square, exploiting $\nabla\psi_\eps=\sqrt{2W( u_\eps)}\nabla u_\eps$,
integrating by parts (relying in the process on $\mathrm{n}_\eps |\nabla u_\eps| = \nabla u_\eps$
as well as yet again the boundary condition~\eqref{eq_boundaryCondVelocity} for the velocity field~$B$)
and finally recalling the definition~\eqref{eq_curvaturePhaseField} of the map~$H_\eps$ yields
\begin{align}
\nonumber
&- \int_0^{T'}\int_{\Omega} \eps \Delta u_\eps (B\cdot\mathrm{n}_\eps)\,|\nabla u_\eps|  \dx \dt
- \int_0^{T'}\int_{\Omega} \eps \nabla u_\eps \otimes B : \nabla^2 u_\eps \dx \dt
\\&\label{eq_relEnergyInequalityStep6}
= \int_0^{T'}\int_{\Omega} (\nabla\cdot B)\frac{1}{2}\bigg(\sqrt{\eps}
|\nabla u_\eps| - \frac{\sqrt{2W( u_\eps)}}{\sqrt{\eps}}\bigg)^2  \dx \dt
\\&~~~\nonumber
+ \int_0^{T'}\int_{\Omega} (\nabla\cdot B) \,|\nabla\psi_\eps| \dx \dt
+ \int_0^{T'}\int_{\Omega} H_\eps (B\cdot\mathrm{n}_\eps)\,|\nabla u_\eps|  \dx \dt.
\end{align}
Inserting back~\eqref{eq_relEnergyInequalityStep6} into~\eqref{eq_relEnergyInequalityStep5},
making use of the chain rule and the surface divergence theorem in form of
(relying in the process also on the higher regularity of~$ u_\eps$ provided by
the first item of~\eqref{eq_higherRegularityPhaseField} and the boundary condition~\eqref{eq_boundaryCondVelocity} for~$B$)
\begin{align*}
- \int_0^{T'}\int_{\partial \Omega} \sigma'( u_\eps) (B\cdot\nabla) u_\eps  \dH \dt
= \int_0^{T'}\int_{\partial \Omega} \sigma( u_\eps) (\nabla^{\mathrm{tan}} \cdot B) \dH \dt,
\end{align*}
and adding zero several times thus implies
\begin{align}
\label{eq_relEnergyInequalityStep7}
&- \int_0^{T'}\int_{\Omega} (\mathrm{Id}{-}\mathrm{n}_\eps\otimes\mathrm{n}_\eps) : \nabla B \,|\nabla\psi_\eps| \dx \dt
\\&\nonumber
= \int_0^{T'}\int_{\Omega} H_\eps (B\cdot\mathrm{n}_\eps)\,|\nabla u_\eps|  \dx \dt
- \int_0^{T'}\int_{\Omega} \mathrm{n}_\eps\otimes\mathrm{n}_\eps : \nabla B 
\,\big(\eps|\nabla u_\eps|^2 {-} |\nabla\psi_\eps|\big) \dx \dt
\\&~~~\nonumber
+ \int_0^{T'}\int_{\Omega} (\nabla\cdot B)\frac{1}{2}\bigg(\sqrt{\eps}
|\nabla u_\eps| - \frac{\sqrt{2W( u_\eps)}}{\sqrt{\eps}}\bigg)^2  \dx \dt
\\&~~~\nonumber
+ \int_0^{T'}\int_{\partial \Omega} \sigma( u_\eps) (\nabla^{\mathrm{tan}} \cdot B) \dH \dt.
\end{align}
Appealing to the boundary condition~\eqref{eq_boundaryCondXi} of the vector field~$\xi$ 
in form of
\begin{align*}
\mathrm{Res}^{(2)} = - \int_0^{T'}\int_{\partial \Omega} 
\psi_\eps \cos\alpha \, (\nabla^{\mathrm{tan}}\cdot B) \dH \dt
\end{align*}
we may thus infer from~\eqref{eq_relEnergyInequalityStep7} and
the definitions~\eqref{eq_relEnergyInequalityStep3}--\eqref{eq_relEnergyInequalityStep4} 
of the two residual terms
that it holds for $\mathrm{Res} = \mathrm{Res}^{(1)} + \mathrm{Res}^{(2)}$
\begin{align}
\label{eq_relEnergyInequalityStep8}
\mathrm{Res} &= - \int_0^{T'}\int_{\Omega} \frac{1}{\eps} H_\eps^2 \dx\dt
+  \int_0^{T'}\int_{\Omega} (\nabla\cdot\xi) \sqrt{2W( u_\eps)} \partial_t u_\eps \dx\dt
\\&~~~\nonumber
+ \int_0^{T'}\int_{\Omega} (\nabla\cdot\xi) (B\cdot\mathrm{n}_\eps) \,|\nabla\psi_\eps| \dx \dt
+ \int_0^{T'}\int_{\Omega} H_\eps (B\cdot\mathrm{n}_\eps)\,|\nabla u_\eps|  \dx \dt
\\&~~~\nonumber
+ \int_{0}^{T'} \int_{\partial \Omega} \big(\sigma( u_\eps) - \psi_\eps\cos\alpha\big)
(\mathrm{Id} {-} \mathrm{n}_{\partial \Omega} \otimes \mathrm{n}_{\partial \Omega}) : \nabla B \dH \dt
\\&~~~\nonumber
- \int_0^{T'}\int_{\Omega} (\mathrm{n}_\eps\otimes\mathrm{n}_\eps {-} \xi\otimes\xi) : \nabla B 
\,\big(\eps|\nabla u_\eps|^2 {-} |\nabla\psi_\eps|\big) \dx \dt
\\&~~~\nonumber
- \int_0^{T'}\int_{\Omega} \xi\otimes\xi : \nabla B 
\,\big(\eps|\nabla u_\eps|^2 {-} |\nabla\psi_\eps|\big) \dx \dt
\\&~~~\nonumber
+ \int_0^{T'}\int_{\Omega} (\nabla\cdot B)\frac{1}{2}\bigg(\sqrt{\eps}
|\nabla u_\eps| - \frac{\sqrt{2W( u_\eps)}}{\sqrt{\eps}}\bigg)^2  \dx \dt.
\end{align}
For the derivation of the desired relative energy inequality, it thus suffices
in view of~\eqref{eq_relEnergyInequalityStep1}, \eqref{eq_relEnergyInequalityStep2}
and~\eqref{eq_relEnergyInequalityStep8} to post-process the first four 
right hand side terms of~\eqref{eq_relEnergyInequalityStep8}. To this end,
one may argue as follows. First, based on the definition~\eqref{eq_curvaturePhaseField}
of the map~$H_\eps$, $\mathrm{n}_\eps |\nabla\psi_\eps| = \nabla\psi_\eps$,
the identity $\nabla\psi_\eps=\sqrt{2W( u_\eps)}\nabla u_\eps$,
and finally the Allen--Cahn equation~\eqref{eq_AllenCahn} expressed
in form of $\partial_t u_\eps = -\smash{\frac{1}{\eps}}H_\eps$
thanks to~\eqref{eq_weakSolutionsAreStrongSolutions} 
and~\eqref{eq_curvaturePhaseField}, we obtain by completing the square and adding zero
\begin{align}
\nonumber
&- \int_0^{T'}\int_{\Omega} \frac{1}{\eps} H_\eps^2 \dx\dt
+  \int_0^{T'}\int_{\Omega} (\nabla\cdot\xi) \sqrt{2W( u_\eps)} \partial_t u_\eps \dx\dt
\\&\nonumber
+ \int_0^{T'}\int_{\Omega} (\nabla\cdot\xi) (B\cdot\mathrm{n}_\eps) \,|\nabla\psi_\eps| \dx \dt
+ \int_0^{T'}\int_{\Omega} H_\eps (B\cdot\mathrm{n}_\eps)\,|\nabla u_\eps|  \dx \dt
\\&\label{eq_relEnergyInequalityStep9}
= - \int_0^{T'}\int_{\Omega} \frac{1}{2\eps} \Big(H_\eps {+}(\nabla\cdot\xi)\sqrt{2W( u_\eps)}\Big)^2 \dx\dt 
\\&~~~\nonumber
- \int_0^{T'}\int_{\Omega} \frac{1}{2\eps} H_\eps^2 \dx\dt
+ \int_0^{T'}\int_{\Omega} H_\eps (B\cdot\mathrm{n}_\eps)\,|\nabla u_\eps|  \dx \dt
\\&~~~\nonumber
+ \int_0^{T'}\int_{\Omega} \frac{1}{2}\bigg((\nabla\cdot\xi)\frac{\sqrt{2W( u_\eps)}}{\sqrt{\eps}}\bigg)^2 \dx \dt
\\&~~~\nonumber
+ \int_0^{T'}\int_{\Omega} (\nabla\cdot\xi) \frac{\sqrt{2W( u_\eps)}}{\sqrt{\eps}}
 (B\cdot\xi) \,\sqrt{\eps}|\nabla u_\eps| \dx \dt
\\&~~~\nonumber
+ \int_0^{T'}\int_{\Omega} \frac{1}{\sqrt{\eps}} (\nabla\cdot\xi) \sqrt{2W( u_\eps)}
 \big(B\cdot(\mathrm{n}_\eps{-}\xi)\big) \,\sqrt{\eps}|\nabla u_\eps| \dx \dt.
\end{align}
Completing the square yet again also entails
\begin{align}
\nonumber
&- \int_0^{T'}\int_{\Omega} \frac{1}{2\eps} H_\eps^2 \dx\dt
+ \int_0^{T'}\int_{\Omega} H_\eps (B\cdot\mathrm{n}_\eps)\,|\nabla u_\eps|  \dx \dt
\\&\label{eq_relEnergyInequalityStep10}
= - \int_0^{T'}\int_{\Omega} \frac{1}{2\eps} \big(H_\eps {-}(B\cdot\xi)\eps|\nabla u_\eps|\big)^2 \dx\dt 
\\&~~~\nonumber
+ \int_0^{T'}\int_{\Omega} \frac{1}{\sqrt{\eps}}H_\eps 
\big(B\cdot(\mathrm{n}_\eps{-}\xi)\big)\,\sqrt{\eps}|\nabla u_\eps|  \dx \dt
+ \int_0^{T'}\int_{\Omega} \frac{1}{2}(B\cdot\xi)^2\,\eps|\nabla u_\eps|^2  \dx \dt.
\end{align}
Finally, observe that it holds
\begin{align}
\nonumber
&\int_0^{T'}\int_{\Omega} \frac{1}{2}\bigg((\nabla\cdot\xi)\frac{\sqrt{2W( u_\eps)}}{\sqrt{\eps}}\bigg)^2 \dx \dt
+ \int_0^{T'}\int_{\Omega} \frac{1}{2}(B\cdot\xi)^2\,\eps|\nabla u_\eps|^2  \dx \dt
\\&\nonumber
+ \int_0^{T'}\int_{\Omega} (\nabla\cdot\xi) \frac{\sqrt{2W( u_\eps)}}{\sqrt{\eps}}
 (B\cdot\xi) \,\sqrt{\eps}|\nabla u_\eps| \dx \dt
\\&\nonumber
= \int_0^{T'}\int_{\Omega} \frac{1}{2} \bigg|(B\cdot\xi)\sqrt{\eps}|\nabla u_\eps|
+  (\nabla\cdot\xi)\frac{\sqrt{2W( u_\eps)}}{\sqrt{\eps}}\bigg|^2 \dx \dt
\\&\nonumber
= \int_0^{T'}\int_{\Omega} \frac{1}{2} \bigg|\big((B\cdot\xi){+}(\nabla\cdot\xi)\big)\sqrt{\eps}|\nabla u_\eps|
-  (\nabla\cdot\xi)\bigg(\sqrt{\eps}|\nabla u_\eps|
{-}\frac{\sqrt{2W( u_\eps)}}{\sqrt{\eps}}\bigg)\bigg|^2 \dx \dt
\\&\label{eq_relEnergyInequalityStep11}
\leq \int_0^{T'}\int_{\Omega} 2 \big|(B\cdot\xi){+}(\nabla\cdot\xi)\big|^2\,\eps|\nabla u_\eps|^2 \dx \dt
\\&~~~\nonumber
+ \int_0^{T'}\int_{\Omega} 2 |\nabla\cdot\xi|^2 \bigg(\sqrt{\eps}|\nabla u_\eps|
{-}\frac{\sqrt{2W( u_\eps)}}{\sqrt{\eps}}\bigg)^2 \dx \dt.
\end{align}
Hence, the combination of~\eqref{eq_relEnergyInequalityStep9}--\eqref{eq_relEnergyInequalityStep11} yields
\begin{align*}
&- \int_0^{T'}\int_{\Omega} \frac{1}{\eps} H_\eps^2 \dx\dt
+  \int_0^{T'}\int_{\Omega} (\nabla\cdot\xi) \sqrt{2W( u_\eps)} \partial_t u_\eps \dx\dt
\\&
+ \int_0^{T'}\int_{\Omega} (\nabla\cdot\xi) (B\cdot\mathrm{n}_\eps) \,|\nabla\psi_\eps| \dx \dt
+ \int_0^{T'}\int_{\Omega} H_\eps (B\cdot\mathrm{n}_\eps)\,|\nabla u_\eps|  \dx \dt
\\&
\leq - \int_0^{T'}\int_{\Omega} \frac{1}{2\eps} \Big(H_\eps {+}(\nabla\cdot\xi)\sqrt{2W( u_\eps)}\Big)^2 \dx\dt 
\\&~~~
- \int_0^{T'}\int_{\Omega} \frac{1}{2\eps} \big(H_\eps {-}(B\cdot\xi)\eps|\nabla u_\eps|\big)^2 \dx\dt 
\\&~~~
+ \int_0^{T'}\int_{\Omega} \frac{1}{\sqrt{\eps}} \Big(H_\eps {+} (\nabla\cdot\xi) \sqrt{2W( u_\eps)}\Big)
 \big(B\cdot(\mathrm{n}_\eps{-}\xi)\big) \,\sqrt{\eps}|\nabla u_\eps| \dx \dt
\\&~~~
+ \int_0^{T'}\int_{\Omega} 2 \big|(B\cdot\xi){+}(\nabla\cdot\xi)\big|^2\,\eps|\nabla u_\eps|^2 \dx \dt
\\&~~~
+ \int_0^{T'}\int_{\Omega} 2 |\nabla\cdot\xi|^2 \bigg(\sqrt{\eps}|\nabla u_\eps|
{-}\frac{\sqrt{2W( u_\eps)}}{\sqrt{\eps}}\bigg)^2 \dx \dt.
\end{align*}
This in turn concludes the proof.
\end{proof}

A post-processing of the relative energy inequality
based on the coercivity properties of the relative
energy functional now yields the asserted stability estimate.

\begin{corollary}
In the setting of Theorem~\ref{theo_convergenceRatesCalibratedPartition},
there exist two constants $c\in (0,1)$ and $C\in (1,\infty)$ such that
\begin{align}
\nonumber
&E_{\mathrm{relEn}}[ u_\eps|\mathscr{A}](T') 
\\& \nonumber
+ \int_0^{T'}\int_{\Omega} \frac{c}{2\eps} \Big(H_\eps {+}(\nabla\cdot\xi)\sqrt{2W( u_\eps)}\Big)^2
{+} \frac{1}{2\eps} \Big(H_\eps {-} (B\cdot\xi)\eps|\nabla u_\eps|\Big)^2 \dx\dt 
\\& \label{eq_stabilityEstimateRelEnergy}
\leq E_{\mathrm{relEn}}[ u_\eps|\mathscr{A}](0) 
+ C \int_{0}^{T'} E_{\mathrm{relEn}}[ u_\eps|\mathscr{A}](t) \dt
\end{align}
for all $T'\in (0,T]$.
\end{corollary}

\begin{proof}
Note that by~\eqref{eq_indicatorPhaseField}, \eqref{eq_auxIndicatorPhaseField},
and the chain rule
\begin{align}
\label{eq_auxEquipartitionEnergy}
\eps|\nabla u_\eps|^2 - |\nabla\psi_\eps|
= \sqrt{\eps}|\nabla u_\eps|\bigg(\sqrt{\eps}|\nabla u_\eps|
- \frac{\sqrt{2W( u_\eps)}}{\sqrt{\eps}}\bigg).
\end{align}
Hence, the right hand side terms of~\eqref{eq_timeEvolutionRelEnergy} can all be estimated
in terms of the relative energy itself (or by absorption into the first quadratic term
on the left hand side of~\eqref{eq_timeEvolutionRelEnergy}) based on straightforward arguments exploiting
the coercivity properties~\eqref{eq_coercivityProp1}--\eqref{eq_coercivityProp6} 
of the relative energy and the properties \eqref{eq_calibration1}--\eqref{eq_calibration4} of 
the vector fields $(\xi, B)$.
\end{proof}

\section{Quantitative stability with respect to a calibrated evolution}
\label{sec_quantStabilityCalibrated}
The main goal of this section is to conclude the proof of Theorem~\ref{theo_convergenceRatesCalibratedPartition}.
To this end, we first define an error functional which gives a direct
control for the $L^1$-distance between the evolving indicator function 
associated with a calibrated evolution and the solution of~\eqref{eq_AllenCahn}--\eqref{eq_initialData}
in terms of~\eqref{eq_indicatorPhaseField}.

\subsection{Definition and coercivity properties of the bulk error functional}
Let $\mathscr{A}=\smash{\bigcup_{t\in [0,T]}\mathscr{A}(t){\times}\{t\}}$
be a calibrated evolution in~$\Omega\subset\Rd$ with associated boundary adapted
gradient flow calibration~$(\xi,B,\vartheta)$ in the sense of Definition~\ref{def_boundaryAdaptedCalibration}.
Denote by $\chi(\cdot,t)$ the characteristic function associated to~$\mathscr{A}(t)$, $t\in [0,T]$.
Let~$ u_\eps$ be a weak solution to the Allen--Cahn problem~\eqref{eq_AllenCahn}--\eqref{eq_initialData}
in the sense of Definition~\ref{def_weakSolutionAllenCahn}
with finite energy initial data satisfying~$ u_{\eps,0}\in [-1,1]$.

Recalling the definitions~\eqref{eq_auxIndicatorPhaseField}, \eqref{def_surfaceTension} and \eqref{eq_indicatorPhaseField} of $\psi$, $ c_0$ and $\psi_\varepsilon$, we then define a bulk error functional by means of
\begin{align}
\label{eq_bulkError}
E_{\mathrm{bulk}}[ u_{\eps}|\mathscr{A}](t)
:= \int_{\Omega} \big(\psi_{\eps}(\cdot,t) -  c_0\chi(\cdot,t)\big)
\vartheta(\cdot,t) \dx,
\quad t\in [0,T].
\end{align}
Note that thanks to~\eqref{eq_comparisonPrincipleWeakSolution} (in particular $\psi_\varepsilon\in[0, c_0]$)
and the fact that $\vartheta(\cdot,t) < 0$ (resp.\ $\vartheta(\cdot,t) > 0$) in the essential interior
of~$\mathscr{A}(t)$ within~$\Omega$ (resp.\ the essential exterior of~$\mathscr{A}(t)$),
definition~\eqref{eq_bulkError} indeed provides a non-negative quantity for all $t\in [0,T]$:
\begin{align}
\label{eq_nonNegativityBulkError}
E_{\mathrm{bulk}}[ u_{\eps}|\mathscr{A}](t)
= \int_{\Omega} |\psi_{\eps}(\cdot,t) -  c_0\chi(\cdot,t)\big|
\big|\vartheta(\cdot,t)\big| \dx \geq 0.
\end{align}
Under additional regularity assumptions on~$\mathscr{A}$,
one may further guarantee that the bulk error 
functional~$E_{\mathrm{bulk}}[ u_{\eps}|\mathscr{A}](t)$
controls the squared $L^1$-distance between~$\psi_{\eps}(\cdot,t)$
and~$ c_0\chi(\cdot,t)$ for all $t\in [0,T]$.
For simplicity, let us state and prove this auxiliary result in 
terms of a strong solution.

\begin{lemma}[Coercivity of the bulk error functional]
\label{lem_coercivityBulkError}
In the setting of Theorem~\ref{theo_mainResult},
there exists a constant~$C>0$ such that for all~$t\in [0,T]$
it holds
\begin{align}
\label{eq_coercivityBulkError}
\|\psi_{\eps}(\cdot,t) -  c_0\chi(\cdot,t)\|^2_{L^1(\Omega)}
\leq C E_{\mathrm{bulk}}[ u_{\eps}|\mathscr{A}](t).
\end{align}
\end{lemma}

\begin{proof}
We divide the proof into two steps.

\textit{Step 1: A slicing argument.} Let~$M\subset\Rd$ be 
a an embedded, compact and oriented $(d{-}1)$-dimensional $C^2$-submanifold
of~$\Rd$ (potentially with boundary). Moreover, let~$\mathrm{n}_M$ denote a unit normal vector field along~$M$. Based on the tubular neighborhood theorem, fix a localization scale $r_M\in (0,1)$ and a constant~$C_M>0$
such that the map
\begin{align*}
\Psi_{M}\colon M \times (-r_M,r_M) \to \Rd,\quad
(x,s) \mapsto x + s\mathrm{n}_M(x)
\end{align*}
defines a $C^2$-diffeomorphism onto its image and such that $\textup{sdist}_M=(\Psi^{-1}_M)_2$ on $\Psi_M(M \times (-r_M,r_M))$ as well as
\begin{align*}
|\nabla\Psi_M| \leq C_M, \quad |\nabla\Psi^{-1}_M| \leq C_M.
\end{align*}
For any measurable~$g\colon\R^d\rightarrow\R$ bounded by $1$ (or by a uniform constant) it holds
\begin{align*}
\int_{M} \bigg|\int_{-r_M}^{r_M}
\big|g\big(x{+}s\mathrm{n}_M(x)\big)\big| \ds \bigg|^2 \dH
\lesssim  \int_{M} \int_{-r_M}^{r_M}
\big|g\big(x{+}s\mathrm{n}_M(x)\big)\big| \big|s\big| \ds \dH .
\end{align*}
This estimate can be shown by splitting the inner integral at $0$, using the Fubini Theorem and by dividing $(0,r_M)^2$ into two triangles, cf.\ Fischer, Laux and Simon~\cite{Fischer2020}, proof of Theorem 1 for a similar argument. By changing variables back and forth by means of~$\Psi_M$ and~$\Psi^{-1}_M$,
respectively, implies the estimate
\begin{align}
\label{eq_auxCoercivityTubNbhd}
\bigg|\int_{\Psi_{M}(M{\times}(-r_M,r_M))} |g| \dx \bigg|^2
\lesssim \int_{\Psi_{M}(M{\times}(-r_M,r_M))} |g| \dist(\cdot,M) \dx.
\end{align}

\textit{Step 2: Proof of~\eqref{eq_coercivityBulkError}.}
We claim that in the setting of Theorem~\ref{theo_mainResult}, 
for any measurable~$\|g\|_{L^\infty(\Rd)}\leq 1$ it holds
\begin{align}
\label{eq_coercivityBulkErrorAux}
\bigg|\int_{\Omega} |g| \dx \bigg|^2 \lesssim 
\int_{\Omega} |g| |\vartheta|(\cdot,t) \dx
\end{align}
uniformly over all~$t\in [0,T]$, which in turn of course implies the claim.

For a proof of~\eqref{eq_coercivityBulkErrorAux}, fix~$t\in [0,T]$ and then define a 
scale $r:= \min\{r_{\partial\Omega},r_{\overline{\partial^*\mathscr{A}(t)\cap \Omega}}\}$,
a map $d_{\mathrm{min}}:=\min\{\dist(\cdot,\partial \Omega),\dist(\cdot,\overline{\partial^*\mathscr{A}(t)\cap \Omega})\}$,
as well as sets 
\begin{align*}
\Omega_{\mathrm{bulk}} &:= \big\{x\in \Omega\colon d_{\mathrm{min}}(x) \geq r \big\} 
\\
\Omega_{\partial\Omega} &:= \big\{x\in \Omega\setminus \Omega_{\mathrm{bulk}}\colon 
d_{\mathrm{min}}(x)=\dist(x,\partial\Omega)\big\}
\\
\Omega_{\overline{\partial^*\mathscr{A}(t)\cap \Omega}} &:= \big\{x\in \Omega\setminus \Omega_{\mathrm{bulk}}\colon 
d_{\mathrm{min}}(x)=\dist(x,\overline{\partial^*\mathscr{A}(t)\cap \Omega})\big\}.
\end{align*}
Then, it holds by a union bound
\begin{align*}
\bigg|\int_{\Omega} |g| \dx \bigg|^2 \lesssim 
\bigg|\int_{\Omega_{\mathrm{bulk}}} |g| \dx \bigg|^2
+ \bigg|\int_{\Omega_{\partial\Omega}} |g| \dx \bigg|^2
+ \bigg|\int_{\Omega_{\overline{\partial^*\mathscr{A}(t)\cap \Omega}}} |g| \dx \bigg|^2.
\end{align*}
Due to the definition of the set~$\Omega_{\mathrm{bulk}}$
and the lower bound~\eqref{eq_weightCoercivity} for the weight~$\vartheta$, 
the first right hand side term of the previous display obviously
admits an estimate of required form. 
For an estimate of the second term,
note that by the definition of the set~$\Omega_{\partial\Omega}$
and the choice of the scale~$r$, it holds $\Omega_{\partial\Omega} \subset \Omega \cap
\Psi_{\partial \Omega}(\partial{\Omega}{\times}(-r_{\partial \Omega},r_{\partial \Omega}))$.
In particular, one may apply the estimate~\eqref{eq_auxCoercivityTubNbhd}
and then post-process it to required form based on 
the definition of the set~$\Omega_{\partial\Omega}$ and
the lower bound~\eqref{eq_weightCoercivity} for the weight~$\vartheta$.
The argument for the third term is essentially analogous, at least once
one carefully noted that $\Omega_{\overline{\partial^*\mathscr{A}(t)\cap \Omega}}$
is contained in the union of
$\Omega \cap \Psi_{\partial \Omega}(\partial{\Omega}{\times}(-r,r))$
and $\Omega \cap \Psi_{\overline{\partial^*\mathscr{A}(t)\cap \Omega}}(\overline{\partial^*\mathscr{A}(t)\cap \Omega}{\times}(-r,r))$. 
This concludes the proof.
\end{proof}

\subsection{Time evolution of the bulk error functional}
In a next step, we derive a suitable representation
for the time evolution of the error functional~$E_{\mathrm{bulk}}[ u_\eps|\mathscr{A}]$.

\begin{lemma}
\label{lem_timeEvolBulkError}
In the setting of Theorem~\ref{theo_convergenceRatesCalibratedPartition},
the time evolution of the bulk error functional~$E_{\mathrm{bulk}}[ u_\eps|\mathscr{A}]$
defined by~\eqref{eq_bulkError} can be represented by
\begin{align}
\nonumber
E_{\mathrm{bulk}}[ u_\eps|\mathscr{A}](T')
&= E_{\mathrm{bulk}}[ u_\eps|\mathscr{A}](0)
+ \int_{0}^{T'} \int_{\Omega} \vartheta (B\cdot\xi) \, 
\big(|\nabla\psi_\eps| - \eps|\nabla u_\eps|^2\big) \dx \dt
\\&~~~ \label{eq_timeEvolBulkError}
+ \int_{0}^{T'} \int_{\Omega} \vartheta\sqrt{\eps}|\nabla u_\eps|
\Big((B\cdot\xi)\sqrt{\eps}|\nabla u_\eps|-\frac{H_\eps}{\sqrt{\eps}}\Big) \dx \dt
\\&~~~ \nonumber
+ \int_{0}^{T'} \int_{\Omega} \vartheta \bigg(\frac{H_\eps}{\sqrt{\eps}} 
{+} (\nabla\cdot\xi)\frac{\sqrt{2W( u_\eps)}}{\sqrt{\eps}}\bigg)
\bigg(\sqrt{\eps}|\nabla u_\eps| {-} \frac{\sqrt{2W( u_\eps)}}{\sqrt{\eps}}\bigg) \dx \dt
\\&~~~ \nonumber
+ \int_{0}^{T'} \int_{\Omega} (\nabla\cdot\xi)\vartheta
\bigg(\sqrt{\eps}|\nabla u_\eps|-\frac{\sqrt{2W( u_\eps)}}{\sqrt{\eps}}\bigg)^2 \dx \dt
\\&~~~ \nonumber
- \int_{0}^{T'} \int_{\Omega} (\nabla\cdot\xi)\vartheta\sqrt{\eps}|\nabla u_\eps|
\bigg(\sqrt{\eps}|\nabla u_\eps|-\frac{\sqrt{2W( u_\eps)}}{\sqrt{\eps}}\bigg) \dx \dt
\\&~~~ \nonumber
+ \int_{0}^{T'} \int_{\Omega} \vartheta \big(B\cdot(\mathrm{n}_\eps - \xi)\big) \, |\nabla\psi_\eps| \dx \dt
\\&~~~ \nonumber
+ \int_{0}^{T'} \int_{\Omega} (\psi_\eps -  c_0\chi) \vartheta (\nabla\cdot B) \dx \dt
\\&~~~ \nonumber
+ \int_{0}^{T'} \int_{\Omega} (\psi_\eps -  c_0\chi) 
\big(\partial_t\vartheta {+} (B\cdot\nabla)\vartheta\big) \dx \dt
\end{align}
for all~$T'\in [0,T]$.
\end{lemma}

\begin{proof}
By an application of the fundamental theorem of calculus together with
a standard mollification argument in the time variable, an application
of the chain rule, as well as by exploiting that the measure $\partial_t\chi$ 
is absolutely continuous with respect to the measure~$|\nabla\chi|$ restricted 
to the set~$\smash{\bigcup_{t\in (0,T)}(\partial^*\mathscr{A}(t) \cap \Omega){\times}\{t\}}$, on which
in turn the weight~$\vartheta$ vanishes due to~\eqref{eq_weightZeroInterface}, it holds
\begin{align*}
E_{\mathrm{bulk}}[ u_\eps|\mathscr{A}](T')
&= E_{\mathrm{bulk}}[ u_\eps|\mathscr{A}](0)
+ \int_{0}^{T'} \int_{\Omega} \vartheta \sqrt{2W( u_\eps)} \partial_t u_\eps \dx \dt
\\&~~~
+ \int_{0}^{T'} \int_{\Omega} (\psi_\eps -  c_0\chi) \partial_t\vartheta \dx \dt.
\end{align*}
Adding zero twice, making use of the chain rule, and integrating by parts 
(exploiting in the process the boundary condition~\eqref{eq_boundaryCondVelocity}
for~$B$ and again the condition~\eqref{eq_weightZeroInterface} for~$\vartheta$) yields 
the following update of the previous display
\begin{align*}
E_{\mathrm{bulk}}[ u_\eps|\mathscr{A}](T')
&= E_{\mathrm{bulk}}[ u_\eps|\mathscr{A}](0)
+ \int_{0}^{T'} \int_{\Omega} \vartheta \sqrt{2W( u_\eps)} \partial_t u_\eps \dx \dt
\\&~~~
+ \int_{0}^{T'} \int_{\Omega} \vartheta (B\cdot\xi) \, |\nabla\psi_\eps| \dx \dt
\\&~~~
+ \int_{0}^{T'} \int_{\Omega} \vartheta \big(B\cdot(\mathrm{n}_\eps - \xi)\big) \, |\nabla\psi_\eps| \dx \dt
\\&~~~
+ \int_{0}^{T'} \int_{\Omega} (\psi_\eps -  c_0\chi) \vartheta (\nabla\cdot B) \dx \dt
\\&~~~
+ \int_{0}^{T'} \int_{\Omega} (\psi_\eps -  c_0\chi) 
\big(\partial_t\vartheta {+} (B\cdot\nabla)\vartheta\big) \dx \dt,
\end{align*}
for which we also recall $\mathrm{n}_\eps |\nabla\psi_\eps| = \nabla\psi_\eps$.
Moreover, inserting the Allen--Cahn equation~\eqref{eq_AllenCahn} 
in form of $\partial_t u_\eps = -\smash{\frac{1}{\eps}}H_\eps$
thanks to~\eqref{eq_weakSolutionsAreStrongSolutions} and~\eqref{eq_curvaturePhaseField}
entails together with adding zero twice that
\begin{align*}
&\int_{0}^{T'} \int_{\Omega} \vartheta \sqrt{2W( u_\eps)} \partial_t u_\eps \dx \dt
+ \int_{0}^{T'} \int_{\Omega} \vartheta (B\cdot\xi) \, |\nabla\psi_\eps| \dx \dt
\\&
= \int_{0}^{T'} \int_{\Omega} \vartheta (B\cdot\xi) \, 
\big(|\nabla\psi_\eps| - \eps|\nabla u_\eps|^2\big) \dx \dt
\\&~~~
+ \int_{0}^{T'} \int_{\Omega} \vartheta\sqrt{\eps}|\nabla u_\eps|
\Big((B\cdot\xi)\sqrt{\eps}|\nabla u_\eps|-\frac{H_\eps}{\sqrt{\eps}}\Big) \dx \dt
\\&~~~
+ \int_{0}^{T'} \int_{\Omega} \vartheta \frac{H_\eps}{\sqrt{\eps}}
\bigg(\sqrt{\eps}|\nabla u_\eps|-\frac{\sqrt{2W( u_\eps)}}{\sqrt{\eps}}\bigg) \dx \dt.
\end{align*}
Continuing in this fashion by adding appropriate zeros moreover gives
\begin{align*}
&\int_{0}^{T'} \int_{\Omega} \vartheta \frac{H_\eps}{\sqrt{\eps}}
\bigg(\sqrt{\eps}|\nabla u_\eps|-\frac{\sqrt{2W( u_\eps)}}{\sqrt{\eps}}\bigg) \dx \dt
\\&
= \int_{0}^{T'} \int_{\Omega} \vartheta \bigg(\frac{H_\eps}{\sqrt{\eps}} 
+ (\nabla\cdot\xi)\frac{\sqrt{2W( u_\eps)}}{\sqrt{\eps}}\bigg)
\bigg(\sqrt{\eps}|\nabla u_\eps|-\frac{\sqrt{2W( u_\eps)}}{\sqrt{\eps}}\bigg) \dx \dt
\\&~~~
+ \int_{0}^{T'} \int_{\Omega} (\nabla\cdot\xi)\vartheta
\bigg(\sqrt{\eps}|\nabla u_\eps|-\frac{\sqrt{2W( u_\eps)}}{\sqrt{\eps}}\bigg)^2 \dx \dt
\\&~~~
- \int_{0}^{T'} \int_{\Omega} (\nabla\cdot\xi)\vartheta\sqrt{\eps}|\nabla u_\eps|
\bigg(\sqrt{\eps}|\nabla u_\eps|-\frac{\sqrt{2W( u_\eps)}}{\sqrt{\eps}}\bigg) \dx \dt.
\end{align*}
The collection of the previous four displays now entails the claim.
\end{proof}

We have everything in place to proceed with the proof of the first main result
of this work concerning quantitative stability for the 
Allen--Cahn problem~\eqref{eq_AllenCahn}--\eqref{eq_initialData} 
with respect to a calibrated evolution.

\subsection{Proof of Theorem~\ref{theo_convergenceRatesCalibratedPartition}}
\label{subsec_proofConvergenceRatesCalibrated}
Recalling the identity~\eqref{eq_auxEquipartitionEnergy}, the 
coercivity properties~\eqref{eq_coercivityProp1}--\eqref{eq_coercivityProp6},
the estimate~\eqref{eq_stabilityEstimateRelEnergy} for the time evolution
of the relative energy functional, the representation~\eqref{eq_timeEvolBulkError}
of the time evolution of the bulk error functional, 
as well as the properties~\eqref{eq_weightZeroInterface}-\eqref{eq_weightEvol}
of the weight~$\vartheta$ (here \eqref{eq_weightZeroInterface} implies the estimate $|\vartheta(\cdot,t)| \leq 
C\|\nabla\vartheta(\cdot,t)\|_{L^\infty(\Omega)}\min\{1,\dist(\cdot,\overline{\partial^*\mathscr{A}(t)\cap \Omega})\})$
for all $t\in [0,T]$), we obtain by straightforward arguments that there exists
two constants $c\in (0,1)$ and $C>0$ such that
\begin{align*}
&E_{\mathrm{bulk}}[ u_\eps|\mathscr{A}](T') 
\\& 
+ \int_0^{T'}\int_{\Omega} \frac{c}{2\eps} \Big(H_\eps {+}(\nabla\cdot\xi)\sqrt{2W( u_\eps)}\Big)^2
{+} \frac{c}{2\eps} \Big(H_\eps {-} (B\cdot\xi)\eps|\nabla u_\eps|\Big)^2 \dx\dt 
\\& 
\leq (E_{\mathrm{relEn}}+E_{\mathrm{bulk}})[ u_\eps|\mathscr{A}](0) 
+ C \int_{0}^{T'} (E_{\mathrm{relEn}}+E_{\mathrm{bulk}})[ u_\eps|\mathscr{A}](t) \dt
\end{align*}
for all $T'\in [0,T]$. Together with~\eqref{eq_stabilityEstimateRelEnergy},
this implies the desired estimates.
\qed

\section{Construction of boundary adapted gradient flow calibrations}
\label{sec_constructionCalibration}

We follow the strategy of~\cite{Fischer2020a} by constructing local candidates for the vector fields $(\xi,B)$ around each topological feature, i.e., the contact points in Section~\ref{sec_calib_contact}, the bulk interface in Section \ref{sec_calib_interface}, and the domain boundary in Section~\ref{sec_calib_bdry}. These local constructions are then merged together into the global one in Section~\ref{sec_calib_global}. The construction of~$\vartheta$ is simpler and carried out in Section~\ref{sec_calib_theta}.

Let~$\mathscr{A}$ be a strong solution for mean curvature flow with contact angle~$\alpha$ on the time interval $[0,T]$ as in Definition~\ref{def_strongSolution}. In the following we summarize some notation and assertions concerning tubular neighbourhoods for~$I$ and~$\partial\Omega$ in Remarks~\ref{th_strongsol_notation_tub}
and~\ref{th_strongsol_notation_bdry}, respectively. Necessary compatibility conditions at the contact points
are collected in Remark~\ref{th_strongsol_comp_cond}.

\begin{remark}[Notation and tubular neighbourhoods for strong solutions of planar mean curvature flow with constant contact angle~$0<\alpha\leq\smash{\frac{\pi}{2}}$] \label{th_strongsol_notation_tub}	
For the following, we refer to~\cite[Definition~21 and Lemma~23]{Fischer2020a} 
and comments there. 
	
	In the situation of Definition \ref{def_strongSolution}, the assumptions imply the existence of a uniform localization scale $r_I\in(0,1]$ such that natural ball conditions at interior and boundary points are satisfied. Moreover, the standard tubular neighbourhood map $X_I:I\times(-r_I,r_I)\rightarrow\R^2\times[0,T]:(x,t,s)\mapsto (x+s \mathrm{n}_I(x,t), t)$ is well-defined, bijective onto its image $\mathrm{im}(X_I)$, 
and the inverse has the regularity $C_tC_x^4\cap C_t^1C_x^2$ on $\overline{\textup{im}(X_I)}$.
	
	We denote by $s^I$ the signed distance function with respect to the normal $\mathrm{n}_I$
	and let $P^I$ be the orthogonal projection. Then $s^I$ is of class $C_tC_x^5\cap C_t^1C_x^3$ on $\overline{\textup{im}(X_I)}$ and $P^I$ the same except one regularity less in space. We note that
	\[
	|s^I(\tilde{x},t)|
	= \dist_x(\cdot,I)(\tilde{x},t)
	:= \dist(\tilde{x},I(t))\quad\text{ for }(\tilde{x},t)\in\overline{\textup{im}(X_I)},
	\]
	where the latter is also defined globally on~$\R^2$ and we will sometimes 
	use the notation $\dist_x(\cdot,I)$ for convenience.
	
	Moreover, the following definitions yield extensions of the inner unit normal $\mathrm{n}_I$ and the (mean) curvature $H^I$ to the tubular neighbourhood: 
	\begin{align}
	\label{eq:normalExtensionNormalCurvature}
	\mathrm{n}_I:=\nabla s^I\quad\text{ and }\quad 
	H^I:=-\Delta s^I|_{(P^I,\textup{pr}_t)}\quad\text{on }\overline{\textup{im}(X_I)},
	\end{align}
	where $\textup{pr}_t$ is the projection onto the time component. Then $\mathrm{n}_I$ has the regularity $C_tC_x^4\cap C_t^1C_x^2$ on $\overline{\textup{im}(X_I)}$ and $H^I$ the same just one order less in space. Moreover,
	\[
	|\nabla s^I|^2=1, \quad \nabla s^I=\nabla s^I|_{(P^I,\textup{pr}_t)}\quad\text{ and }\quad 
	\partial_t s^I=\partial_t s^I|_{(P^I,\textup{pr}_t)}\quad\text{on }\overline{\textup{im}(X_I)}.
	\]
	Finally, let us define $\tau_I:=J^\mathsf{T} \mathrm{n}_I$ pointwise on $\overline{\textup{im}(X_I)}$, where $J$ is the constant rotation by $90$° counter-clockwise. Then by \cite[(128) and (129)]{Fischer2020a}, we have
	\begin{align}\label{eq_nabla_normale_und_tau}
	\nabla n^I=-H^I \tau_I\otimes \tau_I\quad \text{and}\quad \nabla\tau^I = 
	H^I \mathrm{n}_I \otimes \tau_I\quad\text{ on }I.
	\end{align}
	Note that we did not use 2.~and 3.~from Definition \ref{def_strongSolution} up to now. If 2.~holds, then 
	\begin{align}\label{eq_signdist}
	\partial_ts^I=\Delta s^I|_{(P^I,\textup{pr}_t)}=-H^I\quad\text{and}\quad \partial_t\mathrm{n}_I
	=-\nabla H^I\quad\text{ on }\overline{\textup{im}(X_I)}.
	\end{align}
\end{remark}

\begin{remark}[Notation for the boundary]\label{th_strongsol_notation_bdry}
Since the boundary $\partial\Omega$ of the domain is~$C^3$, we can use similar constructions and definitions as in the last Remark \ref{th_strongsol_notation_tub}, except equations \eqref{eq_signdist}. In particular, there is a suitable localization scale $r_{\partial\Omega}\in(0,1]$ and an associated (time-independent) tubular neighbourhood diffeomorphism $X_{\partial\Omega}$, so that  $s^{\partial\Omega}$ denotes the signed distance, $P^{\partial\Omega}$ the orthogonal projection and $\mathrm{n}_{\partial\Omega}$, $\tau_{\partial\Omega}$,
and~$H^{\partial\Omega}$ are defined in the analogous way as in Remark \ref{th_strongsol_notation_tub}. Concerning regularity $s^{\partial\Omega}$ is $C_x^3$, $P^{\partial\Omega}$, $\mathrm{n}_{\partial\Omega}$, $\tau_{\partial\Omega}$ are $C_x^2$ and $H^{\partial\Omega}$ is of class $C_x^1$.
\end{remark}

\begin{remark}[Compatibility conditions for strong solutions of planar mean curvature flow with constant contact angle~$0<\alpha\leq\smash{\frac{\pi}{2}}$]\label{th_strongsol_comp_cond}
	We remark that 1.--3.~in Definition \ref{def_strongSolution} imply compatibility conditions at the boundary points. The latter will be important for the local construction of the calibrations close to the boundary points. Let us fix a boundary point $p\in\partial I(0)$ for the initial interface and set $p(t):=\Phi(p,t)$ for $t\in[0,T]$. Then $p(t)\in\partial\Omega$ and mean curvature flow yield
	\begin{align}
	\label{eq_comp_cond1}
	\frac{\mathrm{d}}{\mathrm{d}t}p(t)\cdot \mathrm{n}_{\partial\Omega}|_{p(t)}=0\quad\text{ and }\quad 
	\frac{\mathrm{d}}{\mathrm{d}t}p(t)\cdot \mathrm{n}_I|_{(p(t),t)}=H^I|_{(p(t),t)}, \quad t\in[0,T].
	\end{align}
	In order to obtain a higher order compatibility condition, 
	we differentiate the angle condition~\eqref{eq_angle} with respect to time. 
	This yields together with~\eqref{eq_nabla_normale_und_tau}
	\begin{align*}
	0=&\left(-H^{\partial\Omega}\tau_{\partial\Omega}\otimes \tau_{\partial\Omega}|_{p(t)} 
	\frac{\mathrm{d}}{\mathrm{d}t}p(t)\right)\cdot \mathrm{n}_I|_{(p(t),t)}\\
	&+\mathrm{n}_{\partial\Omega}|_{p(t)}\cdot \left(-H^I  \tau_I\otimes\tau_I|_{(p(t),t)} 
	\frac{\mathrm{d}}{\mathrm{d}t}p(t) + \partial_t \mathrm{n}_I|_{(p(t),t)}\right)\quad\text{ for all }t\in[0,T].
	\end{align*}
	We insert the identities from~\eqref{eq_comp_cond1} for $\frac{\mathrm{d}}{\mathrm{d}t}p(t)$ and use the properties of the rotation $J$; the latter to rewrite $\mathrm{n}_{\partial\Omega}|_{p(t)}\cdot \tau_I|_{(p(t),t)} =-\tau_{\partial\Omega}|_{p(t)}\cdot \mathrm{n}_I|_{(p(t),t)}$. Therefore we obtain the next compatibility condition, which is third order concerning derivatives: for all $t\in[0,T]$ it holds
	\begin{align}\label{eq_comp_cond2}
	-H^I|_{(p(t),t)} H^{\partial\Omega}|_{p(t)} 
	+ (H^I)^2 \tau_I|_{(p(t),t)}\cdot \tau_{\partial\Omega}|_{p(t)}  
	- \mathrm{n}_{\partial\Omega}|_{p(t)} \cdot \nabla H^I|_{(p(t),t)}=0.
	\end{align}
\end{remark}

\subsection{Local building block for $(\xi, B)$ at contact points}\label{sec_calib_contact}

For the construction at the contact points we proceed in a similar way as in the case of a triple junction for multiphase mean curvature flow, see~\cite[Section~6]{Fischer2020a}. Therefore we introduce an appropriate localization radius~$r_p$ for the contact points, such that there are suitable evolving sectors confining the topological features on an evolving ball on this scale. This is done in Lemma~\ref{th_local_radius_contact} below. Then in Section~\ref{sec_calib_contact_local} we construct candidates for $(\xi,B)$ defined on tubular neighborhoods of the interface~$I$ and the boundary~$\partial\Omega$, respectively, which will serve as a definition on corresponding sectors. Here ideas from~\cite[Section~6.1]{Fischer2020a} are adjusted for the present situation. Finally, these constructions will be interpolated in Section~\ref{sec_calib_contact_intpol} analogously to~\cite[Section 6.2]{Fischer2020a}.

\begin{lemma}\label{th_local_radius_contact}
	Let $\mathscr{A}$ be a strong solution for mean curvature flow with contact angle $\alpha$ on the time interval $[0,T]$ as in Definition \ref{def_strongSolution}. Moreover, let $p\in\partial I(0)$, $p(t):=\Phi(p,t)$ for $t\in[0,T]$ and $\mathscr{P}:=\bigcup_{t\in[0,T]} \{p(t)\} \times \{t\}$ be the corresponding evolving contact point. Then there is a localization radius $r=r_p\in(0,\min\{r_I,r_{\partial\Omega}\}]$ such that the evolving ball $\mathscr{B}_r(p):=\bigcup_{t\in[0,T]} B_r(p(t))\times\{t\}$ has a wedge-decomposition in the following sense:
	\begin{enumerate}[leftmargin=0.6cm]
	\item[1.] $\mathscr{B}_r(p)$ is separated at each time $t\in[0,T]$ into open wedge-type domains $W_I(t)$, $W_\pm(t)$, $W_0(t)$ and an open double-wedge-type domain $W_{\partial\Omega}(t)$. The latter are disjoint and the union of the closures gives $\overline{B_r(p(t))}$. These domains are the intersections of $B_r(p(t))$ with cones defined from unit $C^1$-vector fields in time with constant-in-time angle relation (analogous 
	to~\cite[Definition~24]{Fischer2020a}). The corresponding space-time 
	domains are denoted by $W_I$, $W_\pm$, $W_0$ and $W_{\partial\Omega}$.
	\item[2.] Moreover, $\overline{W_\pm(t)}$, $\overline{W_{\partial\Omega}(t)}$, $\overline{W_0(t)}$
	are contained in the tubular neighborhood for~$\partial\Omega$, and $\overline{W_\pm(t)}$, $\overline{W_I(t)}$ are contained in the tubular neighborhood of~$I(t)$ for all $t\in[0,T]$. Additionally, for all $t\in[0,T]$ it holds
	\begin{align*}
	W_+(t)\subset\mathscr{A}(t), \quad W_-(t)\subset\Omega\setminus\overline{\mathscr{A}(t)},
	\quad W_0(t)\subset\R^2\setminus\overline{\Omega},
	\quad W_I(t) \subset \Omega,
	\end{align*}
	and finally $I(t)\cap B_r(p(t)) \subset W_I(t)\cup\{p(t)\}$ 
	and $\partial\Omega \cap B_r(p(t))
	\subset W_{\partial\Omega}(t)\cup\{p(t)\}$ for all $t \in [0,T]$.
	\item[3.] Finally, on each of the separating domains, there are uniform natural estimates comparing the distances to the different topological features (similar to~\cite[Definition 24]{Fischer2020a}).
	\end{enumerate}
\end{lemma}

We henceforth call $W_I$ interface wedge, $W_{\partial\Omega}$ boundary (double-)wedge, $W_\pm$ bulk (or interpolation) wedges and $W_0$ outer wedge, cf.\ Figure~\ref{fig_loc_sectors}.

\begin{figure}
	\def\svgwidth{0.6\linewidth}
	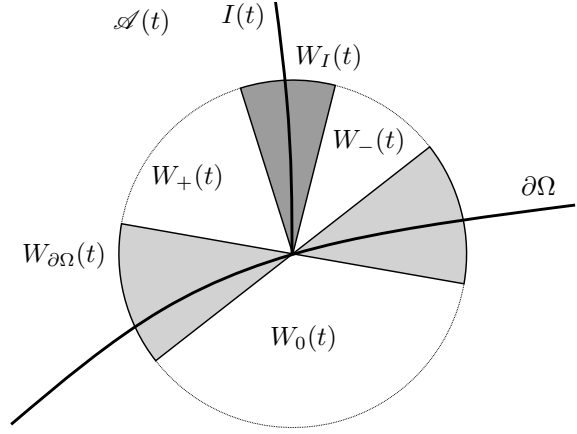
	\caption{Illustration of wedge decomposition at a contact point.}\label{fig_loc_sectors}
\end{figure}

\begin{proof}
The separating domains can be defined in a purely geometric way, 
and one may argue simply along the lines of the proof of~\cite[Lemma~25]{Fischer2020a}. 
Therefore, we refrain from going into details. 
\end{proof} 

We may now formulate the main result of this subsection.

\begin{theorem}
\label{th_localConstructionsContactPoint}
Let $\mathscr{A}$ be a strong solution for mean curvature flow 
with contact angle~$\alpha$ on the time interval $[0,T]$ as in 
Definition~\ref{def_strongSolution}, and let the notation
of Remark~\ref{th_strongsol_notation_tub} and Remark~\ref{th_strongsol_notation_bdry}
be in place. For each of the two contact points~$p_{\pm} \in \partial I(0)$
with associated trajectory~$p_{\pm}(t) \in \partial I(t)$, let $r_{\pm}=r_{p_{\pm}}$
be an associated localization radius in the sense of Lemma~\ref{th_local_radius_contact} above.
For a given $\hat r_{\pm} \in (0,r_{\pm}]$, we define
a space-time domain $\mathscr{B}_{\hat r_{\pm}}(p_{\pm}) := \bigcup_{t\in[0,T]}
B_{\hat r_{\pm}}(p_{\pm}(t)) {\times} \{t\}$.

There then exists a localization scale $\hat r_\pm \in (0,r_{p_{\pm}}]$
and a pair of local vector fields $\xi^{p_{\pm}},B^{p_{\pm}}\colon
\overline{\mathscr{B}_{\hat r_{\pm}}(p_{\pm})}
\cap (\overline{\Omega}{\times}[0,T]) \to \Rd[2]$ such that
the following conditions hold:
\begin{enumerate}[leftmargin=0.6cm]
\item[1.] \emph{(Regularity)} It holds
\begin{align}
\label{eq:qualRegularityXiContactPoint}
\xi^{p_{\pm}} &\in C^1(\overline{\mathscr{B}_{\hat r_{\pm}}(p_{\pm})} 
\cap (\overline{\Omega}{\times}[0,T])) \cap
C_tC^2_x(\mathscr{B}_{\hat r_{\pm}}(p_{\pm}) \cap (\Omega {\times} [0,T])),
\\ \label{eq:qualRegularityVelContactPoint}
B^{p_{\pm}} &\in 
C_tC^1_x(\overline{\mathscr{B}_{\hat r_{\pm}}(p_{\pm})} 
\cap (\overline{\Omega}{\times}[0,T])) \cap
C_tC^2_x(\mathscr{B}_{\hat r_{\pm}}(p_{\pm}) \cap (\Omega {\times} [0,T])),
\end{align}
and there exists $C>0$ such that
\begin{align}
\label{eq:quantRegularityContactPoint}
|\nabla^2\xi^{p_{\pm}}| + |\nabla^2 B^{p_{\pm}}| &\leq C
\quad \text{in } \mathscr{B}_{\hat r_{\pm}}(p_{\pm}) \cap (\Omega {\times} [0,T]).
\end{align}

\item[2.] \emph{(Consistency)} We have $|\xi^{p_{\pm}}| = 1$ in $\mathscr{B}_{\hat r_{\pm}}(p_{\pm})\cap(\Omega\times[0,T])$ as well as
\begin{align}
\label{eq:consistencyContactPoints}
\xi^{p_{\pm}} = \mathrm{n}_{I}
\text{ and } \big(\nabla\xi^{p_{\pm}}\big)^\mathsf{T}\mathrm{n}_{I} &= 0
&& \text{along } \mathscr{B}_{\hat r_{\pm}}(p_{\pm}) \cap I,
\\
\label{eq:consistencyVelocity} 
B^{p_{\pm}}(p_{\pm}(t),t) &= \frac{\mathrm{d}}{\mathrm{d}t}p_\pm(t)
&& \text{for all } t \in [0,T].
\end{align}

\item[3.] \emph{(Boundary conditions)} Moreover, it holds
\begin{align}
\label{eq:boundaryConditionsContactPoints}
\xi^{p_{\pm}} \cdot \mathrm{n}_{\partial\Omega} &= \cos\alpha
\text{ and } B^{p_{\pm}} \cdot \mathrm{n}_{\partial\Omega} = 0
&& \text{along } \mathscr{B}_{\hat r_{\pm}}(p_{\pm}) \cap (\partial\Omega {\times} [0,T]). 
\end{align}

\item[4.] \emph{(Motion laws)} In terms of evolution equations, there exists $C>0$ such that
\begin{align}
\label{eq:localEvolXiContactPoint}
|\partial_t \xi^{p_{\pm}} + (B^{p_{\pm}}\cdot\nabla)\xi^{p_{\pm}}
+ (\nabla B^{p_{\pm}})^\mathsf{T}\xi^{p_{\pm}}| 
&\leq C\dist_x(\cdot,I),
\\ \label{eq:localEvolLengthXiContactPoint}
|(\partial_t + B^{p_{\pm}}\cdot\nabla)|\xi^{p_{\pm}}|^2| 
&=0,
\\
\label{eq:localMCFContactPoint}
|B^{p_{\pm}}\cdot \xi^{p_{\pm}} + \nabla\cdot\xi^{p_{\pm}}|
&\leq C\dist_x(\cdot,I),
\\
\label{eq:localGradVelXiXiContactPoint}
|\xi^{p_{\pm}} \otimes \xi^{p_{\pm}} : \nabla B^{p_{\pm}}|
&\leq C\dist_x(\cdot,I)
\end{align}
throughout $\mathscr{B}_{\hat r_{\pm}}(p_{\pm}) \cap (\Omega\times[0,T])$.

\item[5.] \emph{(Additional constraints)} Finally, the construction of~$B^{p_{\pm}}$ 
may be arranged in a way to guarantee that
\begin{align}
\label{eq:gradVelSkewSym}
\nabla^{\mathrm{sym}} B^{p_{\pm}} = 0
\quad\text{along } \mathscr{B}_{\hat r_{\pm}}(p_{\pm}) \cap (I
\cup (\partial\Omega{\times}[0,T])).
\end{align}
\end{enumerate}
\end{theorem}

The proof of this result occupies the whole remainder of this subsection.

\subsubsection{Construction of local candidates for $(\xi, B)$ at contact points}\label{sec_calib_contact_local}
We fix the contact point $\mathscr{P}$ in this section and consider a localization radius $r=r_p$ as in Lemma~\ref{th_local_radius_contact}. Then there is a unique rotation $R_\alpha$ (rotation by $-\alpha$ or $\alpha$) such that 
\begin{align}
\label{eq:choiceRotation}
R_\alpha \mathrm{n}_{\partial\Omega}|_{p(t)}=\mathrm{n}_I|_{(p(t),t)}\quad\text{ and hence }\quad 
R_\alpha \tau_{\partial\Omega}|_{p(t)}=\tau_I|_{(p(t),t)}.
\end{align}

Motivated from \cite[Section 6.1]{Fischer2020a}, and the conditions in Definition \ref{def_boundaryAdaptedCalibration}, we consider the following candidate vector fields
\begin{alignat}{2}\label{eq_xi_I}
	\tilde{\xi}^I &:= \mathrm{n}_I + s^I \beta^I \tau_I - \frac{1}{2} (s^I\beta^I)^2 \mathrm{n}_I&\quad&\text{ on } \mathscr{B}_r(p)\cap\overline{\textup{im}(X_I)},\\
	\tilde{\xi}^{\partial\Omega} &:= R_\alpha \left[ \mathrm{n}_{\partial\Omega} {+} s^{\partial\Omega} \beta^{\partial\Omega}\tau_{\partial\Omega} {-} \frac{1}{2}(s^{\partial\Omega} \beta^{\partial\Omega})^2 \mathrm{n}_{\partial\Omega} \right] &\quad&\text{ on } 
	\mathscr{B}_r(p)\cap (\overline{\textup{im}(X_{\partial\Omega})}{\times}[0,T]),\label{eq_xi_dO}
\end{alignat}
where $\beta^I=\hat{\beta}^I(P^I,\textup{pr}_t)$ on $\overline{\textup{im}(X_I)}$ and $\beta^{\partial\Omega}=\hat{\beta}^{\partial\Omega}(P^{\partial\Omega},\textup{pr}_t)$ on 
$(\overline{\textup{im}(X_{\partial\Omega})}{\times}[0,T])$ with $\hat{\beta}^I\colon I\rightarrow\R$ and $\hat{\beta}^{\partial\Omega}\colon\partial\Omega\times[0,T]\rightarrow\R$. Note that the quadratic terms are just added for a length correction later. Moreover, we introduce
\begin{alignat}{2}\label{eq_B_I}
\tilde{B}^I &:= H^I\mathrm{n}_I + (\gamma^I+s^I \rho^I) \tau_I &\quad&\text{ on } \mathscr{B}_r(p)\cap\overline{\textup{im}(X_I)},\\
\tilde{B}^{\partial\Omega} &:= (\gamma^{\partial\Omega}+s^{\partial\Omega} \rho^{\partial\Omega}) \tau_{\partial\Omega} &\quad&\text{ on } 
\mathscr{B}_r(p)\cap(\overline{\textup{im}(X_{\partial\Omega})}{\times}[0,T]),\label{eq_B_dO}
\end{alignat}
where $\gamma^I$, $\rho^I$ are defined via projection from some $\hat{\gamma}^I$, $\hat{\rho}^I\colon I\rightarrow\R$ and $\gamma^{\partial\Omega}$, $\rho^{\partial\Omega}$ are defined via projection from some $\hat{\gamma}^{\partial\Omega}$, $\hat{\rho}^{\partial\Omega}\colon \partial\Omega\times[0,T]\rightarrow\R$ analogously as before.

Our task is to choose the ansatz functions $\hat{\beta}^I, \hat{\gamma}^I, \hat{\rho}^I$ and $\hat{\beta}^{\partial\Omega}, \hat{\gamma}^{\partial\Omega}, \hat{\rho}^{\partial\Omega}$ in such a way that the above vector fields $\tilde{\xi}^I$, $\tilde{\xi}^{\partial\Omega}$ and $\tilde{B}^I$, $\tilde{B}^{\partial\Omega}$ are compatible at $\mathscr{P}$ up to first order in space derivatives, respectively, and that the latter equal 
$\frac{\mathrm{d}}{\mathrm{d}t}p$ at $\mathscr{P}$. Moreover, the property \eqref{eq_consistencyProperty} should be satisfied at $\mathscr{P}$ for both $\tilde{\xi}^I$ and $\tilde{\xi}^{\partial\Omega}$, and the left hand side of the equations \eqref{eq_calibration1}--\eqref{eq_calibration4} should be zero exactly on $I\cap\mathscr{B}_r(p)$ for $\tilde{\xi}^I, \tilde{B}^I$ and be zero at $\mathscr{P}$ for $\tilde{\xi}^{\partial\Omega}$, $\tilde{B}^{\partial\Omega}$. Finally, the boundary conditions \eqref{eq_boundaryCondXi}--\eqref{eq_boundaryCondVelocity} should be satisfied for $\tilde{\xi}^I$, $B^I$ at $\mathscr{P}$ and for $\tilde{\xi}^{\partial\Omega}$, $\tilde{B}^{\partial\Omega}$ on $(\partial\Omega\times[0,T])\cap\mathscr{B}_r(p)$. See also Theorem~\ref{th_calib_contact} below for more precise statements. Note that the consistency for second order space derivatives in the regularity class from Definition~\ref{def_boundaryAdaptedCalibration} is not needed and will be taken care of via a suitable interpolation in Section~\ref{sec_calib_contact_intpol}.  

Therefore let us compute the required derivatives to first order for the above vector fields:

\begin{proposition}\label{th_calib_contact_abl}
	Let $\hat{\beta}^I$, $\hat{\beta}^{\partial\Omega}$ be of class $C^1$
	on their respective domains of definition, and let $\hat{\gamma}^I, \hat{\rho}^I, \hat{\gamma}^{\partial\Omega}, \hat{\rho}^{\partial\Omega}$ have the regularity $C_tC_x^1$ on their
	respective domains of definition. Then it holds
	\begin{alignat*}{2}
	\partial_t\tilde{\xi}^I|_I &= -\nabla H^I - \beta^IH^I \tau_I&\quad&\text{ on }\mathscr{B}_r(p)\cap I,\\
	\partial_t\tilde{\xi}^{\partial\Omega}|_{\partial\Omega{\times}[0,T]} &= 0
	&\quad&\text{ on }\mathscr{B}_r(p)\cap(\partial\Omega{\times}[0,T]),\\
	\nabla\tilde{\xi}^I|_I &= \tau_I \otimes [-H^I\tau_I+\beta^I \mathrm{n}_I]&\quad&\text{ on }\mathscr{B}_r(p)\cap I,\\
	\nabla\tilde{\xi}^{\partial\Omega}|_{\partial\Omega{\times}[0,T]}
	&= R_\alpha\tau_{\partial\Omega} \otimes [-H^{\partial\Omega}\tau_{\partial\Omega}+\beta^{\partial\Omega} \mathrm{n}_{\partial\Omega}]&\quad&\text{ on }\mathscr{B}_r(p)\cap(\partial\Omega{\times}[0,T]),\\
	\nabla\tilde{B}^I|_I &=(\tau_I\cdot\nabla H^I+\gamma^I H^I) \mathrm{n}_I\otimes \tau_I&&\\ 
	&\phantom{=}+ (\tau_I\cdot\nabla \gamma^I-(H^I)^2) \tau_I\otimes\tau_I
	\\&~~~
	+\rho^I \tau_I\otimes \mathrm{n}_I
	&\quad&\text{ on }\mathscr{B}_r(p)\cap I,
	\\
	\nabla\tilde{B}^{\partial\Omega}|_{\partial\Omega{\times}[0,T]}
	&=(\gamma^{\partial\Omega} H^{\partial\Omega}) \mathrm{n}_{\partial\Omega}\otimes \tau_{\partial\Omega}&&\\ 
	&\phantom{=}+ (\tau_{\partial\Omega}\cdot\nabla \gamma^{\partial\Omega}) \tau_{\partial\Omega}\otimes\tau_{\partial\Omega}
	\\&~~~+\rho^{\partial\Omega} \tau_{\partial\Omega}\otimes \mathrm{n}_{\partial\Omega}
	&\quad&\text{ on }\mathscr{B}_r(p)\cap (\partial\Omega{\times}[0,T]).
	\end{alignat*}
\end{proposition}
\begin{proof}
	This follows from a straightforward calculation using the 
	identities from Remark \ref{th_strongsol_notation_tub} and Remark \ref{th_strongsol_notation_bdry}
	and the definitions~\eqref{eq_xi_I}--\eqref{eq_B_dO}.
\end{proof}
Now we can insert the compatibility conditions 
and derive equations for the ansatz functions $\hat{\beta}^I, \hat{\gamma}^I, \hat{\rho}^I$ and $\hat{\beta}^{\partial\Omega}, \hat{\gamma}^{\partial\Omega}, \hat{\rho}^{\partial\Omega}$, respectively,
in order to satisfy the requirements mentioned just before Proposition~\ref{th_calib_contact_abl}.

First, we have by~\eqref{eq:choiceRotation}, \eqref{eq_xi_I} and~\eqref{eq_xi_dO}
\[
\tilde{\xi}^I|_{(p(t),t)} = \mathrm{n}_I|_{(p(t),t)} = R_\alpha \mathrm{n}_{\partial\Omega}|_{p(t)}=\tilde{\xi}^{\partial\Omega}|_{(p(t),t)}\quad\text{ for all }t\in[0,T].
\]
Moreover, note that it holds $R_\alpha \tau_{\partial\Omega}|_{p(t)}=\tau_I|_{(p(t),t)}$ for all $t\in[0,T]$. Therefore, due to Proposition \ref{th_calib_contact_abl} the compatibility of the gradient at the contact point, i.e., $\nabla\tilde{\xi}^I|_{(p(t),t)}=\nabla\tilde{\xi}^{\partial\Omega}|_{(p(t),t)}$ for $t\in[0,T]$, is equivalent to
\begin{align}\label{eq_beta_I}
\beta^I|_{(p(t),t)} &= 
-H^{\partial\Omega}|_{p(t)}\tau_{\partial\Omega}\cdot \mathrm{n}_I|_{(p(t),t)}+\beta^{\partial\Omega}|_{(p(t),t)} \mathrm{n}_{\partial\Omega}\cdot \mathrm{n}_I|_{(p(t),t)},\\
-H^I|_{(p(t),t)}&=-H^{\partial\Omega}|_{p(t)}\tau_{\partial\Omega}\cdot\tau_I|_{(p(t),t)}+\beta^{\partial\Omega}|_{(p(t),t)} \mathrm{n}_{\partial\Omega}\cdot\tau_I|_{(p(t),t)}\notag
\end{align}
for $t\in[0,T]$. Hence we obtain for $t\in[0,T]$ 
\begin{align}\label{eq_beta_dO}
\beta^{\partial\Omega}|_{(p(t),t)}&=\frac{1}{\mathrm{n}_{\partial\Omega}\cdot\tau_I|_{(p(t),t)}}\left(-H^I|_{(p(t),t)}+H^{\partial\Omega}|_{p(t)}\tau_{\partial\Omega}\cdot\tau_I|_{(p(t),t)}\right),
\end{align}
where $|\mathrm{n}_{\partial\Omega}\cdot\tau_I|_{(p(t),t)}|=\cos(\frac{\pi}{2}-\alpha)>0$. This determines also $\beta^I|_{(p(t),t)}$. Note that in the case $\alpha=\frac{\pi}{2}$ one simply gets 
$\beta^I|_{(p(t),t)}= -H^{\partial\Omega}|_{p(t)}$ and 
$\beta^{\partial\Omega}|_{(p(t),t)}= -H^I|_{p(t)}$, respectively.

Next, we consider the requirement
\[
\tilde{B}^I|_{(p(t),t)} = \frac{\mathrm{d}}{\mathrm{d}t}p(t) = \tilde{B}^{\partial\Omega}|_{(p(t),t)}\quad\text{ for }t\in[0,T].
\]
Because of~\eqref{eq_comp_cond1} for $\frac{\mathrm{d}}{\mathrm{d}t}p$ from Remark \ref{th_strongsol_comp_cond}, we simply obtain 
that for $t\in[0,T]$
\begin{align}\label{eq_gamma_I}
\gamma^I|_{(p(t),t)}&=\frac{\mathrm{d}}{\mathrm{d}t}p(t)\cdot\tau_{I}|_{(p(t),t)},\\
\gamma^{\partial\Omega}|_{(p(t),t)}&=
\frac{\mathrm{d}}{\mathrm{d}t}p(t)\cdot\tau_{\partial\Omega}|_{(p(t),t)}.\label{eq_gamma_dO}
\end{align}

Now let us consider $\nabla\tilde{B}^I$. Let us note that \eqref{eq_calibration1} is an approximate equation for the transport and rotation of the vector field $\xi$. This motivates us to require $\nabla\tilde{B}^I$ to be anti-symmetric on the interface $I$, since then the latter can be interpreted as an infinitesimal rotation. Hence in the formula for $\nabla\tilde{B}^I|_I$ in Proposition \ref{th_calib_contact_abl} the coefficient of $\tau_I\otimes\tau_I$ should vanish and the prefactors of $\mathrm{n}_I\otimes\tau_I$ and $\tau_I\otimes \mathrm{n}_I$ should be the negative of each other. This yields
\begin{alignat}{2}\label{eq_gamma1_I}
\tau_I\cdot\nabla \gamma^I&=(H^I)^2 &\quad&\text{ on }I,\\
\quad \rho^I&=-\tau_I\cdot\nabla H^I-\gamma^I H^I&\quad&\text{ on }I.\label{eq_rho_I}
\end{alignat}
Then the equation for $\nabla\tilde{B}^I|_I$ becomes
\begin{align}\label{eq_nablaB_I}
\nabla\tilde{B}^I=\rho^I (\tau_I\otimes \mathrm{n}_I-\mathrm{n}_I\otimes\tau_I)=\rho^I J\quad\text{ on }I,
\end{align}
with the counter-clockwise rotation $J$ by $90°$. Due to the same reason, we require $\nabla\tilde{B}^{\partial\Omega}$ to be anti-symmetric on $\partial\Omega\times[0,T]$ which yields 
\begin{align}\label{eq_gamma1_dO}
\tau_{\partial\Omega}\cdot\nabla \gamma^{\partial\Omega}&=0&\quad&\text{ on }\partial\Omega\times[0,T],\\
\rho^{\partial\Omega}&=-\gamma^{\partial\Omega} H^{\partial\Omega} &\quad&\text{ on }\partial\Omega\times[0,T],\label{eq_rho_dO}
\end{align}
and thus 
\begin{align}
\label{eq:skewSymmetryGradientTildeVelBoundary}
\nabla\tilde{B}^{\partial\Omega} &= \rho^{\partial\Omega} J
&&\text{on } \partial\Omega{\times}[0,T].
\end{align}
Hence, the compatibility condition at first order 
$\nabla\tilde{B}^I|_{(p(t),t)}=\nabla\tilde{B}^{\partial\Omega}|_{(p(t),t)}$ is equivalent to
\begin{align}\label{eq_rho_consistent}
\rho^{\partial\Omega}|_{(p(t),t)}=\rho^I|_{(p(t),t)}\quad\text{ for }t\in[0,T].
\end{align}
Because of \eqref{eq_gamma1_I}--\eqref{eq_rho_I} and \eqref{eq_gamma1_dO}--\eqref{eq_rho_dO} the latter is the same as
\[
-\gamma^{\partial\Omega}|_{(p(t),t)}H^{\partial\Omega}|_{p(t)}=
-\tau_I\cdot\nabla H^I-\gamma^I H^I|_{(p(t),t)}\quad\text{ for }t\in[0,T].
\]
By inserting \eqref{eq_gamma_I}--\eqref{eq_gamma_dO} and using 
\[
(\tau_{\partial\Omega}|_{p(t)}\cdot \mathrm{n}_I|_{(p(t),t)})\tau_I\cdot\nabla H^I=-\mathrm{n}_{\partial\Omega}|_{p(t)}\cdot\nabla H^I|_{(p(t),t)}
\]
due to the properties of~$J$ and~$H^{I}$ being constant in direction of~$\vec{n}_I$, 
we see that~\eqref{eq_rho_consistent} 
is in fact equivalent to the compatibility condition \eqref{eq_comp_cond2}, 
which in turn is valid because of Remark \ref{th_strongsol_comp_cond}. 

It will turn out that these choices will ensure the requirements for the candidate vector fields. Therefore let us fix these vector fields in the following definition.

\begin{definition}\label{def_calib_contact_local}
	We define $\tilde{\xi}^I, \tilde{\xi}^{\partial\Omega}$ and $\tilde{B}^I, \tilde{B}^{\partial\Omega}$ as in \eqref{eq_xi_I}--\eqref{eq_xi_dO} and \eqref{eq_B_I}--\eqref{eq_B_dO},
	respectively, with the following choices of
	the coefficient functions $\hat{\beta}^I, \hat{\gamma}^I, \hat{\rho}^I\colon I\rightarrow\R$ and $\hat{\beta}^{\partial\Omega}, \hat{\gamma}^{\partial\Omega}, \hat{\rho}^{\partial\Omega}\colon \partial\Omega{\times}[0,T]\rightarrow\R$:
	\begin{enumerate}[leftmargin=0.7cm]
		\item[1.] Let $\hat{\beta}^I\colon I\rightarrow\R$ and
		$\hat{\beta}^{\partial\Omega}\colon {\partial\Omega}{\times}[0,T] \to \R$
		be defined by the right hand side of~\eqref{eq_beta_I} and~\eqref{eq_beta_dO}, respectively,
		in the sense that these coefficient functions are independent of
		the space variable.
		\item[2.] Let $\hat{\gamma}^I\colon I\rightarrow\R$ be determined 
		by \eqref{eq_gamma_I} and \eqref{eq_gamma1_I}.
		\item[3.] Let $\hat{\gamma}^{\partial\Omega}\colon {\partial\Omega}{\times}[0,T]\rightarrow\R$ 
		be defined by the right hand side of~\eqref{eq_gamma_dO} in the sense
		that~$\hat{\gamma}^{\partial\Omega}$ is independent of the space variable.
		\item[4.] Finally, $\hat{\rho}^I\colon I\rightarrow\R$ is given by \eqref{eq_rho_I} and $\hat{\rho}^{\partial\Omega}\colon\partial\Omega\times[0,T]\rightarrow\R$ by \eqref{eq_rho_dO}.
	\end{enumerate}
\end{definition}
 	Note that the equations for $\hat{\gamma}^I$ can be reduced to a parameter-dependent ODE which can be explicitly solved, cf.\ \cite[Proof of Lemma~27]{Fischer2020a}. In the next theorem we prove the properties of the above construction. One may compare with Definition~\ref{def_boundaryAdaptedCalibration} and Theorem~\ref{th_localConstructionsContactPoint}.
\begin{theorem}\label{th_calib_contact}
	In the above situation and with the choices from Definition \ref{def_calib_contact_local} the following holds:
	\begin{enumerate}[leftmargin=0.6cm]
		\item[1.] \textup{Regularity:} $\tilde{\xi}^I$, $\tilde{\xi}^{\partial\Omega}$ are of class $C_tC_x^2\cap C_t^1C_x$ and $\tilde{B}^I$, $\tilde{B}^{\partial\Omega}$ have the regularity $C_tC_x^2$ on their respective domains of definition.
		\item[2.] \textup{Compatibility:} For $t\in[0,T]$ it holds 
		\begin{align}
		\label{eq:compContactPointXi}
		\tilde{\xi}^I|_{(p(t),t)}=\tilde{\xi}^{\partial\Omega}|_{(p(t),t)}, \quad(\partial_t,\nabla)\tilde{\xi}^I|_{(p(t),t)}=(\partial_t,\nabla)\tilde{\xi}^{\partial\Omega}|_{(p(t),t)},\\
		\label{eq:compContactPointVel}
		\tilde{B}^I|_{(p(t),t)}=\frac{\mathrm{d}}{\mathrm{d}t}p(t)=\tilde{B}^{\partial\Omega}|_{(p(t),t)}\quad\text{ and }\quad \nabla\tilde{B}^I|_{(p(t),t)}=\nabla\tilde{B}^{\partial\Omega}|_{(p(t),t)}.
		\end{align}
		\item[3.] \textup{Local gradient flow calibration properties:} We have 
		\begin{align}\label{eq_calib_local1}
		\tilde{\xi}^I|_I=\mathrm{n}_I\quad\text{ and }\quad(\nabla\tilde{\xi}^I)^\mathsf{T}\mathrm{n}_I|_I=0\quad\text{ on }\mathscr{B}_r(p)\cap I.
		\end{align}
		Moreover, it holds $|\tilde{\xi}^I|^2=1-\frac{1}{4}(\beta^I s^I)^4$ 
		\text{ on }$\mathscr{B}_r(p)\cap\overline{\textup{im}(X_I)}$ as well as
		\begin{alignat}{2}\label{eq_calib_local2}
		|\partial_t\tilde{\xi}^I + (\tilde{B}^I\cdot\nabla)\tilde{\xi}^I + 
		(\nabla \tilde{B}^I)^\mathsf{T}\tilde{\xi}^I|&\leq C|s^I|&\quad&\text{ on }\mathscr{B}_r(p)\cap\overline{\textup{im}(X_I)},\\\label{eq_calib_local3}
		|(\partial_t+\tilde{B}^I\cdot\nabla)|\tilde{\xi}^I|^2|&\leq C|s^I|^4 &\quad&\text{ on }\mathscr{B}_r(p)\cap\overline{\textup{im}(X_I)},\\\label{eq_calib_local4}
		|\tilde{\xi}^I\cdot \tilde{B}^I + \nabla\cdot\tilde{\xi}^I|&\leq C|s^I|&\quad&\text{ on }\mathscr{B}_r(p)\cap\overline{\textup{im}(X_I)},\\\label{eq_calib_local5}
		|\tilde{\xi}^I\cdot(\tilde{\xi}^I\cdot\nabla)\tilde{B}^I|&\leq C|s^I|&\quad&\text{ on }\mathscr{B}_r(p)\cap\overline{\textup{im}(X_I)}.
		\end{alignat}
		Additionally, $|\tilde{\xi}^{\partial\Omega}|^2=1-\frac{1}{4}(\beta^{\partial\Omega} s^{\partial\Omega})^4$ \text{ on }$\mathscr{B}_r(p)\cap\overline{\textup{im}(X_{\partial\Omega})}$ and
		\begin{alignat}{2}\label{eq_calib_local6}
		\partial_t\tilde{\xi}^{\partial\Omega} 
		{+} (\tilde{B}^{\partial\Omega}\cdot\nabla)\tilde{\xi}^{\partial\Omega} {+} 
		(\nabla \tilde{B}^{\partial\Omega})^\mathsf{T}\tilde{\xi}^{\partial\Omega}&=0 &\quad&\text{ at }\mathscr{P},\\\label{eq_calib_local7}
		|(\partial_t{+}\tilde{B}^{\partial\Omega}\cdot\nabla)|\tilde{\xi}^{\partial\Omega}|^2
		|&\leq C|s^{\partial\Omega}|^4 
		&\quad&\text{ on }\mathscr{B}_r(p)\cap(\overline{\textup{im}(X_{\partial\Omega})}{\times}[0,T]),
		\\\label{eq_calib_local8}
		\tilde{\xi}^{\partial\Omega}\cdot \tilde{B}^{\partial\Omega} + \nabla\cdot\tilde{\xi}^{\partial\Omega}&=0&\quad&\text{ at }\mathscr{P},\\\label{eq_calib_local9}
		\tilde{\xi}^{\partial\Omega}\cdot(\tilde{\xi}^{\partial\Omega}\cdot\nabla)\tilde{B}^{\partial\Omega}&=0&\quad&\text{ at }\mathscr{P}.
		\end{alignat}
		\item[4.] \textup{Boundary Conditions:} It holds 
		$\tilde{\xi}^{\partial\Omega}\cdot \mathrm{n}_{\partial\Omega}=\cos\alpha$ 
		as well as $\tilde{B}^{\partial\Omega}\cdot \mathrm{n}_{\partial\Omega}=0$ 
		on $\mathscr{B}_r(p)\cap(\partial\Omega{\times}[0,T])$. 
		\item[5.] \textup{Additional Constraints:} $\nabla\tilde{B}^I$ is anti-symmetric on $\mathscr{B}_r(p)\cap I$ and $\nabla\tilde{B}^{\partial\Omega}$ is anti-symmetric on $\mathscr{B}_r(p)\cap
		(\partial\Omega{\times}[0,T])$.
	\end{enumerate}
\end{theorem}
Note that the anti-symmetry condition 5.~for $\nabla\tilde{B}^{\partial\Omega}$ is only used to derive the corresponding condition in Theorem \ref{th_localConstructionsContactPoint}. The latter will be used to obtain the additional conditions \eqref{eq:GradVelXiXi}--\eqref{eq:skewSymmetryGradVelInt}. If these are not needed, then it would suffice to require \eqref{eq_gamma1_dO}--\eqref{eq_rho_dO} at the contact point. Hence, $\hat{\rho}^{\partial\Omega}$ could be chosen space-independent.
\begin{proof}
	\textit{Ad 1.} The regularity properties can be derived by considering the equations determining the functions in Definition \ref{def_calib_contact_local}. For the coefficient, $\hat{\gamma}^I$ this can be done as 
	in~\cite[Proof of Lemma~27]{Fischer2020a}.\\
	\textit{Ad 2.} These compatibility assertions at the contact point follow from the choices in Definition~\ref{def_calib_contact_local} and the derivations from above between Proposition~\ref{th_calib_contact_abl}
	and Definition~\ref{def_calib_contact_local}, 
	except for the time derivative. Concerning the latter, observe that due to Proposition \ref{th_calib_contact_abl} we have to show $\partial_t\tilde{\xi}^I|_{(p(t),t)}=0$ for all $t\in[0,T]$, which 
	by the first identity of Proposition~\ref{th_calib_contact_abl} and $(\mathrm{n}_I\cdot\nabla)H^{I}=0$
	is equivalent to
	\begin{align}\label{eq_th_calib_contact_proof1}
	(-\tau_I\cdot\nabla H^I-\beta^IH^I)|_{(p(t),t)}=0\quad\text{ for }t\in[0,T].
	\end{align}
	We then use \eqref{eq_beta_I}--\eqref{eq_beta_dO},
	again $(\mathrm{n}_I\cdot\nabla)H^{I}=0$,
	and multiply by $\mathrm{n}_{\partial\Omega}\cdot\tau_I|_{(p(t),t)}$ to rewrite the 
	left side of~\eqref{eq_th_calib_contact_proof1} as
	\[
	-\mathrm{n}_{\partial\Omega}\cdot\nabla H^I 
	+ H^{\partial\Omega}H^I\left((\tau_{\partial\Omega}\cdot \mathrm{n}_I)(\mathrm{n}_{\partial\Omega}\cdot\tau_I)-(\mathrm{n}_{\partial\Omega}\cdot \mathrm{n}_I)(\tau_{\partial\Omega}\cdot\tau_I)\right)
	+ (H^I)^2 (\mathrm{n}_{\partial\Omega}\cdot \mathrm{n}_I),
	\]
	where all terms are evaluated at $(p(t),t)$ for arbitrary $t\in[0,T]$. However, due to
	\begin{align*}
	&(\tau_{\partial\Omega}\cdot \mathrm{n}_I)(\mathrm{n}_{\partial\Omega}\cdot\tau_I)-(\mathrm{n}_{\partial\Omega}\cdot \mathrm{n}_I)(\tau_{\partial\Omega}\cdot\tau_I)|_{(p(t),t)}
	=-|\mathrm{n}_{\partial\Omega}\cdot\tau_I|^2-|\mathrm{n}_{\partial\Omega}\cdot \mathrm{n}_I|^2|_{(p(t),t)}\\
	&=-|(\mathrm{n}_{\partial\Omega}\cdot\tau_I)\tau_I+(\mathrm{n}_{\partial\Omega}\cdot \mathrm{n}_I)\mathrm{n}_I|^2|_{(p(t),t)}=-|\mathrm{n}_{\partial\Omega}|^2|_{(p(t),t)}=-1
	\end{align*}
	and $\mathrm{n}_{\partial\Omega}\cdot \mathrm{n}_I|_{(p(t),t)}=\tau_{\partial\Omega}\cdot\tau_I|_{(p(t),t)}$ for $t\in[0,T]$, it turns out that the validity of \eqref{eq_th_calib_contact_proof1} is in fact equivalent to the compatibility condition \eqref{eq_comp_cond2}. The latter holds because of Remark \ref{th_strongsol_comp_cond}.\\
	\textit{Ad 3.} Equation \eqref{eq_calib_local1} is directly clear from the definition \eqref{eq_xi_I} of $\tilde{\xi}^I$ and the formula for $\nabla\tilde{\xi}^I$ in Proposition \ref{th_calib_contact_abl}. Moreover, the identities for $|\tilde{\xi}^I|^2$ and $|\tilde{\xi}^{\partial\Omega}|^2$ follow directly from the definitions \eqref{eq_xi_I}--\eqref{eq_xi_dO}. The latter yield the estimates \eqref{eq_calib_local3} and \eqref{eq_calib_local7} by using the product and chain rule for the differential operators $\partial_t+\tilde{B}^I\cdot\nabla$ and $\partial_t+\tilde{B}^{\partial\Omega}\cdot\nabla$, respectively, as well as
	\begin{alignat*}{2}
	(\partial_t+\tilde{B}^I\cdot\nabla)s^I=\partial_t s^I+H^I&=0&\quad&\text{ on }\mathscr{B}_r(p)\cap\overline{\textup{im}(X_I)},\\ (\partial_t+\tilde{B}^{\partial\Omega}\cdot\nabla)s^{\partial\Omega}=\partial_ts^{\partial\Omega}&=0&\quad&\text{ on }\mathscr{B}_r(p)\cap
	(\overline{\textup{im}(X_{\partial\Omega})}{\times}[0,T]).
	\end{alignat*}
	where we used Remarks~\ref{th_strongsol_notation_tub} and~\ref{th_strongsol_notation_bdry}
	as well as $(\tilde{B}^{\partial\Omega}\cdot\nabla)s^{\partial\Omega}
	= \tilde{B}^{\partial\Omega}\cdot\mathrm{n}_{\partial\Omega}=0$
	along $\partial\Omega{\times}[0,T]$ due to~\eqref{eq_B_dO}.
	
	Next, we observe 
	\[
	(\tilde{\xi}^I\cdot \tilde{B}^I + \nabla\cdot\tilde{\xi}^I)|_I=(H^I+\textup{Tr}\nabla\tilde{\xi}^I)|_I=0\quad\text{ on }\mathscr{B}_r(p)\cap I
	\] 
	because of $\textup{Tr}(\vec{a}\otimes\vec{b})=\vec{a}\cdot\vec{b}$ for vectors $\vec{a},\vec{b}$ 
	in $\R^d$. Then \eqref{eq_calib_local4} follows from a Taylor
	expansion argument, and \eqref{eq_calib_local8} holds due to the compatibility conditions~\eqref{eq:compContactPointXi}--\eqref{eq:compContactPointVel}. Furthermore, by~\eqref{eq_nablaB_I}
	and~\eqref{eq_calib_local1}
	\[
	\tilde{\xi}^I\cdot (\tilde{\xi}^I\cdot\nabla)\tilde{B}^I|_I=\tilde{\xi}^I\cdot (\nabla\tilde{B}^I)\tilde{\xi}^I|_I=\rho^I \mathrm{n}_I\cdot J\mathrm{n}_I|_I=0\quad\text{ on }\mathscr{B}_r(p)\cap I,
	\]
	Hence, \eqref{eq_calib_local5} and \eqref{eq_calib_local9} follow as above.
	
	Finally, we compute the left hand side of \eqref{eq_calib_local2} on $I$. By Proposition \ref{th_calib_contact_abl} and \eqref{eq_nablaB_I}
	\begin{align*}
	&[\partial_t\tilde{\xi}^I + (\tilde{B}^I\cdot\nabla)\tilde{\xi}^I + 
	(\nabla \tilde{B}^I)^\mathsf{T}\tilde{\xi}^I]|_I\\
	&=-\nabla H^I|_I - \beta^I H^I\tau_I + \tau_I \otimes [-H^I\tau_I+\beta^I \mathrm{n}_I]|_I (H^I\mathrm{n}_I + \gamma^I\tau_I)|_I+\rho^I J^\mathsf{T} \mathrm{n}_I|_I\\
	&=-\nabla H^I|_I-(\gamma^I H^I+\rho^I)\tau_I|_I=0,
	\end{align*}
	where the last equation follows from the form \eqref{eq_rho_I} of $\rho^I$
	and $(\mathrm{n}_I\cdot\nabla)H^{I}=0$. Therefore \eqref{eq_calib_local2} is valid because of
	a Taylor expansion argument. Finally, \eqref{eq_calib_local6} holds due to the compatibility
	conditions~\eqref{eq:compContactPointXi}--\eqref{eq:compContactPointVel}.\\
	\textit{Ad 4.} The boundary conditions are evident from the definitions \eqref{eq_xi_dO} and \eqref{eq_B_dO} for the vector fields $\tilde{\xi}^{\partial\Omega}$ and $\tilde{B}^{\partial\Omega}$, respectively.\\
	\textit{Ad 5.} This follows directly from the choices of Definition~\ref{def_calib_contact_local}.
	Indeed, recall that these imply~\eqref{eq_nablaB_I} and~\eqref{eq:skewSymmetryGradientTildeVelBoundary}.
\end{proof}

\subsubsection{Interpolation of local candidates for $(\xi, B)$ at contact points}
\label{sec_calib_contact_intpol}
In this section we piece together the local candidates from the last Section \ref{sec_calib_contact_local} in order to construct the ones in Theorem \ref{th_localConstructionsContactPoint}. Therefore, we use the wedge decomposition from Lemma \ref{th_local_radius_contact} and suitable interpolation functions on the interpolation wedges $W_\pm$:

\begin{lemma}[Interpolation Functions]\label{th_calib_intpol}
	Let $\mathscr{A}$ be a strong solution for mean curvature flow with contact angle $\alpha$ on the time interval $[0,T]$ as in Definition \ref{def_strongSolution}. Moreover, let $p\in\partial I(0)$, $p(t):=\Phi(p,t)$ for $t\in[0,T]$ and $\mathscr{P}=\bigcup_{t\in[0,T]} \{p(t)\} \times \{t\}$ be the corresponding evolving contact point. Finally, let $r=r_p$ be an admissible localization scale as in Lemma \ref{th_local_radius_contact} and recall the notation there.
	
	Then there exists a constant $C>0$ and interpolation functions 
	\[
	\lambda_\pm:\bigcup_{t\in [0,T]} \left(B_r(p(t))\cap\overline{W}_\pm(t)\setminus\{p(t)\}\right)\times\{t\}\rightarrow [0,1]
	\]
	of the class $C_t^1C_x^2$ such that:
	\begin{enumerate}[leftmargin=0.7cm]
		\item[1.] For all $t\in[0,T]$ it holds
		\begin{align}\label{eq_calib_intpol0}
		\lambda_\pm(.,t)&=0\quad\text{ on }(\partial W_\pm(t)\cap\partial W_{\partial\Omega}(t))\setminus\{p(t)\},\\
		\lambda_\pm(.,t)&=1\quad\text{ on }(\partial W_\pm(t)\cap\partial W_{I}(t))\setminus\{p(t)\}.
		\label{eq_calib_intpol1}
		\end{align}
		\item[2.] There is controlled blow-up of the derivatives when approaching the contact point. More precisely for all $t\in[0,T]$ we have in $B_r(p(t))\cap\overline{W}_\pm(t)\setminus\{p(t)\}$:
		\begin{align}\label{eq_calib_intpol_blowup1}
		|(\partial_t,\nabla)\lambda_\pm(.,t)| & \leq C \dist(\cdot,p(t))^{-1},\\
		|\nabla^2\lambda_\pm(.,t)| & \leq C \dist(\cdot,p(t))^{-2}.\label{eq_calib_intpol_blowup2}
		\end{align}
		Moreover, on the wedge lines these derivatives vanish: for all $t\in[0,T]$ it holds
		\begin{align}
		\label{eq_calib_intpol2}
		(\partial_t,\nabla,\nabla^2)\lambda_\pm=0\quad\text{ on } B_r(p(t))\cap\partial W_\pm(t)\setminus\{p(t)\}.
		\end{align}
		\item[3.] The advective derivative with respect to $\frac{\mathrm{d}}{\mathrm{d}t}p$ 
		stays bounded, i.e., for all $t\in[0,T]$:
		\begin{align}\label{eq_calib_intpol_advect}
		\left|\partial_t\lambda_\pm(.,t)+\left(\frac{\mathrm{d}}{\mathrm{d}t}p(t)\cdot\nabla\right)\lambda_\pm(.,t)\right|\leq C\quad\text{ in }B_r(p(t))\cap\overline{W}_\pm(t)\setminus\{p(t)\}.
		\end{align}
	\end{enumerate}
\end{lemma}

\begin{proof}
	One can define these interpolation functions in an explicit and purely geometric way, 
	in fact completely analogously to~\cite[Proof of Lemma~32]{Fischer2020a}.
\end{proof}

\begin{proof}[Proof of Theorem \ref{th_localConstructionsContactPoint}] 
	The procedure is similar to~\cite[Proof of Proposition~26]{Fischer2020a}. In the following, we fix one of the two contact points $p \in \{p_\pm\}$ for convenience. Let $\hat{r}\in(0,r)$
	where~$r=r_p$ denotes the localization scale from Lemma~\ref{th_local_radius_contact}.

	\textit{Step 1: Definition of interpolations.} We define 
	$\hat{\xi}\colon \mathscr{B}_{\hat{r}}(p)\rightarrow\R^2$ by
	\begin{align*}
	\hat{\xi} := \begin{cases}
	\tilde{\xi}^I & \text{ on }\overline{W_I} \cap \mathscr{B}_{\hat{r}}(p),\\
	\tilde{\xi}^{\partial\Omega}
	& \text{ on } (\overline{W_{\partial\Omega}}\cup \overline{W_0})
	\cap \mathscr{B}_{\hat{r}}(p),\\
	\lambda_\pm\tilde{\xi}^I + (1{-}\lambda_\pm)\tilde{\xi}^{\partial\Omega}
	\quad&\text{ on }W_\pm \cap \mathscr{B}_{\hat{r}}(p),
	\end{cases}
	\end{align*}
	where $\tilde{\xi}^I, \tilde{\xi}^{\partial\Omega}$ are from Theorem~\ref{th_calib_contact} and $\lambda_\pm$ are from Lemma~\ref{th_calib_intpol}. In the analogous way, we define $B\colon\mathscr{B}_{\hat{r}}(p)\rightarrow\R^2$ with the $\tilde{B}^I, \tilde{B}^{\partial\Omega}$ from Theorem~\ref{th_calib_contact}. It will turn out to be enough to prove the desired properties for $\hat{\xi},B$ first and then to normalize $\hat{\xi}$ in the end.
This last step gives rise to the smaller domain of definition $\mathscr{B}_{\hat r}(p)$.
	
   \textit{Step 2: 
	Regularity~\emph{\eqref{eq:qualRegularityXiContactPoint}--\eqref{eq:qualRegularityVelContactPoint}} 
	and~\eqref{eq:quantRegularityContactPoint} for $\hat{\xi}, B$.} 
	In terms of the required qualitative regularity, in the following we even show that
	$\hat{\xi} \in C^1(\overline{\mathscr{B}_{\hat{r}}(p)}) \cap C_tC^2_x(\mathscr{B}_{\hat{r}}(p)\setminus\mathscr{P})$
	and $B \in C_tC^1_x(\overline{\mathscr{B}_{\hat{r}}(p)}) \cap C_tC^2_x(\mathscr{B}_{\hat{r}}(p)\setminus\mathscr{P})$.
	First, $\hat{\xi}, B$ are well-defined and have the 
	asserted regularity	on $W_I$ and on $W_{\partial\Omega}\cup \overline{W_0}\setminus\mathscr{P}$ 
	due to Theorem~\ref{th_calib_contact}. Within the interpolation wedges $W_\pm$, 
	we also have this qualitative regularity by Theorem~\ref{th_calib_contact} and Lemma~\ref{th_calib_intpol}. 
	Next, note that thanks to~\eqref{eq_calib_intpol0}, \eqref{eq_calib_intpol1}
	and~\eqref{eq_calib_intpol2} no jumps occur for 
	the vector fields $\hat{\xi},B$ and their required derivatives across the wedge lines 
	$\overline{B_{\hat{r}}(p(t))} \cap \partial W_{\pm}(t) \setminus \{p(t)\}$, $t\in[0,T]$,
	which proves $\hat{\xi},B \in C_tC^2_x(\mathscr{B}_{\hat{r}}(p)\setminus\mathscr{P})$.
	
	For a proof of $\hat{\xi} \in C^1(\overline{\mathscr{B}_{\hat{r}}(p)})$,
	$B \in C_tC^1_x(\overline{\mathscr{B}_{\hat{r}}(p)})$, and the
	quantitative regularity estimate~\eqref{eq:quantRegularityContactPoint},
	we need to study the behaviour when approaching the contact point.
	To this end, one employs the controlled blow-up 
	rates~\eqref{eq_calib_intpol_blowup1}--\eqref{eq_calib_intpol_blowup2} 
	of $\lambda_\pm$ from Lemma~\ref{th_calib_intpol} as well as the compatibility up to 
	first order for $\tilde{\xi}^I,\tilde{\xi}^{\partial\Omega}$ and 
	$\tilde{B}^I,\tilde{B}^{\partial\Omega}$ from Theorem~\ref{th_calib_contact};
	the latter in fact in form of the Lipschitz estimates
	\begin{align}
	\label{eq:secondOrderCompEstimate}
	|\tilde\xi^{I} {-} \tilde\xi^{\partial\Omega}| + 
	|\tilde B^{I} {-} \tilde B^{\partial\Omega}| &\leq C\dist^2_x(\cdot,\mathscr{P})
	&&\text{on } \overline{W_\pm}, 
	\\ \label{eq:firstOrderCompGradientEstimate}
	|(\partial_t,\nabla)\tilde\xi^{I} {-} (\partial_t,\nabla)\tilde\xi^{\partial\Omega}| +
	|\nabla\tilde B^{I} {-} \nabla\tilde B^{\partial\Omega}| &\leq C\dist_x(\cdot,\mathscr{P})
	&&\text{on } \overline{W_\pm}.
	\end{align}
	For example, $\nabla\hat{\xi}$ is continuous on $\overline{\mathscr{B}_{\hat{r}}(p)}$ 
	because on one side $\nabla\tilde{\xi}^I|_{\mathscr{P}}=\nabla\tilde{\xi}^{\partial\Omega}|_{\mathscr{P}}$
	due to~\eqref{eq:compContactPointXi}, so that on the other side by~\eqref{eq:secondOrderCompEstimate},
	\eqref{eq:firstOrderCompGradientEstimate} and~\eqref{eq_calib_intpol_blowup1}
   \begin{align}
	\nonumber
	&\big|\big(\nabla(\lambda_{\pm}\tilde{\xi}^{I}) {+} \nabla((1{-}\lambda_\pm)\tilde\xi^{\partial\Omega})\big)
	- \nabla\tilde{\xi}^I|_{\mathscr{P}}\big|
	\\ \nonumber
	&\leq \lambda_{\pm}\big|\nabla\tilde{\xi}^{I} {-} \nabla\tilde{\xi}^I|_{\mathscr{P}}\big|
	+ (1{-}\lambda_{\pm})\big|\nabla\tilde{\xi}^{\partial\Omega} 
	{-} \nabla\tilde{\xi}^{\partial\Omega}|_{\mathscr{P}}\big|
	+ |\nabla\lambda_{\pm}|\big|\tilde{\xi}^{I} {-} \tilde{\xi}^{\partial\Omega}\big|
	\leq C\dist_x(\cdot,\mathscr{P}).
   \end{align}
  Continuing in this fashion for the remaining
	first order derivatives $\partial_t\hat\xi$ and $\nabla B$,
	and employing an even simpler argument for~$\hat\xi$ and $B$ itself,
	we indeed obtain $\hat{\xi} \in C^1(\overline{\mathscr{B}_{\hat{r}}(p)})$ and
	$B \in C_tC^1_x(\overline{\mathscr{B}_{\hat{r}}(p)})$. 
	A similar argument based on the same 
	ingredients~\eqref{eq_calib_intpol_blowup1}--\eqref{eq_calib_intpol_blowup2} 
	and~\eqref{eq:secondOrderCompEstimate}--\eqref{eq:firstOrderCompGradientEstimate}
	also implies~\eqref{eq:quantRegularityContactPoint}.
   
   \textit{Step 3: Additional properties for $\hat{\xi},B$.} 
	Consider 2.-3.~and 5.~in Theorem~\ref{th_localConstructionsContactPoint} first. 
	The consistency in 2.~(except the $|.|$-constraint) is satisfied for 
	$\hat{\xi}$ on $\mathscr{B}_{\hat{r}}(p)\cap I$ because of its definition and 
	since this is true for $\hat{\xi}^I$ by Theorem~\ref{th_calib_contact}. 
	Moreover, the boundary conditions in Theorem~\ref{th_localConstructionsContactPoint}, 
	3., hold on $\mathscr{B}_{\hat{r}}(p)\cap(\partial\Omega{\times}[0,T])$ since these 
	are valid for $\tilde{\xi}^{\partial\Omega}$ and $\tilde{B}^{\partial\Omega}$ 
	due to Theorem~\ref{th_calib_contact}. Finally, $\nabla B$ is anti-symmetric 
	on $\mathscr{B}_{\hat{r}}(p)	\cap (I\cup(\partial\Omega{\times}[0,T]))$ because of 
	its definition and Theorem~\ref{th_calib_contact}.
   
  Next, we consider the required motion laws in 
	Theorem~\ref{th_localConstructionsContactPoint}, 4.,
	except for~\eqref{eq:localEvolLengthXiContactPoint}. The following estimates 
   \begin{align*}
   |\partial_t \hat{\xi} + (B\cdot\nabla)\hat{\xi}
   + (\nabla B)^\mathsf{T}\hat{\xi}| &\leq C\dist_x(\cdot,I),\\
   |B\cdot\hat{\xi} + \nabla\cdot\hat{\xi}|&\leq C\dist_x(\cdot,I),\\
   |\hat{\xi}\otimes\hat{\xi}:\nabla B|&\leq C\dist_x(\cdot,I)
   \end{align*}
   in $\mathscr{B}_{\hat{r}}(p) \cap (\Omega {\times} [0,T])$ can be shown in a similar manner as 
	in~\cite[Proof of Proposition~26, Steps~3 and~4]{Fischer2020a}, 
	where admittedly the analogue of the third estimate is not proven. 
	The idea, however, is the same for all 
	three estimates: away from the interpolation wedges,
	i.e., in $W_I$ and $W_{\partial\Omega}$, one can simply 
	use the corresponding estimates obtained from Theorem~\ref{th_calib_contact} 
	(with an additional Taylor expansion argument throughout $W_{\partial\Omega}$). 
	On the interpolation wedges, one 
	uses the definition of~$\hat{\xi}$ and~$B$ together with the product rule, 
	the corresponding estimates in Theorem~\ref{th_calib_contact},
	the Lipschitz estimates~\eqref{eq:secondOrderCompEstimate}--\eqref{eq:firstOrderCompGradientEstimate},
	and the controlled blow-up rates for the~$\lambda_{\pm}$ from Lemma~\ref{th_calib_intpol}. 
	For the first estimate, also the control of the advective derivative of $\lambda_\pm$ with respect 
	to $\frac{\mathrm{d}}{\mathrm{d}t}p$ in form of~\eqref{eq_calib_intpol_advect} enters.
   
   \textit{Step 4: Normalization of $\hat{\xi}$ and conclusion of the proof.} 
	In order to divide by $|\hat{\xi}|$ and to carry over the estimates and properties, 
	we have to control $|\hat{\xi}|$ and the first derivatives in a uniform way. 
	Indeed, one can prove
   \begin{align*}
   |1-|\hat{\xi}|^2| \leq C \dist^2_x(\cdot,I),\\
   |(\partial_t,\nabla)|\hat{\xi}|^2| \leq C \dist_x(\cdot,I)
   \end{align*}
   in $\mathscr{B}_{\hat{r}}(p)$ similar as in~\cite[Proof of Proposition~26, Step~5]{Fischer2020a}. 
	Again, away from the interpolation wedges this is a consequence of 
	Theorem~\ref{th_calib_contact} (even with the rates increased by $2$ in the orders). 
	On the interpolation wedges one uses the definition of $\hat{\xi}$, Theorem~\ref{th_calib_contact}, 
	and again the 
	Lipschitz estimates~\eqref{eq:secondOrderCompEstimate}--\eqref{eq:firstOrderCompGradientEstimate}
	as well as the controlled blow-up rates for the~$\lambda_{\pm}$ from Lemma~\ref{th_calib_intpol}. 
   
   Finally, we can choose $\hat{r}>0$ small such 
	that $\frac{1}{2}\leq|\hat{\xi}|^2\leq\frac{3}{2}$ 
	in $\mathscr{B}_{\hat{r}}(p)$. Then we define $\xi:=\hat{\xi}/|\hat{\xi}|$ 
	on~$\overline{\mathscr{B}_{\hat{r}}(p)}$. One can directly check that the 
	properties of $\hat{\xi}$ above carry over to $\xi$. Here one uses the chain 
	rule and the above estimates, cf.\ \cite[Proof of Proposition~26, Step~7]{Fischer2020a} 
	for a similar calculation. Additionally, it holds $|\xi|=1$ in $\mathscr{B}_{\hat{r}}(p)$ 
	by definition and this finally also yields \eqref{eq:localEvolLengthXiContactPoint}. 
	The proof of Theorem~\ref{th_localConstructionsContactPoint} is
	therefore completed.
\end{proof}

\subsection{Local building block for $(\xi, B)$ at the bulk interface}\label{sec_calib_interface}
We proceed with the less technical parts of the local constructions. In this subsection,
we take care of the local building blocks in the vicinity of the bulk interface.
Recalling the notation from Remark~\ref{th_strongsol_notation_tub}, we simply define
\begin{align}
\label{eq:choiceXiBulkInterface}
\xi^{I} := \mathrm{n}_I \quad \text{on } \overline{\mathrm{im}(X_I)} \cap (\overline{\Omega}{\times}[0,T]).
\end{align}
For a suitable definition of the velocity field~$B^{I}$ 
on~$\overline{\mathrm{im}(X_I)} \cap (\overline{\Omega}{\times}[0,T])$,
we first provide some auxiliary constructions. 
Denote by $\theta\colon\R \to [0,1]$ a standard smooth cutoff satisfying
$\theta \equiv 1$ on~$[-\frac{1}{2},\frac{1}{2}]$ and $\theta \equiv 0$ on~$\R \setminus (-1,1)$.
Furthermore, for each of the two contact points~$p_{\pm} \in \partial I(0)$
with associated trajectory~$p_{\pm}(t) \in \partial I(t)$,
denote by~$\hat r_{\pm}$ and~$B^{p_{\pm}}$ the associated localization scale
and local velocity field from Theorem~\ref{th_localConstructionsContactPoint}, respectively.
Define 
\begin{align}
\label{eq:locRadiusTangentialCorrection}
\hat r := \min\Big\{\hat r_{+},\hat r_{-},\frac{1}{3}\min_{t\in[0,T]}\dist(p_+(t),p_-(t))\Big\}
\end{align}
and 
\begin{align}
\widetilde\gamma^I \colon &\overline{\mathrm{im}(X_I)} \cap (\overline{\Omega}{\times}[0,T]) \to \R,
\\ \nonumber
&(x,t) \mapsto \theta\Big(\frac{\dist(x,p_+(t))}{\hat r}\Big) (\tau_I \cdot B^{p_+})(x,t)
+  \theta\Big(\frac{\dist(x,p_-(t))}{\hat r}\Big) (\tau_I \cdot B^{p_-})(x,t),
\\ \label{eq:choiceSkewSymmetricCorrection}
\widetilde\rho^{I} \colon &\overline{\mathrm{im}(X_I)} \cap (\overline{\Omega}{\times}[0,T]) \to \R,
\\ \nonumber
&(x,t) \mapsto -\big((\mathrm{n}_I\cdot\nabla)\widetilde\gamma^{I}\big)(x,t) - H^{I}(x,t)\widetilde\gamma^{I}(x,t)
-\big((\tau_I\cdot\nabla)H^{I}\big)(x,t),
\end{align}
as well as 
\begin{align}
\label{eq:choiceVelBulkInterface}
B^{I} := H^{I} \mathrm{n}_I + (\widetilde\gamma^{I} + \widetilde\rho^{I} s^{I}) \tau_I 
\quad \text{on } \overline{\mathrm{im}(X_I)} \cap (\overline{\Omega}{\times}[0,T]).
\end{align}
With these definitions in place, we then have the following result. 
%

\begin{lemma}
\label{lem:localConstructionsBulkInterface}
Let $\mathscr{A}$ be a strong solution for mean curvature flow 
with contact angle~$\alpha$ on the time interval $[0,T]$ as in 
Definition~\ref{def_strongSolution}, and let the notation 
from Remark~\ref{th_strongsol_notation_tub} be in place. 
Then, the local vector fields~$\xi^{I}$ and~$B^{I}$ defined by~\eqref{eq:choiceXiBulkInterface}
and~\eqref{eq:choiceVelBulkInterface} satisfy
\begin{align}
\label{eq:regBulkInterface}
\xi^{I} &\in (C^0_tC^4_x \cap C^1_tC^2_x)(\overline{\textup{im}(X_I)} 
\cap (\overline{\Omega}{\times}[0,T])),
\\
B^{I} &\in C_tC^1_x(\overline{\textup{im}(X_I)}
\cap (\overline{\Omega}{\times}[0,T])) \cap
C_tC^2_x(\textup{im}(X_I) \cap (\Omega {\times} [0,T])),
\end{align}
and there exists $C>0$ such that
\begin{align}
\label{eq:quantRegBulkInterface}
|\nabla^2 B^{I}| \leq C \quad\text{in } \textup{im}(X_I) \cap (\Omega {\times} [0,T]).
\end{align}
Moreover, it holds
\begin{align}
\label{eq:bulkInterface1}
\partial_t s^{I} + (B^{I}\cdot\nabla) s^{I} &= 0,
\\ \label{eq:bulkInterface2}
\partial_t\xi^{I} + (B^{I}\cdot\nabla) \xi^{I} + (\nabla B^{I})^\mathsf{T} \xi^{I} &= 0,
\\ \label{eq:bulkInterface3}
\xi^{I} \cdot (\partial_t+(B^{I}\cdot\nabla))\xi^{I} &= 0,
\\ \label{eq:bulkInterface4}
|B^{I}\cdot\xi^{I} + \nabla\cdot\xi^{I}| &\leq C |s^{I}|,
\\ \label{eq:bulkInterface5}
|\xi^{I} \cdot \nabla^\mathrm{sym} B^{I}| &\leq C|s^{I}|
\end{align}
on the whole space-time domain $\textup{im}(X_I) \cap (\Omega {\times} [0,T])$ as well as
\begin{align}
\label{eq:consistencyBulkInterface}
\xi^{I} &= \mathrm{n}_{I}
\text{ and } \big(\nabla\xi^{I}\big)^\mathsf{T}\mathrm{n}_{I} = 0
&& \text{along } I. 
\end{align}

For each~$p_{\pm} \in \partial I(0)$
with associated trajectory~$p_{\pm}(t) \in \partial I(t)$, denote further
by~$\xi^{p_{\pm}}$ and~$B^{p_{\pm}}$ the local vector fields 
from Theorem~\ref{th_localConstructionsContactPoint}, respectively.
These vector fields are compatible with~$\xi^{I}$ and~$B^{I}$ in the sense that
\begin{align}
\label{eq:localComp1}
\big|\xi^{I} - \xi^{p_\pm}\big| + \big|(\nabla\xi^{I}  -  \nabla\xi^{p_\pm})^\mathsf{T}\xi^{I}\big|
&\leq C|s^{I}|,
\\ \label{eq:localComp2}
\big|(\xi^{I}  -  \xi^{p_\pm})\cdot\xi^{I}\big| &\leq C|s^{I}|^2,
\\ \label{eq:localComp3}
\big|B^{I}  -  B^{p_\pm}\big| &\leq C|s^{I}|,
\\ \label{eq:localComp4}
\big|(\nabla B^{I}  -  \nabla B^{p_\pm})^\mathsf{T}\xi^{I}\big| &\leq C|s^{I}|.
\end{align}
on $B_{\hat r/2}(p_\pm(t)) \cap \big(W^{p_\pm}_I(t) \cup W^{p_\pm}_+(t) \cup W^{p_\pm}_-(t)\big) \subset \Omega$
for all $t \in [0,T]$, where~$\hat r$ was defined by~\eqref{eq:locRadiusTangentialCorrection}
and $W^{p_\pm}_I(t),W^{p_\pm}_+(t),W^{p_\pm}_-(t)$  denote the wedges from Lemma~\ref{th_local_radius_contact}
with respect to the contact points~$p_\pm$, respectively. 
\end{lemma}

\begin{proof}
The asserted regularity~\eqref{eq:regBulkInterface}--\eqref{eq:quantRegBulkInterface} is a consequence
of the definitions~\eqref{eq:choiceXiBulkInterface} and~\eqref{eq:choiceVelBulkInterface},
Remark~\ref{th_strongsol_notation_tub}, and Theorem~\ref{th_localConstructionsContactPoint}.
The identities~\eqref{eq:bulkInterface1}--\eqref{eq:bulkInterface3}
and the estimate~\eqref{eq:bulkInterface4} follow by straightforward
arguments, e.g., along the lines of~\cite[Proof of Lemma~22]{Fischer2020a}.
The estimate~\eqref{eq:bulkInterface5} is immediate from
the definitions~\eqref{eq:choiceXiBulkInterface} and~\eqref{eq:choiceVelBulkInterface}, 
the fact that it holds $(\mathrm{n}_I\cdot\nabla)H^{I}=0$ due to~\eqref{eq:normalExtensionNormalCurvature},
and the precise choice~\eqref{eq:choiceSkewSymmetricCorrection} of~$\widetilde\rho^{I}$.
Note also for this computation that~$\widetilde\gamma^{I}$ is not constant
in normal direction.
The properties of~\eqref{eq:consistencyBulkInterface} hold true
because of~\eqref{eq:choiceXiBulkInterface} and~\eqref{eq_nabla_normale_und_tau}.

The local compatibility estimates~\eqref{eq:localComp1}--\eqref{eq:localComp4}
follow along the lines of~\cite[Proof of Proposition~33]{Fischer2020a}.
Note for the proof of~\eqref{eq:localComp4} that the first order perturbation in the 
definition~\eqref{eq:localComp4} of~$B^{I}$ does not play a role since we contract in the end with~$\xi^{I}$.
\end{proof}

\subsection{Local building block for $(\xi, B)$ at the domain boundary}\label{sec_calib_bdry}
This subsection concerns the definition of local building blocks for $(\xi,B)$, which will be used
near to the domain boundary but away from the bulk interface.
This constitutes the by far easiest part of the local constructions.
Indeed, the conditions~\eqref{eq_calibration1}--\eqref{eq_calibration4}
only require to provide estimates with respect to the distance to the bulk interface and 
not the domain boundary. Essentially, we only have to respect the required boundary
conditions~\eqref{eq_boundaryCondXi} and~\eqref{eq_boundaryCondVelocity}.
The most straightforward choice to satisfy these consists of
\begin{align}
\label{eq:choiceXiVelBoundary}
\xi^{\partial\Omega}(x,t) := (\cos\alpha) \mathrm{n}_{\partial\Omega}(P^{\partial\Omega}(x)),
\,B^{\partial\Omega}(x,t) := 0, \,
(x,t) \in \overline{\mathrm{im}(X_{\partial\Omega})} {\times} [0,T],
\end{align}
for which we also recall the notation from Remark~\ref{th_strongsol_notation_bdry}.
This choice will also become handy for a proof of
the additional requirements~\eqref{eq:GradVelXiXi} and~\eqref{eq:GradVelTangent}.
Note finally that by Remark~\ref{th_strongsol_notation_bdry} it holds
\begin{align}
\label{eq:regXiBoundary}
\xi^{\partial\Omega} \in C^\infty_tC^2_x(\overline{\mathrm{im}(X_{\partial\Omega})}{\times}[0,T]).
\end{align}

\subsection{Global construction of $(\xi, B)$}\label{sec_calib_global}
We finally perform a gluing construction to lift
the local constructions from the previous three subsections
to a global construction. To fix notation, we denote 
again by $\xi^{I},B^{I}\colon\overline{\mathrm{im}(X_I)}\cap(\overline{\Omega}{\times}[0,T])\to\Rd[2]$
the local building blocks in the vicinity of the bulk interface
as defined by~\eqref{eq:choiceXiBulkInterface} and~\eqref{eq:choiceVelBulkInterface},
and by $\xi^{\partial\Omega},B^{\partial\Omega}\colon
\overline{\mathrm{im}(X_{\partial\Omega})}{\times}[0,T]\to\Rd[2]$
the local building blocks in the vicinity of the domain boundary as
defined by~\eqref{eq:choiceXiVelBoundary}. For each of the two contact points~$p_{\pm} \in \partial I(0)$
with associated trajectory~$p_{\pm}(t) \in \partial I(t)$, we further denote
by $\xi^{p_{\pm}},B^{p_{\pm}} \colon \overline{\mathscr{B}_{\hat r_{\pm}}(p_{\pm})}
\cap(\overline{\Omega}{\times}[0,T]) \to \Rd[2]$ the local building blocks
in the vicinity of the two moving contact points as provided by 
Theorem~\ref{th_localConstructionsContactPoint}, respectively.
For a recap of the definition of the associated space-time domains~$\mathrm{im}(X_I)$,
$\mathrm{im}(X_{\partial\Omega})$ and~$\mathscr{B}_{\hat r_{\pm}}(p_{\pm})$,
we refer to Remark~\ref{th_strongsol_notation_tub},
Remark~\ref{th_strongsol_notation_bdry} and Theorem~\ref{th_localConstructionsContactPoint}, respectively.

Before we proceed with the gluing construction, let us 
fix a final localization scale~$\bar r \in (0,1]$. To this
end, recall first from Remark~\ref{th_strongsol_notation_tub}	 
and Remark~\ref{th_strongsol_notation_bdry}	the choice of
the localization scales~$r_I\in (0,1]$ and~$r_{\partial\Omega} \in (0,1]$,
respectively. For each of the moving contact points, we then chose
a corresponding localization radius~$r_{\pm} \in (0,\min\{r_I,r_{\partial\Omega}\}]$
such that the conclusions of Lemma~\ref{th_local_radius_contact} hold true.
Next, we derived the existence of a potentially even smaller 
radius~$\hat r_{\pm} \in (0, r_{\pm}]$ so that also the conclusions
of Theorem~\ref{th_localConstructionsContactPoint} are satisfied.
Recalling in the end the definition~\eqref{eq:locRadiusTangentialCorrection}
of the localization scale~$\hat r$, we eventually define
\begin{align}
\label{eq:choiceMinLocScale}
\bar r := \frac{1}{2}\hat r = \frac{1}{2}
\min\Big\{\hat r_{+},\hat r_{-},\frac{1}{3}\min_{t\in[0,T]}\dist(p_+(t),p_-(t))\Big\}.
\end{align}

Apart from~$\bar r$, it turns out to be convenient to introduce a second
localization scale~$\bar\delta \in (0,1]$ which is chosen as follows. 
Recall from Remark~\ref{th_strongsol_notation_tub}	
and Remark~\ref{th_strongsol_notation_bdry}	
the definition of the tubular neighborhood diffeomorphisms~$X_I$
and~$X_{\partial\Omega}$, respectively. Their restrictions
to $I {\times} (-\bar\delta\bar r,\bar\delta\bar r)$ and
$\partial\Omega {\times} (-\bar\delta\bar r,\bar\delta\bar r)$
will be denoted by~$X_I^{\bar r,\bar\delta}$ and~$X_{\partial\Omega}^{\bar r,\bar\delta}$,
respectively. We then choose $\bar\delta\in (0,1]$ small enough such that
for all $t \in [0,T]$ the images of $X_I^{\bar r,\bar\delta}(\cdot,t,\cdot)$
and~$X_{\partial\Omega}^{\bar r,\bar\delta}$ do not overlap away from
the contact points~$p_\pm(t)$:
\begin{align}
\overline{\mathrm{im}(X_I^{\bar r,\bar\delta}(\cdot,t,\cdot))}
\setminus \bigcup_{p\in\{p_+,p_-\}} B_{\bar r}(p(t)) 
\label{eq:2ndMinLocScale}
\cap \overline{\mathrm{im}(X_{\partial\Omega}^{\bar r,\bar\delta})} 
\setminus \bigcup_{p\in\{p_+,p_-\}} B_{\bar r}(p(t)) = \emptyset.
\end{align}

The implementation of the gluing construction now works as follows. Given a set of
localization functions
\begin{align*}
\eta_I,\eta_{p_{\pm}},\eta_{\partial\Omega},\widetilde\eta_I,
\widetilde\eta_{p_{\pm}},\widetilde\eta_{\partial\Omega}
\colon &\overline{\Omega}{\times}[0,T] \to [0,1],
\end{align*}
whose supports are at least required to satisfy the natural conditions
$\supp\eta_I \cup \supp\widetilde\eta_I \subset \overline{\mathrm{im}(X_I)}\cap(\overline{\Omega}{\times}[0,T])$,
$\supp\eta_{p_{\pm}} \cup \supp\widetilde\eta_{p_{\pm}} \subset \overline{\mathscr{B}_{\hat r_{\pm}}(p_{\pm})}
\cap(\overline{\Omega}{\times}[0,T])$ and $\supp\eta_{\partial\Omega} \cup \supp\widetilde\eta_{\partial\Omega}
\subset \overline{\mathrm{im}(X_{\partial\Omega})}{\times}[0,T]$, one then defines
\begin{align}
\label{def:globalXi}
\xi \colon &\overline{\Omega}{\times}[0,T] \to \Rd[2],
&& (x,t) \mapsto \big(\eta_I \xi^{I} + \eta_{p_+} \xi^{p_+} 
+ \eta_{p_-} \xi^{p_-} + \eta_{\partial\Omega} \xi^{\partial\Omega}\big)(x,t),
\\ \label{def:globalVel}
B \colon &\overline{\Omega}{\times}[0,T] \to \Rd[2],
&& (x,t) \mapsto \big(\widetilde\eta_I B^{I} + \widetilde\eta_{p_+} B^{p_+} 
+ \widetilde\eta_{p_-} B^{p_-} + \widetilde\eta_{\partial\Omega} B^{\partial\Omega}\big)(x,t).
\end{align}
The main task then is to extract conditions on the localization functions
guaranteeing that the vector fields~$\xi$ and~$B$ defined by~\eqref{def:globalXi}
and~\eqref{def:globalVel}, respectively, satisfy the requirements
of a boundary adapted gradient flow calibration of 
Definition~\ref{def_boundaryAdaptedCalibration}. Such conditions 
are captured by the following definition. If one does not rely
on the additional constraints~\eqref{eq:GradVelXiXi}--\eqref{eq:skewSymmetryGradVelInt},
we remark that one may in fact choose $\widetilde\eta_I = \eta_I$,
$\widetilde\eta_{p_{\pm}} = \eta_{p_{\pm}}$
and~$\widetilde\eta_{\partial\Omega} = \eta_{\partial\Omega}$.

\begin{definition}
\label{def:admissibleLocFunctions}
In the setting of this subsection, we call a collection of
maps  $\eta_I,\eta_{p_{\pm}},\eta_{\partial\Omega},\widetilde\eta_I,
\widetilde\eta_{p_{\pm}},\widetilde\eta_{\partial\Omega}
\colon \overline{\Omega}{\times}[0,T] \to [0,1]$
an \emph{admissible family of localization functions}
if they satisfy the following list of requirements:
\begin{enumerate}[leftmargin=0.5cm]
\item[1.] \textit{(Regularity)} It holds
					\begin{align}
					\label{eq:qualRegCutoffs}
					\eta_I,\eta_{p_{\pm}},\eta_{\partial\Omega},\widetilde\eta_I,
					\widetilde\eta_{p_{\pm}},\widetilde\eta_{\partial\Omega}
					\in C^1(\overline{\Omega}{\times}[0,T])
					\cap C_tC^2_x(\Omega{\times}[0,T]),
					\end{align}
					and there exists $C>0$ such that
					\begin{align}
					\label{eq:quantRegCutoffs}
					|\nabla^2(\eta_I,\eta_{p_{\pm}},\eta_{\partial\Omega},\widetilde\eta_I,
					\widetilde\eta_{p_{\pm}},\widetilde\eta_{\partial\Omega})|
					\leq C \quad\text{in } \Omega{\times}[0,T].
					\end{align}
\item[2.] \textit{(Localization)} We have for all~$t \in [0,T]$
					\begin{align}
					\label{eq:supportBulkInterfaceCutoff}
					\supp\eta_I(\cdot,t) \subset \supp\widetilde\eta_I(\cdot,t) &\subset 
					\big(\mathrm{im}(X_I^{\bar r,\bar\delta}(\cdot,t,\cdot))\setminus\partial\Omega\big) \cup \partial I(t),
					\\ \label{eq:supportBoundaryCutoff}
					\supp\eta_{\partial\Omega}(\cdot,t) \subset \supp\widetilde\eta_{\partial\Omega}(\cdot,t)
					&\subset \mathrm{im}(X_{\partial\Omega}^{\bar r,\bar\delta}) \setminus (I(t) \cap \Omega),
					\\ \label{eq:supportContactPointCutoff}
					\supp\eta_{p_{\pm}}(\cdot,t) \subset \supp\widetilde\eta_{p_{\pm}}(\cdot,t)
					&\subset B_{\bar r}(p_{\pm}(t)) \cap \overline{\Omega},
					\end{align}
					such that for all~$t \in [0,T]$ one has minimal overlaps in the sense of
					\begin{align}
					\label{eq:overlapTwoContactPoints}
					&\supp\widetilde\eta_{p_+}(\cdot,t) \cap \supp\widetilde\eta_{p_-}(\cdot,t) = \emptyset,
					\\ \label{eq:overlapContactPointBulkInterface}
					&B_{\bar r}(p_{\pm}(t)) \cap \supp\widetilde\eta_I(\cdot,t) 
					\subset B_{\bar r}(p_{\pm}(t))
					\cap \big(W^{p_\pm}_I(t) \cup W^{p_\pm}_+(t) \cup W^{p_\pm}_-(t)\big),
					\\ \label{eq:overlapContactPointBoundary}
					&B_{\bar r}(p_{\pm}(t)) \cap \supp\widetilde\eta_{\partial\Omega}(\cdot,t) 
					\subset B_{\bar r}(p_{\pm}(t))
					\cap \big(W^{p_\pm}_{\partial\Omega}(t) \cup W^{p_\pm}_+(t) \cup W^{p_\pm}_-(t)\big),
					\\ \label{eq:overlapBoundaryBulkInterface}
					&\supp\widetilde\eta_{\partial\Omega}(\cdot,t) \cap \supp\widetilde\eta_I(\cdot,t) \subset 
					\bigcup_{p\in\{p_{\pm}\}} B_{\bar r}(p(t)) \cap \big(W^{p}_+(t) \cup W^{p}_-(t)\big).
					\end{align}
					Here, $W^{p_\pm}_I(t),W^{p_\pm}_+(t),W^{p_\pm}_-(t),W^{p_\pm}_{\partial\Omega}(t)$  
					denote the wedges from Lemma~\ref{th_local_radius_contact}
					with respect to the contact points~$p_\pm$. We emphasize that the 
					relations~\eqref{eq:overlapTwoContactPoints}--\eqref{eq:overlapBoundaryBulkInterface}
					also hold with $(\widetilde\eta_I,\widetilde\eta_{p_\pm},\widetilde\eta_{\partial\Omega})$
					replaced by $(\eta_I,\eta_{p_\pm},\eta_{\partial\Omega})$ thanks to the first inclusions
					of~\eqref{eq:supportBulkInterfaceCutoff}--\eqref{eq:supportContactPointCutoff}, respectively.
\item[3.] \textit{(Partition of unity)} Define $\eta_{\mathrm{bulk}} := 
					1 - \eta_I - \eta_{p_+} - \eta_{p_-} - \eta_{\partial\Omega}$. Then
					\begin{align}
					\label{eq:partitionOfUnity}
					\eta_{\mathrm{bulk}} \in [0,1] \text{ on } \overline{\Omega}{\times}[0,T]
					\quad\text{and}\quad 
					\eta_{\mathrm{bulk}} = 0 \text{ along } I \cup (\partial\Omega{\times}[0,T]).
					\end{align}
					The same properties are satisfied by $\widetilde\eta_{\mathrm{bulk}} := 
					1 - \widetilde\eta_I - \widetilde\eta_{p_+} - \widetilde\eta_{p_-} - \widetilde\eta_{\partial\Omega}$.
\item[4.] \textit{(Coercivity estimates)} There exists $C \geq 1$ such that 
					\begin{align}
					\label{eq:lowerBoundBulkCutoff}
					C^{-1} \min\{1,\dist^2(\cdot,I(t)),\dist^2(\cdot,\partial\Omega)\} 
					&\leq \eta_{\mathrm{bulk}}(\cdot,t),
					\\ \label{eq:upperBoundsCutoffs}
					(\eta_{\mathrm{bulk}} + \widetilde\eta_{\mathrm{bulk}} + 
					\eta_{\partial\Omega} + \widetilde\eta_{\partial\Omega})(\cdot,t)
					&\leq C\min\{1,\dist^2(\cdot,I(t))\},
					\\ \label{eq:upperBoundsGradientCutoffs}
					|(\nabla,\partial_t)(\eta_{\mathrm{bulk}},\widetilde\eta_{\mathrm{bulk}},
					\eta_{\partial\Omega},\widetilde\eta_{\partial\Omega})|(\cdot,t)
					&\leq C\min\{1,\dist(\cdot,I(t))\},
					\\ \label{eq:upperBoundGradientInterfaceCutoffBoundary}
					|(\nabla,\partial_t)(\eta_I,\widetilde\eta_I)|(\cdot,t) 
					&\leq C\min\{1,\dist(\cdot,\partial\Omega))\},
					\end{align}
					throughout~$\Omega$ for all~$t\in[0,T]$. Moreover, 
					there exists $C > 0$ such that
					\begin{align}
					\label{eq:tangentialSquaredCutoff}
					\dist^2(x,I(t)) \leq C(1 {-} \eta_{p_{\pm}})(x,t),
					\quad t \in [0,T],\,
					x \in B_{\bar r}(p_{\pm}(t)) \cap W^{p_\pm}_{\partial\Omega}(t).
					\end{align}
\item[5.] \textit{(Motion laws)} There exists $C>0$ such that
					\begin{align}
					\label{eq:transportBulkCutoff}
					|\partial_t\eta_{\mathrm{bulk}} + (B\cdot\nabla)\eta_{\mathrm{bulk}}|(\cdot,t)
					&\leq C\min\{1,\dist^2(\cdot,I(t))\},
					\\ \label{eq:transportBoundaryCutoff}
					|\partial_t\eta_{\partial\Omega} + (B\cdot\nabla)\eta_{\partial\Omega}|(\cdot,t)
					&\leq C\min\{1,\dist^2(\cdot,I(t))\} 
					\end{align}
					throughout~$\Omega$ for all~$t\in[0,T]$, where~$B$ is defined by~\eqref{def:globalVel}.
\item[6.] \textit{(Additional boundary constraints)}
					Finally, it holds for all $t \in [0,T]$
					\begin{align}
					\label{eq:tildePlateau}
					\widetilde\eta_{p_\pm}(\cdot,t) = 1 - \widetilde\eta_{\partial\Omega}(\cdot,t)
					&= 1 &&\text{along } \supp\eta_{p_\pm}(\cdot,t) \cap \partial\Omega,
					\\ \label{eq:tangentialBump}
					(\mathrm{n}_{\partial\Omega}\cdot\nabla)\widetilde\eta_{\partial\Omega}(\cdot,t)
					= (\mathrm{n}_{\partial\Omega}\cdot\nabla)\widetilde\eta_{p_{\pm}}(\cdot,t)
					&= 0 &&\text{along } \partial\Omega.
					\end{align}
\end{enumerate}
\end{definition}

With the above definition in place, we then have the following result.

\begin{proposition}
\label{prop:gluingConstruction}
Let $\eta_I,\eta_{p_{\pm}},\eta_{\partial\Omega},\widetilde\eta_I,
\widetilde\eta_{p_{\pm}},\widetilde\eta_{\partial\Omega}
\colon \overline{\Omega}{\times}[0,T] \to [0,1]$
be an admissible family of localization functions
in the sense of Definition~\ref{def:admissibleLocFunctions}.
Then the vector fields~$\xi$ and~$B$ defined by
means of~\eqref{def:globalXi} and~\eqref{def:globalVel}, respectively,
satisfy the requirements~\emph{\eqref{eq_regularityXi}--\eqref{eq_regularityB}}
and~\emph{\eqref{eq_consistencyProperty}--\eqref{eq_boundaryCondVelocity}}
of a boundary adapted gradient flow calibration of 
Definition~\ref{def_boundaryAdaptedCalibration} as well as 
the additional constraints~\emph{\eqref{eq:GradVelXiXi}--\eqref{eq:skewSymmetryGradVelInt}}.
\end{proposition}

It thus remains to construct an admissible family of localization functions.

\begin{proposition}
\label{prop:existenceAdmissibleLocFunctions}
In the setting as described at the beginning of this subsection,
there exist $\eta_I,\eta_{p_{\pm}},\eta_{\partial\Omega},\widetilde\eta_I,
\widetilde\eta_{p_{\pm}},\widetilde\eta_{\partial\Omega}
\colon \overline{\Omega}{\times}[0,T] \to [0,1]$ which
form an admissible family of localization functions
in the sense of Definition~\ref{def:admissibleLocFunctions}.
\end{proposition}

The remainder of this subsection is devoted to the proofs of
these two results.

\begin{proof}[Proof of Proposition~\ref{prop:gluingConstruction}]
The proof is split into several steps.

\textit{Step 1: Proof of regularity~\emph{\eqref{eq_regularityXi}--\eqref{eq_regularityB}}.} 
This is an obvious consequence of the definitions~\eqref{def:globalXi}
and~\eqref{def:globalVel}, the regularity of the local building 
blocks~\eqref{eq:qualRegularityXiContactPoint}--\eqref{eq:quantRegularityContactPoint},
\eqref{eq:regBulkInterface}--\eqref{eq:quantRegBulkInterface},
and~\eqref{eq:regXiBoundary} (recall from~\eqref{eq:choiceXiVelBoundary}
that $B^{\partial\Omega}=0$), respectively, as well as the regularity of the localization 
functions~\eqref{eq:qualRegCutoffs}--\eqref{eq:quantRegCutoffs}.

\textit{Step 2: Proof of consistency~\emph{\eqref{eq_consistencyProperty}} 
and boundary conditions~\emph{\eqref{eq_boundaryCondXi}--\eqref{eq_boundaryCondVelocity}}.}
Plugging in the definition~\eqref{def:globalXi}, exploiting the
properties~\eqref{eq:supportBulkInterfaceCutoff}--\eqref{eq:supportContactPointCutoff}
and~\eqref{eq:partitionOfUnity}, as well as recalling~\eqref{eq:consistencyContactPoints}
and~\eqref{eq:consistencyBulkInterface} yields the first part of~\eqref{eq_consistencyProperty} due to
\begin{align*}
\xi(\cdot,t)|_{I(t) \cap \Omega} = 
\sum_{n \in \{I,p_+,p_-\}} (\eta_n\xi^n)(\cdot,t)|_{I(t) \cap \Omega}
&= \Big(\sum_{n \in \{I,p_+,p_-\}} \eta_n(\cdot,t)|_{I(t) \cap \Omega}\Big)\mathrm{n}_I(\cdot,t) 
\\
&= \mathrm{n}_I(\cdot,t).
\end{align*}
Relying in addition on~\eqref{eq:upperBoundsGradientCutoffs}
shows the second part of~\eqref{eq_consistencyProperty} due to
\begin{align*}
&(\nabla\xi(\cdot,t))^{\mathsf{T}}|_{I(t) \cap \Omega} \mathrm{n}_I(\cdot,t)
\\&
= \sum_{n \in \{I,p_+,p_-\}} \eta_n(\cdot,t)(\nabla\xi^n(\cdot,t))^\mathsf{T}|_{I(t) \cap \Omega}
\mathrm{n}_I(\cdot,t) + \sum_ {n \in \{I,p_+,p_-\}} \nabla\eta_n(\cdot,t)|_{I(t) \cap \Omega}
\\&
= - \nabla\eta_{\mathrm{bulk}}(\cdot,t)|_{I(t) \cap \Omega} = 0.
\end{align*}
The same properties of the localization functions together
with~\eqref{eq:boundaryConditionsContactPoints} and~\eqref{eq:choiceXiVelBoundary}
also imply~\eqref{eq_boundaryCondXi} as the following computation reveals:
\begin{align*}
\xi(\cdot,t)|_{\partial\Omega}\cdot\mathrm{n}_{\partial\Omega}
= \sum_{n \in \{\partial\Omega,p_+,p_-\}} (\eta_n\xi^n)(\cdot,t)|_{\partial\Omega}
\cdot\mathrm{n}_{\partial\Omega}
&= \Big(\sum_{n \in \{\partial\Omega,p_+,p_-\}} \eta_n(\cdot,t)|_{\partial\Omega}\Big)\cos\alpha
\\&
= \cos\alpha.
\end{align*}
One may finally infer~\eqref{eq_boundaryCondVelocity} analogously.

\textit{Step 3: Proof of coercivity estimate~\emph{\eqref{eq_quadraticLengthConstraint}}.}
Fix a point $(x,t) \in \Omega {\times} [0,T]$. 					
Let $n_{\mathrm{max}}(x,t) \in \{I,p_+,p_-,\partial\Omega\}$ be defined by
$n_{\mathrm{max}} = \arg\max_{n \in \{I,p_+,p_-,\partial\Omega\}} \eta_n(x,t)$.
Without loss of generality, we may assume that 
there exists $n \in \{I,p_+,p_-,\partial\Omega\}$ such that $x \in \supp\eta_{n}(\cdot,t)$
and that this topological feature satisfies $n=n_{\mathrm{max}}(x,t)$. Moreover, we may assume without
loss of generality that it holds $\eta_n(x,t) \geq \frac{1}{4}$. Indeed, 
otherwise we get $|\xi(x,t)| \leq \frac{3}{4}$ as a consequence
of the definition~\eqref{def:globalXi}, the triangle inequality,
and the fact that at most three localization functions can
be simultaneously strictly positive due to~\eqref{eq:overlapTwoContactPoints}.
The estimate $|\xi(x,t)| \leq \frac{3}{4}$ in turn of course implies~\eqref{eq_quadraticLengthConstraint}
for such~$(x,t)$. 

We now distinguish between two cases. First, if $n = n_{\mathrm{max}}(x,t) = \partial\Omega$,
it follows from~\eqref{eq:partitionOfUnity}, $\eta_{\partial\Omega}(x,t) \geq \frac{1}{4}$ and
$|\xi^{\partial\Omega}(x,t)|=\cos\alpha$, cf.\ \eqref{eq:choiceXiVelBoundary}, that
\begin{align*}
|\xi(x,t)| &\leq \eta_{\partial\Omega}(x,t)\cos\alpha
+ \sum_{n \in \{I,p_+,p_-\}} \eta_n(x,t)
\\&
\leq 1 - \eta_{\partial\Omega}(x,t)(1{-}\cos\alpha) 
\leq 1 - \frac{1}{4}(1{-}\cos\alpha),
\end{align*}
which in turn implies~\eqref{eq_quadraticLengthConstraint}.

If instead $n = n_{\mathrm{max}}(x,t) \in \{I,p_-,p_+\}$,
it follows from the localization properties~\eqref{eq:supportBulkInterfaceCutoff}
and~\eqref{eq:supportContactPointCutoff}
that $x \in \big(B_{\bar r}(p_+(t)) \cup B_{\bar r}(p_-(t))\big)
\cup \mathrm{im}(X_I^{\bar r,\bar\delta}(\cdot,t,\cdot))$.
In case of $x \notin B_{\bar r}(p_+(t)) \cup B_{\bar r}(p_-(t))$,
condition~\eqref{eq:2ndMinLocScale} ensures that there exists
$C \geq 1$ such that $\dist(x,I(t)) \leq C\dist(x,\partial\Omega)$.
We thus infer~\eqref{eq_quadraticLengthConstraint} for such~$x$
from~\eqref{eq:lowerBoundBulkCutoff} due to $|\xi(x,t)| \leq 1 - \eta_{\mathrm{bulk}}(x,t)$.
In case of $x \in B_{\bar r}(p_+(t)) \cup B_{\bar r}(p_-(t))$,
say for concreteness $x \in B_{\bar r}(p_+(t))$, the same conclusions
hold true if in addition $x \in W^{p_+}_I(t) \cup W^{p_+}_+(t) \cup W^{p_-}_I(t)$.
Hence, consider finally $x \in W^{p_+}_{\partial\Omega}(t) \cap B_{\bar r}(p_+(t))$.
Due to the localization 
properties~\eqref{eq:overlapTwoContactPoints}--\eqref{eq:overlapBoundaryBulkInterface}
it follows $\eta_I(x,t) = \eta_{p_-}(x,t) = 0$. 
Recalling further $|\xi^{\partial\Omega}(x,t)|=\cos\alpha$, cf.\ \eqref{eq:choiceXiVelBoundary},
we may then estimate
\begin{align*}
1 - |\xi(x,t)| \geq 1 - \eta_{p_+}(x,t) - (\cos\alpha)\eta_{\partial\Omega}(x,t)
\geq (1 {-} \cos\alpha)(1 {-} \eta_{p_+})(x,t).
\end{align*}
Hence, \eqref{eq_quadraticLengthConstraint} follows from~\eqref{eq:tangentialSquaredCutoff}.

\textit{Step 4: From local to global compatibility estimates.}
We claim that there exists a constant $C>0$ such that for all $n \in \{I,p_+,p_-\}$ it holds
in $\Omega{\times}[0,T]$ 
\begin{align}
\label{eq:globalComp1}
\chi_{\supp\widetilde\eta_n}
\big(|\xi^n {-} \xi| + |(\nabla\xi^{n} {-} \nabla\xi)^\mathsf{T}\xi^n|\big)
&\leq C\min\{1,\dist(\cdot,I)\}, 
\\ \label{eq:globalComp2}
\chi_{\supp\widetilde\eta_n} 
|(\xi^{n} {-} \xi) \cdot \xi^{n}|
&\leq C\min\{1,\dist^2(\cdot,I)\}, 
\\ \label{eq:globalComp3}
\chi_{\supp\widetilde\eta_n} |(B^{n} {-} B)|
&\leq C\min\{1,\dist(\cdot,I)\},
\\ \label{eq:globalComp4}
\chi_{\supp\widetilde\eta_n} |(\nabla B^{n} {-} \nabla B)^{\mathsf{T}}\xi^{n}|
&\leq C\min\{1,\dist(\cdot,I)\}.
\end{align}

Plugging in the definitions~\eqref{def:globalXi} and~\eqref{def:globalVel}
and making use of the estimate~\eqref{eq:upperBoundsCutoffs} entails 
\begin{align*}
\chi_{\supp\widetilde\eta_n} (\xi^{n} {-} \xi)
&= \chi_{\supp\widetilde\eta_n} \sum_{n' \in \{I,p_+,p_-\} \setminus\{n\}}
\eta_{n'}(\xi^{n} {-} \xi^{n'}) + O(\min\{1,\dist^2(\cdot,I)\}),
\\
\chi_{\supp\widetilde\eta_n} (B^n {-} B)
&= \chi_{\supp\widetilde\eta_n} \sum_{n' \in \{I,p_+,p_-\} \setminus\{n\}}
\widetilde\eta_{n'}(B^{n} {-} B^{n'}) + O(\min\{1,\dist^2(\cdot,I)\}).
\end{align*}
Hence, due to~\eqref{eq:overlapTwoContactPoints}, \eqref{eq:overlapContactPointBulkInterface}
and the choice~\eqref{eq:choiceMinLocScale}, the first part of~\eqref{eq:globalComp1}
follows from the first part of~\eqref{eq:localComp1}
and the first identity of the previous display.
The estimate~\eqref{eq:globalComp3} in turn follows
from~\eqref{eq:localComp3} and the second identity of the previous display.
Furthermore, the estimate~\eqref{eq:globalComp2} follows from~\eqref{eq:localComp2} and
contracting the first identity of the previous display with~$\xi^{n}$. 

We proceed computing based on the definition~\eqref{def:globalXi},
the estimate~\eqref{eq:upperBoundsGradientCutoffs}, the property~\eqref{eq:overlapTwoContactPoints},
and the first part of the estimate~\eqref{eq:localComp1}
\begin{align*}
&\chi_{\supp\widetilde\eta_n} (\nabla\xi^{n} {-} \nabla\xi)^\mathsf{T}\xi^n
\\&
= \chi_{\supp\widetilde\eta_n} \sum_{n' \in \{I,p_+,p_-\} \setminus\{n\}}
\big(\eta_{n'} (\nabla\xi^{n} {-} \nabla\xi^{n'})^\mathsf{T}\xi^n
+ \big((\xi^{n} {-} \xi^{n'})\cdot\xi^{n}\big) \nabla\eta_n\big)
\\&~~~
+ O(\min\{1,\dist^2(\cdot,I)\})
\\&
= \chi_{\supp\widetilde\eta_n} \sum_{n' \in \{I,p_+,p_-\} \setminus\{n\}}
\eta_{n'} (\nabla\xi^{n} {-} \nabla\xi^{n'})^\mathsf{T}\xi^n
+ O(\min\{1,\dist(\cdot,I)\}).
\end{align*}
Hence, the second part of~\eqref{eq:globalComp1}
follows now from~\eqref{eq:overlapTwoContactPoints},
\eqref{eq:overlapContactPointBulkInterface}, the choice~\eqref{eq:choiceMinLocScale},
and the second part of~\eqref{eq:localComp1}.
The proof of the remaining estimate~\eqref{eq:globalComp4} is analogous.

\textit{Step 5: Proof of error estimates~\emph{\eqref{eq_calibration1}--\eqref{eq_calibration4}}.}
For a proof of the estimates~\eqref{eq_calibration1}
and~\eqref{eq_calibrationEvolByMCF}, we may simply refer 
to the corresponding argument given in~\cite[Proof of Lemma~42]{Fischer2020a}.
Indeed, the whole structure of this argument solely relies on the structure of the 
definitions~\eqref{def:globalXi} and~\eqref{def:globalVel},
the coercivity estimates~\eqref{eq:upperBoundsCutoffs} and~\eqref{eq:upperBoundsGradientCutoffs},
the compatibility estimates~\eqref{eq:globalComp1}, \eqref{eq:globalComp3},
and~\eqref{eq:globalComp4}, the local counterparts~\eqref{eq:localEvolXiContactPoint}
and~\eqref{eq:bulkInterface2} of~\eqref{eq_calibration1}, the local
counterparts~\eqref{eq:localMCFContactPoint} and~\eqref{eq:bulkInterface4}
of~\eqref{eq_calibrationEvolByMCF}, and finally the regularity estimates
of the involved constructions. 

Next, we provide a proof of~\eqref{eq_calibration4}.
Starting with the definition~\eqref{def:globalXi},
the bound~\eqref{eq:upperBoundsCutoffs}, adding zero in form 
of $\xi^{n'} = (\xi^{n'} {-} \xi) + (\xi {-} \xi^{n}) + \xi^{n}$,
and the estimate~\eqref{eq:globalComp1}, we get
\begin{align*}
\xi \otimes \xi : \nabla B &= \sum_{n,n'\in\{I,p_+,p_-\}} \eta_{n'}\eta_{n}
\xi^{n'}\cdot(\nabla B)^\mathsf{T}\xi^{n} + O(\min\{1,\dist^2(\cdot,I)\})
\\&
= \sum_{n \in \{I,p_+,p_-\}} \eta_{n} \xi^{n} \cdot(\nabla B)^\mathsf{T}\xi^{n}
+ O(\min\{1,\dist(\cdot,I)\}).
\end{align*}
Hence, \eqref{eq_calibration4} is entailed by its local 
counterparts~\eqref{eq:localGradVelXiXiContactPoint} and~\eqref{eq:bulkInterface5}. 

In comparison to~\cite[Proof of Lemma~42]{Fischer2020a},
some changes are necessary for the proof of~\eqref{eq_calibration2}
due to the weaker compatibility estimate~\eqref{eq:globalComp3}.
In fact, the only essential difference concerns the verification 
of the preliminary estimate
\begin{align}
\label{eq:2ndOrderTransport}
&\xi \cdot (\partial_t\xi {+} (B\cdot\nabla)\xi)
\\ \nonumber
&= \sum_{n,n'\in\{I,p_+,p_-\}} \eta_{n'}\eta_{n}
\xi^{n} \cdot (\partial_t\xi^{n'} {+} (B^{n'}\cdot\nabla)\xi^{n'})
+ O(\min\{1,\dist^2(\cdot,I)\}).
\end{align}
Post-processing~\eqref{eq:2ndOrderTransport} to~\eqref{eq_calibration2}
can be done analogously to~\cite[Proof of Lemma~42]{Fischer2020a} 
because this argument solely relies on exploiting
the local estimates~\eqref{eq:localEvolXiContactPoint}--\eqref{eq:localEvolLengthXiContactPoint}
and~\eqref{eq:bulkInterface2}--\eqref{eq:bulkInterface3}, respectively,
as well as the compatibility estimates~\eqref{eq:globalComp1} and~\eqref{eq:globalComp4}.

Hence, it remains to carry out a proof of~\eqref{eq:2ndOrderTransport} for which we give details now.
Inserting the definition~\eqref{def:globalVel}, making use of
the estimates~\eqref{eq:upperBoundsCutoffs} and~\eqref{eq:transportBoundaryCutoff}, 
and adding zero in form of $B = (B{-}B^{n'}) + B^{n'}$ 
we obtain
\begin{align}
\nonumber
\xi \cdot (\partial_t\xi {+} (B\cdot\nabla)\xi)
&= \sum_{n \in \{I,p_+,p_-\}} \eta_n\xi^{n} \cdot (\partial_t\xi {+} (B\cdot\nabla)\xi)
+ O(\min\{1,\dist^2(\cdot,I)\})
\\& \label{eq:2ndOrderTransportAux1}
= \sum_{n,n' \in \{I,p_+,p_-\}}
\eta_n\eta_{n'}\xi^{n} \cdot (\partial_t\xi^{n'} {+} (B^{n'}\cdot\nabla)\xi^{n'})
\\&~~~ \nonumber
+ \sum_{n,n' \in \{I,p_+,p_-\}}
\eta_n\eta_{n'}\xi^{n} \cdot ((B{-}B^{n'})\cdot\nabla)\xi^{n'}
\\&~~~ \nonumber
+ \sum_{n,n' \in \{I,p_+,p_-\}} 
\eta_n (\xi^{n} \cdot \xi^{n'}) (\partial_t\eta_{n'} {+} (B\cdot\nabla)\eta_{n'})
\\&~~~ \nonumber
+ O(\min\{1,\dist^2(\cdot,I)\}).
\end{align}
Adding zero several times in form of $\xi^{n} \cdot \xi^{n'}
= |\xi|^2 - |\xi^{n} {-} \xi|^2 + (\xi^{n}{-}\xi)\cdot\xi^{n}
+ \xi^{n'} \cdot (\xi^{n'} {-} \xi) + (\xi^{n} {-} \xi^{n'}) \cdot (\xi^{n'} {-} \xi)$,
we get from~\eqref{eq:overlapTwoContactPoints},
\eqref{eq:overlapContactPointBulkInterface}, the choice~\eqref{eq:choiceMinLocScale},
and the compatibility estimates~\eqref{eq:localComp1},
\eqref{eq:globalComp1} and~\eqref{eq:globalComp2} that
\begin{align*}
&\sum_{n,n' \in \{I,p_+,p_-\}} 
\eta_n (\xi^{n} \cdot \xi^{n'}) (\partial_t {+} (B\cdot\nabla))\eta_{n'}
\\&
= |\xi|^2 \sum_{n,n' \in \{I,p_+,p_-\}}
\eta_n (\partial_t {+} (B\cdot\nabla))\eta_{n'}
+ O(\min\{1,\dist^2(\cdot,I)\}).
\end{align*}
Based on~\eqref{eq:upperBoundsCutoffs}, \eqref{eq:transportBulkCutoff}
and~\eqref{eq:transportBoundaryCutoff} this may be upgraded to
\begin{align}
\nonumber
&\sum_{n,n' \in \{I,p_+,p_-\}} 
\eta_n (\xi^{n} \cdot \xi^{n'}) (\partial_t {+} (B\cdot\nabla))\eta_{n'}
\\& \nonumber
= |\xi|^2 \sum_{n' \in \{I,p_+,p_-\}} 
(\partial_t {+} (B\cdot\nabla))\eta_{n'} + O(\min\{1,\dist^2(\cdot,I)\})
\\& 
= -|\xi|^2 (\partial_t {+} (B\cdot\nabla))\eta_{\mathrm{bulk}}
+ O(\min\{1,\dist^2(\cdot,I)\})
\label{eq:2ndOrderTransportAux2}
=  O(\min\{1,\dist^2(\cdot,I)\}).
\end{align}
Due to~\eqref{eq:overlapTwoContactPoints},
\eqref{eq:overlapContactPointBulkInterface}, the choice~\eqref{eq:choiceMinLocScale},
as well as the estimates~\eqref{eq:localComp1}, \eqref{eq:globalComp1}, \eqref{eq:globalComp3},
and~\eqref{eq:upperBoundsCutoffs}, we may further estimate
\begin{align}
\nonumber
&\sum_{n,n' \in \{I,p_+,p_-\}}
\eta_n\eta_{n'}\xi^{n} \cdot ((B{-}B^{n'})\cdot\nabla)\xi^{n'}
\\& \nonumber
= \sum_{n' \in \{I,p_+,p_-\}}
\eta_{n'}\xi^{n'} \cdot ((B{-}B^{n'})\cdot\nabla)\xi^{n'}
+ O(\min\{1,\dist^2(\cdot,I)\})
\\& \nonumber
= \sum_{n' \in \{I,p_+,p_-\}}
\eta_{n'}\xi^{n'} \cdot ((B{-}B^{n'})\cdot\nabla)\xi
+ O(\min\{1,\dist^2(\cdot,I)\})
\\& \nonumber
= \sum_{n' \in \{I,p_+,p_-\}}
\eta_{n'}\xi \cdot ((B{-}B^{n'})\cdot\nabla)\xi
+ O(\min\{1,\dist^2(\cdot,I)\})
\\& \label{eq:2ndOrderTransportAux3}
= O(\min\{1,\dist^2(\cdot,I)\}).
\end{align}
The combination of~\eqref{eq:2ndOrderTransportAux1},
\eqref{eq:2ndOrderTransportAux2}, and~\eqref{eq:2ndOrderTransportAux3}
thus implies~\eqref{eq:2ndOrderTransport} and therefore
concludes the proof of~\eqref{eq_calibration2}.

\textit{Step 6: Proof of additional estimates~\emph{\eqref{eq:GradVelXiXi}--\eqref{eq:skewSymmetryGradVelInt}}.}
Plugging in the definition~\eqref{def:globalVel} and exploiting the
properties~\eqref{eq:upperBoundsCutoffs}--\eqref{eq:upperBoundsGradientCutoffs}, we obtain
\begin{align*}
&(\xi \cdot \nabla^\mathrm{sym} B)(\cdot,t)
\\
&= \sum_{n \in \{I,p_+,p_-\}} (\widetilde\eta_n \xi \cdot \nabla^\mathrm{sym} B^n)(\cdot,t) 
+ \sum_{n \in \{I,p_+,p_-\}} (\xi \cdot (B^n \otimes \nabla\widetilde\eta_n)^\mathrm{sym})(\cdot,t)
\\&~~~
+ O(\min\{1,\dist(\cdot,I(t))\}).
\end{align*}
Due to the estimates~\eqref{eq:globalComp1}, \eqref{eq:globalComp3} 
and~\eqref{eq:upperBoundsGradientCutoffs}, the previous display upgrades to
\begin{align*}
&(\xi \cdot \nabla^\mathrm{sym} B)(\cdot,t)
\\
&= \sum_{n \in \{I,p_+,p_-\}} (\widetilde\eta_n \xi^n \cdot \nabla^\mathrm{sym} B^n)(\cdot,t) 
- (\xi \cdot (B \otimes \nabla\widetilde\eta_{\mathrm{bulk}})^\mathrm{sym})(\cdot,t)
\\&~~~
+ O(\min\{1,\dist(\cdot,I(t))\})
\\
&= \sum_{n \in \{I,p_+,p_-\}} (\widetilde\eta_n \xi^n \cdot \nabla^\mathrm{sym} B^n)(\cdot,t) 
+ O(\min\{1,\dist(\cdot,I(t))\}).
\end{align*}
Hence, \eqref{eq:skewSymmetryGradVelInt} follows from
the previous display and exploiting~\eqref{eq:gradVelSkewSym} and~\eqref{eq:bulkInterface5}.

For a proof of~\eqref{eq:GradVelXiXi} and~\eqref{eq:GradVelTangent},
let~$v$ be either $\xi(\cdot,t)$ or $\mathrm{n}_{\partial\Omega}$.
We first compute based on the definition~\eqref{def:globalXi}, the
properties~\eqref{eq:supportBulkInterfaceCutoff}--\eqref{eq:supportContactPointCutoff},
\eqref{eq:upperBoundGradientInterfaceCutoffBoundary}
and~\eqref{eq:tildePlateau} of the localization functions, as well as
the properties~\eqref{eq:gradVelSkewSym} and~\eqref{eq:choiceXiVelBoundary}
\begin{align*}
&(v \cdot \nabla^{\mathrm{sym}}B)(\cdot,t)|_{\partial\Omega}
\\&
= \sum_{n \in \{p_+,p_-\}} (\widetilde\eta_n v \cdot 
\nabla^\mathrm{sym} B^n)(\cdot,t)|_{\partial\Omega}
+ \sum_{n \in \{p_+,p_-\}} (v \cdot (B^n \otimes 
\nabla\widetilde\eta_n)^\mathrm{sym})(\cdot,t)|_{\partial\Omega}
\\&
= \sum_{n \in \{p_+,p_-\}} (v \cdot (B^n \otimes 
\nabla\widetilde\eta_n)^\mathrm{sym})(\cdot,t)|_{\partial\Omega \setminus \supp \eta_n(\cdot,t)}.
\end{align*}
Due to the second item of~\eqref{eq:boundaryConditionsContactPoints} and~\eqref{eq:tangentialBump},
we have on one side that $(B^n\otimes\nabla\widetilde\eta_n)(\cdot,t)$
only carries a $\tau_{\partial\Omega}\otimes\tau_{\partial\Omega}$ component
along $\partial\Omega \cap \supp\widetilde\eta_n(\cdot,t)$, $n \in \{p_+,p_-\}$.
On the other side, by the localization 
properties~\eqref{eq:supportBulkInterfaceCutoff}--\eqref{eq:supportContactPointCutoff}
and~\eqref{eq:overlapTwoContactPoints} as well as the definition~\eqref{eq:choiceXiVelBoundary},
we have $\tau_{\partial\Omega} \cdot \xi(\cdot,t) = 0$ along
$(\partial\Omega\cap\supp\widetilde\eta_n(\cdot,t)) \setminus \supp\eta_n(\cdot,t)$,
$n \in \{p_+,p_-\}$. Hence, for both choices of $v=\xi(\cdot,t)$
and $v=\mathrm{n}_{\partial\Omega}$ we obtain from these two facts and
the previous display that~\eqref{eq:GradVelXiXi} and~\eqref{eq:GradVelTangent}
are satisfied.

This in turn concludes the proof of Proposition~\ref{prop:gluingConstruction}.
\end{proof}

\begin{proof}[Proof of Proposition~\ref{prop:existenceAdmissibleLocFunctions}]
First, we provide the definition of the localization functions.
Afterwards, we prove that the required properties are satisfied.
This second part of the proof will be split into several steps.

Let us start with the choice of some suitable quadratic cutoff functions.
To this end, fix two smooth cutoffs $\theta,\widetilde\theta\colon\Rd[] \to [0,1]$
such that $\theta \equiv 1$ on $[-1/2,1/2]$ and $\theta \equiv 0$ on $\Rd[] \setminus (-1,1)$
as well as $\widetilde\theta \equiv 1$ on $[-3/2,3/2]$ and $\widetilde\theta \equiv 0$
on $\Rd[] \setminus (-2,2)$. Define
\begin{align}
\label{eq:auxCutoffXi}
\zeta(s) &:= \theta(s^2)(1 - s^2), \quad s \in \Rd[],
\\
\label{eq:auxCutoffVel}
\widetilde\zeta(s) &:= \widetilde\theta(s^2)
\begin{cases}
1 & |s| \leq 1, \\
1 - (s - 1)^2 & ~s ~> 1, \\
1 - (s - ({-}1))^2 & ~s ~< {-}1.
\end{cases}
\end{align}
We refer to Figure~\ref{fig_zetas} for a sketch.

For a given $\delta \in (0,\bar\delta]$ and a given~$\bar c \in (0,1]$ which we fix later,
we next define
\begin{align}
\label{eq:auxInterfaceCutoff}
(\zeta_I,\widetilde\zeta_I)(x,t) &:=
\begin{cases}
(\zeta,\widetilde\zeta)\big(\frac{s^I(x,t)}{\delta\bar r}\big) &
(x,t) \in \overline{\mathrm{im}(X_I)}, \\
(0,0) & \text{else},
\end{cases}
\\
\label{eq:auxBoundaryCutoff}
(\zeta_{\partial\Omega},\widetilde\zeta_{\partial\Omega})(x,t) 
&:= (\zeta,\widetilde\zeta)\Big(\frac{s^{\partial\Omega}(x)}{\delta\bar r}\Big),
\quad (x,t) \in \Rd[2] \times [0,T],
\\
\label{eq:auxContactPointsCutoff}
(\zeta_{p_\pm},\widetilde\zeta_{p_\pm})(x,t) 
&:= \begin{cases}
(\zeta,\widetilde\zeta)\Big(\frac{\dist(P^{\partial\Omega}(x),p_{\pm}(t))}{\bar c\bar r}\Big)
& (x,t) \in \overline{\mathrm{im}(X_{\partial\Omega})}, \\
(0,0) & \text{else}.
\end{cases}
\end{align}
For each of the two contact points~$p_{\pm}$,
denote by $\lambda^{p}_{\pm}$ the two associated interpolation
functions from Lemma~\ref{th_calib_intpol}. We have everything
in place to write down the definition of the localization
functions $\eta_I,\eta_p,\eta_{\partial\Omega}$: for 
all $(x,t) \in \overline{\Omega} {\times} [0,T]$, let
\begin{align}
\label{def:interfaceCutoff}
\eta_I(x,t) :=
\begin{cases}
\zeta_I(x,t) & x \in \overline{\mathrm{im}(X_I^{\bar r,\bar\delta}(\cdot,t,\cdot))}
\setminus \bigcup_{p \in \{p_+,p_-\}} B_{\bar r}(p(t)), 
\\
(1{-}\zeta_{\partial\Omega})\zeta_I(x,t)
& x \in B_{\bar r}(p_{\pm}(t)) \cap W^{p_{\pm}}_I(t),
\\
\lambda^{p_{\pm}}_{\pm}(1{-}\zeta_{\partial\Omega})\zeta_I(x,t)
& x \in B_{\bar r}(p_{\pm}(t)) \cap W^{p_{\pm}}_{\pm}(t),
\\
0 & \text{else},
\end{cases}
\end{align}
and
\begin{align}
\label{def:boundaryCutoff}
\eta_{\partial\Omega}(x,t) :=
\begin{cases}
\zeta_{\partial\Omega}(x,t) 
& x \in \overline{\mathrm{im}(X_{\partial\Omega}^{\bar r,\bar\delta}(\cdot,t,\cdot))}
\setminus \bigcup_{p \in \{p_+,p_-\}} B_{\bar r}(p(t)), 
\\
(1{-}\zeta_{p_{\pm}})\zeta_{\partial\Omega}(x,t)
& x \in B_{\bar r}(p_{\pm}(t)) \cap W^{p_{\pm}}_{\partial\Omega}(t),
\\
(1{-}\lambda^{p_{\pm}}_{\pm})(1{-}\zeta_{p_{\pm}})\zeta_{\partial\Omega}(x,t)
& x \in B_{\bar r}(p_{\pm}(t)) \cap W^{p_{\pm}}_{\pm}(t),
\\
0 & \text{else},
\end{cases}
\end{align}
as well as
\begin{align}
\label{def:contactPointCutoff}
\eta_{p_{\pm}}(x,t) :=
\begin{cases}
\zeta_{\partial\Omega}\zeta_I(x,t)
& x \in B_{\bar r}(p_{\pm}(t)) \cap W^{p_{\pm}}_{I}(t) , 
\\
\zeta_{p_{\pm}}\zeta_{\partial\Omega}(x,t)
& x \in B_{\bar r}(p_{\pm}(t)) \cap W^{p_{\pm}}_{\partial\Omega}(t),
\\
\lambda^{p_{\pm}}_{\pm}\zeta_{\partial\Omega}\zeta_I(x,t)
+ (1 {-} \lambda^{p_{\pm}}_{\pm})\zeta_{p_{\pm}}\zeta_{\partial\Omega}(x,t)
& x \in B_{\bar r}(p_{\pm}(t)) \cap W^{p_{\pm}}_{\pm}(t),
\\
0 & \text{else}.
\end{cases}
\end{align}
The localization functions $\widetilde\eta_I,\widetilde\eta_{p_{\pm}},\widetilde\eta_{\partial\Omega}$
are defined analogously in the sense that one simply replaces the cutoffs
$(\zeta_I,\zeta_{p_{\pm}},\zeta_{\partial\Omega})$ by
$(\widetilde\zeta_I,\widetilde\zeta_{p_{\pm}},\widetilde\zeta_{\partial\Omega})$.
We now continue with the verification of the required properties
from Definition~\ref{def:admissibleLocFunctions}.

\begin{figure}
	\begin{tikzpicture}[scale=2]
	\draw[->] (-2.3,0)--(2.3,0) node [below] {$r$};
	
	\draw (0.5,-0.01)--(0.5,0.01) node [below] {$1/2$};
	\draw (1,-0.01)--(1,0.01) node [below] {$1$};
	\draw (1.5,-0.01)--(1.5,0.01) node [below] {$3/2$};
	\draw (2,-0.01)--(2,0.01) node [below] {$2$};
	\draw (-0.5,-0.01)--(-0.5,0.01) node [below] {$-1/2$};
	\draw (-1,-0.01)--(-1,0.01) node [below] {$-1$};
	\draw (-1.5,-0.01)--(-1.5,0.01) node [below] {$-3/2$};
	\draw (-2,-0.01)--(-2,0.01) node [below] {$-2$};
	
	\draw[->] (0,-0.1)--(0,1.2);

	\draw (-0.15,0.975) node [above] {$1$};
	
	\draw (0.43,0.5) node  {$\zeta(r)$};
	
	\draw [thick, color=black, domain=-0.5:0.5, samples=200, smooth]
	plot (\x,{1-\x*\x});
	
	\draw [dashed, color=black, domain=-1:-0.5, samples=200, smooth]
	plot (\x,{1-\x*\x});
	\draw [dashed, color=black, domain=0.5:1, samples=200, smooth]
	plot (\x,{1-\x*\x});
	
	\draw[thick,smooth] (.5,.75)    to [out=-45,in=180]   (0.9,0);
	\draw[thick,smooth] (0.9,0)--(1.2,0);
	
	\draw[thick,smooth] (-.5,.75)    to[out=-135,in=0]    (-0.9,0);
	\draw[thick,smooth] (-0.9,0)--(-1.2,0);
	
	\draw (1.43,0.5) node  {$\tilde{\zeta}(r)$};
	
	\draw[thick,smooth] (-1,1)--(1,1);
	
	\draw [thick, color=black, domain=1:1.5, samples=200, smooth]
	plot (\x,{1-(\x-1)*(\x-1)});
	\draw [dashed, color=black, domain=1.5:2, samples=200, smooth]
	plot (\x,{1-(\x-1)*(\x-1)});
	
	\draw [thick, color=black, domain=-1:-1.5, samples=200, smooth]
	plot (\x,{1-(\x+1)*(\x+1)});
	\draw [dashed, color=black, domain=-1.5:-2, samples=200, smooth]
	plot (\x,{1-(\x+1)*(\x+1)});
	
	\draw[thick,smooth] (1.5,0.75)    to [out=-45,in=180]   (1.9,0);
	\draw[thick,smooth] (1.9,0)--(2.2,0);
	
	\draw[thick,smooth] (-1.5,0.75)    to [out=-135,in=0]   (-1.9,0);
	\draw[thick,smooth] (-1.9,0)--(-2.2,0);
	\end{tikzpicture}
	\caption{The cutoff functions $\zeta$ and $\widetilde{\zeta}$.}
	\label{fig_zetas}
\end{figure}
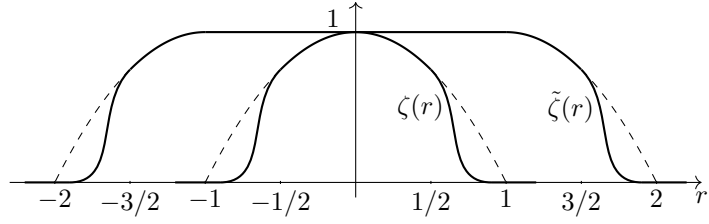

\textit{Step 1: Regularity and localization properties.}
In order to guarantee the required regularity~\eqref{eq:qualRegCutoffs}
and~\eqref{eq:quantRegCutoffs} for the piecewise 
definitions~\eqref{def:interfaceCutoff}--\eqref{def:contactPointCutoff}
it suffices to choose~$\delta \in (0,\bar{\delta}]$ and~$\bar c \in (0,1]$
small enough, respectively, and to recall the regularity assertions
from Remark~\ref{th_strongsol_notation_tub}, Remark~\ref{th_strongsol_notation_bdry}
and Lemma~\ref{th_calib_intpol}. For more details, one may consult
the arguments given in~\cite[Proof of Lemma~34, Steps~1--3]{Fischer2020a}.
Furthermore, the localization 
properties~\eqref{eq:supportBulkInterfaceCutoff}--\eqref{eq:overlapBoundaryBulkInterface}
are straightforward consequences of the definitions~\eqref{eq:auxCutoffXi}--\eqref{def:contactPointCutoff},
the choices~\eqref{eq:choiceMinLocScale}--\eqref{eq:2ndMinLocScale},
the properties~\eqref{eq_calib_intpol0}--\eqref{eq_calib_intpol1},
as well as choosing~$\delta \in (0,\bar{\delta}]$ and~$\bar c \in (0,1]$
sufficiently small again.

\textit{Step 2: Partition of unity.} For a proof of~\eqref{eq:partitionOfUnity},
we first provide some useful identities which will also be 
of help in later stages of the proof. Fix $t \in [0,T]$.
Due to the localization 
properties~\eqref{eq:supportBulkInterfaceCutoff}--\eqref{eq:supportContactPointCutoff}, 
it holds
\begin{align}
\label{eq:bulkCutoffBulk}
\eta_{\mathrm{bulk}}(\cdot,t) = 1
\quad\text{in } \overline{\Omega} \setminus \Big(\overline{\mathrm{im}(X_I^{\bar r,\bar\delta}(\cdot,t,\cdot))}
\cup \overline{\mathrm{im}(X_{\partial\Omega}^{\bar r,\bar\delta})}
\cup \bigcup_{p \in \{p_\pm\}} B_{\bar r}(p(t))\Big).
\end{align}
Using in addition the 
properties~\eqref{eq:overlapTwoContactPoints}--\eqref{eq:overlapBoundaryBulkInterface}
and the choice~\eqref{eq:2ndMinLocScale}, we also obtain from
plugging in the definitions~\eqref{def:interfaceCutoff}--\eqref{def:contactPointCutoff}
\begin{align}
\label{eq:bulkCutoffInterfaceBulk}
\eta_{\mathrm{bulk}}(\cdot,t)
&= (1 {-} \eta_I)(\cdot,t) 
= (1 {-} \zeta_I)(\cdot,t) 
\\& \nonumber
\qquad\qquad\text{in } \overline{\Omega} \cap \Big(
\overline{\mathrm{im}(X_I^{\bar r,\bar\delta}(\cdot,t,\cdot))}
\setminus \bigcup_{p \in \{p_\pm\}} B_{\bar r}(p(t))\Big),
\\
\label{eq:bulkCutoffBoundaryBulk}
\eta_{\mathrm{bulk}}(\cdot,t)
&= (1 {-} \eta_{\partial\Omega})(\cdot,t) 
= (1 {-} \zeta_{\partial\Omega})(\cdot,t)
\\& \nonumber
\qquad\qquad\text{in } \overline{\Omega} \cap \Big(
\overline{\mathrm{im}(X_{\partial\Omega}^{\bar r,\bar\delta})}
\setminus \bigcup_{p \in \{p_\pm\}} B_{\bar r}(p(t))\Big),
\\
\label{eq:bulkCutoffContactPointInterfaceWedge}
\eta_{\mathrm{bulk}}(\cdot,t)
&= (1 {-} \eta_{I} {-} \eta_{p_{\pm}})(\cdot,t)
= (1 {-} \zeta_{I})(\cdot,t)
\\& \nonumber
\qquad\qquad\text{in } \overline{\Omega} \cap \big(
B_{\bar r}(p_{\pm}(t)) \cap W^{p_\pm}_I(t)\big),
\\
\label{eq:bulkCutoffContactPointBoundaryWedge}
\eta_{\mathrm{bulk}}(\cdot,t)
&= (1 {-} \eta_{\partial\Omega} {-} \eta_{p_{\pm}})(\cdot,t)
= (1 {-} \zeta_{\partial\Omega})(\cdot,t)
\\& \nonumber
\qquad\qquad\text{in } \overline{\Omega} \cap \big(
B_{\bar r}(p_{\pm}(t)) \cap W^{p_\pm}_{\partial\Omega}(t)\big),
\\
\label{eq:bulkCutoffContactPointInterpolationWedge}
\eta_{\mathrm{bulk}}(\cdot,t)
&= (1 {-} \eta_{I} {-} \eta_{\partial\Omega} {-} \eta_{p_{\pm}})(\cdot,t)
= \big(\lambda^{p_{\pm}}_{\pm}(1 {-} \zeta_{I})
+ (1 {-} \lambda^{p_{\pm}}_{\pm})(1 {-} \zeta_{\partial\Omega})\big)(\cdot,t)
\\& \nonumber
\qquad\qquad\text{in } \overline{\Omega} \cap \big(
B_{\bar r}(p_{\pm}(t)) \cap W^{p_\pm}_{\pm}(t)\big).
\end{align}
The identities~\eqref{eq:bulkCutoffBulk}--\eqref{eq:bulkCutoffContactPointInterpolationWedge}
immediately imply~\eqref{eq:partitionOfUnity} due to the
definitions~\eqref{eq:auxInterfaceCutoff}--\eqref{eq:auxContactPointsCutoff}
and the properties of the wedges, cf.\ Lemma~\ref{th_local_radius_contact}.
Since the identities~\eqref{eq:bulkCutoffBulk}--\eqref{eq:bulkCutoffContactPointInterpolationWedge}
hold analogously with the localization functions $(\eta_I,\eta_{p_\pm},\eta_{\partial\Omega})$
replaced by $(\widetilde\eta_I,\widetilde\eta_{p_\pm},\widetilde\eta_{\partial\Omega})$
and the cutoff functions $(\zeta_I,\zeta_{p_\pm},\zeta_{\partial\Omega})$ replaced by
$(\widetilde\zeta_I,\widetilde\zeta_{p_\pm},\widetilde\zeta_{\partial\Omega})$,
respectively, \eqref{eq:partitionOfUnity} also follows in terms
of~$\widetilde\eta_{\mathrm{bulk}}$.

\textit{Step 3: Additional boundary constraints.}
The identities~\eqref{eq:tildePlateau} and~\eqref{eq:tangentialBump}
are straightforward consequences of the 
definitions~\eqref{eq:auxCutoffXi} and~\eqref{eq:auxCutoffVel},
the definitions~\eqref{eq:auxBoundaryCutoff} and~\eqref{eq:auxContactPointsCutoff},
and the definitions~\eqref{def:boundaryCutoff} and~\eqref{def:contactPointCutoff}, respectively.

\textit{Step 4: Coercivity estimates.} We first note that
by the properties of the wedges from Lemma~\ref{th_local_radius_contact}
and the choice~\eqref{eq:2ndMinLocScale} that there exists a constant $C \geq 1$ such that
for all $t \in [0,T]$ it holds
\begin{align}
\label{eq:auxCompDistances1}
1 & \leq C\min\{\dist(\cdot,I(t)),\dist(\cdot,\partial\Omega)\}
\text{ on the domain of } \eqref{eq:bulkCutoffBulk},
\\ \label{eq:auxCompDistances2}
\dist(\cdot,I(t)) &\leq C\dist(\cdot,\partial\Omega)
\text{ on the domains of } \eqref{eq:bulkCutoffInterfaceBulk},\,
\eqref{eq:bulkCutoffContactPointInterfaceWedge},\,
\eqref{eq:bulkCutoffContactPointInterpolationWedge},
\\ \label{eq:auxCompDistances3}
\dist(\cdot,\partial\Omega) &\leq C\dist(\cdot,I(t))
\text{ on the domains of } \eqref{eq:bulkCutoffBoundaryBulk},\,
\eqref{eq:bulkCutoffContactPointBoundaryWedge}\,,
\eqref{eq:bulkCutoffContactPointInterpolationWedge},
\\ \label{eq:auxCompDistances4}
\dist(\cdot,p_{\pm}(t)) &\leq C\dist(\cdot,I(t))
\text{ on the domain of } \eqref{eq:bulkCutoffContactPointBoundaryWedge}.
\\ \label{eq:auxCompDistances5}
\dist(\cdot,p_{\pm}(t)) &\leq C\min\{\dist(\cdot,I(t)),\dist(\cdot,\partial\Omega)\}
\text{ on the domain of } \eqref{eq:bulkCutoffContactPointInterpolationWedge}.
\end{align}
Furthermore, by the definitions~\eqref{eq:auxCutoffXi}--\eqref{def:contactPointCutoff}
it follows that there exists $C \geq 1$ such that for all $t \in [0,T]$ it holds
\begin{align}
\label{eq:auxLowerBoundCutoffs1}
C^{-1}\dist^2(\cdot,I(t)) &\leq |1 {-} \zeta_I(\cdot,t)|
\text{ on the domains of } \eqref{eq:bulkCutoffInterfaceBulk},\,
\eqref{eq:bulkCutoffContactPointInterfaceWedge},\,
\eqref{eq:bulkCutoffContactPointInterpolationWedge},
\\ \label{eq:auxLowerBoundCutoffs2}
C^{-1}\dist^2(\cdot,\partial\Omega) &\leq |1 {-} \zeta_{\partial\Omega}(\cdot,t)|
\text{ on the domains of } \eqref{eq:bulkCutoffBoundaryBulk},\,
\eqref{eq:bulkCutoffContactPointBoundaryWedge}\,,
\eqref{eq:bulkCutoffContactPointInterpolationWedge},
\\ \label{eq:auxLowerBoundCutoffs3}
C^{-1}\dist^2(\cdot,p_{\pm}(t)) &\leq |1 {-} \zeta_{p_\pm}(\cdot,t)|
\text{ on the domain of } \eqref{eq:bulkCutoffContactPointBoundaryWedge}.
\end{align}
The combination of the 
identities~\eqref{eq:bulkCutoffBulk}--\eqref{eq:bulkCutoffContactPointInterpolationWedge}
from the previous step with the estimates~\eqref{eq:auxCompDistances1}--\eqref{eq:auxLowerBoundCutoffs3}
from the current step and the definition~\eqref{def:contactPointCutoff}
therefore implies the coercivity estimates~\eqref{eq:lowerBoundBulkCutoff} and~\eqref{eq:tangentialSquaredCutoff}.

For a verification of the upper 
bounds~\eqref{eq:upperBoundsCutoffs}--\eqref{eq:upperBoundGradientInterfaceCutoffBoundary},
we first remark that as a straightforward consequence of the 
definitions~\eqref{eq:auxCutoffXi}--\eqref{def:contactPointCutoff}
there exists $C \geq 1$ such that for all $t \in [0,T]$ we have
\begin{align}
\label{eq:auxUpperBoundCutoffs1}
|1{-}\zeta_I(\cdot,t)| &\leq C\dist^2(\cdot,I(t))
\text{ on the domains of } \eqref{eq:bulkCutoffInterfaceBulk},\,
\eqref{eq:bulkCutoffContactPointInterfaceWedge},\,
\eqref{eq:bulkCutoffContactPointInterpolationWedge},
\\ \label{eq:auxUpperBoundCutoffs2}
|(\nabla,\partial_t)\zeta_I(\cdot,t)| &\leq C\dist(\cdot,I(t))
\text{ on the domains of } \eqref{eq:bulkCutoffInterfaceBulk},\,
\eqref{eq:bulkCutoffContactPointInterfaceWedge},\,
\eqref{eq:bulkCutoffContactPointInterpolationWedge},
\\ \label{eq:auxUpperBoundCutoffs3}
|1{-}\zeta_{\partial\Omega}(\cdot,t)| &\leq C\dist^2(\cdot,\partial\Omega)
\text{ on the domains of } 
\eqref{eq:bulkCutoffBoundaryBulk}\text{--}\eqref{eq:bulkCutoffContactPointInterpolationWedge},
\\ \label{eq:auxUpperBoundCutoffs4}
|(\nabla,\partial_t)\zeta_{\partial\Omega}(\cdot,t)| &\leq C\dist(\cdot,\partial\Omega)
\text{ on the domains of } 
\eqref{eq:bulkCutoffBoundaryBulk}\text{--}\eqref{eq:bulkCutoffContactPointInterpolationWedge},
\\ \label{eq:auxUpperBoundCutoffs5}
|(\zeta_{I} {-} \zeta_{\partial\Omega})(\cdot,t)| &\leq C\dist^2(\cdot,p_{\pm}(t))
\text { on the domain of } \eqref{eq:bulkCutoffContactPointInterpolationWedge},
\\ \label{eq:auxUpperBoundCutoffs6}
|1 {-} \zeta_{p_{\pm}}(\cdot,t)| & \leq C\dist^2(\cdot,p_{\pm}(t))
\text{ on the domains of } \eqref{eq:bulkCutoffContactPointBoundaryWedge},\,
\eqref{eq:bulkCutoffContactPointInterpolationWedge}, 
\\ \label{eq:auxUpperBoundCutoffs7}
|(\nabla,\partial_t)\zeta_{p_{\pm}}(\cdot,t)| & \leq C\dist(\cdot,p_{\pm}(t))
\text{ on the domains of } \eqref{eq:bulkCutoffContactPointBoundaryWedge},\,
\eqref{eq:bulkCutoffContactPointInterpolationWedge}.
\end{align}
The upper bounds~\eqref{eq:upperBoundsCutoffs}--\eqref{eq:upperBoundsGradientCutoffs}
with respect to~$\eta_{\mathrm{bulk}}$ thus follow from the 
estimates~\eqref{eq:auxUpperBoundCutoffs1}--\eqref{eq:auxUpperBoundCutoffs5}, the 
estimates~\eqref{eq:auxCompDistances1} and~\eqref{eq:auxCompDistances3},
the estimate~\eqref{eq_calib_intpol_blowup1}, as well as the 
identities~\eqref{eq:bulkCutoffBulk}--\eqref{eq:bulkCutoffContactPointInterpolationWedge}.
The upper bounds~\eqref{eq:upperBoundsCutoffs}--\eqref{eq:upperBoundsGradientCutoffs}
with respect to~$\eta_{\mathrm{\partial\Omega}}$ in turn are implied by the
estimates~\eqref{eq:auxUpperBoundCutoffs4} 
and~\eqref{eq:auxUpperBoundCutoffs6}--\eqref{eq:auxUpperBoundCutoffs7},
the estimates~\eqref{eq:auxCompDistances3}--\eqref{eq:auxCompDistances5},
the estimate~\eqref{eq_calib_intpol_blowup1}, as well as
the definition~\eqref{def:boundaryCutoff}. We also obtain the desired upper 
bound~\eqref{eq:upperBoundGradientInterfaceCutoffBoundary}
as a consequence of the estimates~\eqref{eq:auxUpperBoundCutoffs2}
and~\eqref{eq:auxUpperBoundCutoffs4},
the estimate~\eqref{eq:auxCompDistances2},
the estimate~\eqref{eq_calib_intpol_blowup1}, as well as
the definition~\eqref{def:interfaceCutoff}.

Finally, we remark that the upper
bounds~\eqref{eq:upperBoundsCutoffs}--\eqref{eq:upperBoundGradientInterfaceCutoffBoundary}
in terms of $(\widetilde\eta_{\mathrm{bulk}},\widetilde\eta_{I},\widetilde\eta_{\partial\Omega})$
follow analogously.

\textit{Step 5: Motion laws.} We claim that there exists $C>0$ such that it holds
\begin{align}
\label{eq:auxTimeEvolCutoffs1}
|(\partial_t + B\cdot\nabla)\zeta_I| &\leq C\dist^2(\cdot,I) 
&& \text{on } \overline{\mathrm{im}(X_I)} \cap (\overline{\Omega} {\times} [0,T]),
\\ \label{eq:auxTimeEvolCutoffs2}
|(\partial_t + B\cdot\nabla)\zeta_{\partial\Omega}| &\leq C\dist^2(\cdot,\partial\Omega)
&& \text{on } (\overline{\mathrm{im}(X_{\partial\Omega})} {\times} [0,T])
\cap (\overline{\Omega} {\times} [0,T]),
\\ \label{eq:auxTimeEvolCutoffs3}
|(\partial_t + B\cdot\nabla)\zeta_{p_\pm}| &\leq C\dist^2(\cdot,p_\pm)
&& \text{on } \bigcup_{t \in [0,T]} B_{\bar r}(p_{\pm}(t)) {\times} \{t\},
\\ \label{eq:auxTimeEvolCutoffs4}
|(\partial_t + B\cdot\nabla)\lambda^{p_\pm}_{\pm}| &\leq C
&& \text{on } \bigcup_{t \in [0,T]} \big(B_{\bar r}(p_{\pm}(t)) 
\cap W^{p_{\pm}}_{\pm}(t)\big) {\times} \{t\}.
\end{align}
Once these estimates are proved, one may argue along the lines
of~\cite[Proof of Lemma~40]{Fischer2020a} to establish~\eqref{eq:transportBulkCutoff}
and~\eqref{eq:transportBoundaryCutoff}. Indeed, apart
from~\eqref{eq:auxTimeEvolCutoffs1}--\eqref{eq:auxTimeEvolCutoffs4}
the structure of the argument of~\cite[Proof of Lemma~40]{Fischer2020a} only relies 
on the already established ingredients from Step~2 and Step~4 of this proof,
the localization properties~\eqref{eq:supportBulkInterfaceCutoff}--\eqref{eq:overlapBoundaryBulkInterface},
the structure of the definitions~\eqref{def:globalXi}--\eqref{def:globalVel}
and~\eqref{def:interfaceCutoff}--\eqref{def:contactPointCutoff}.

The estimate~\eqref{eq:auxTimeEvolCutoffs3} is an easy consequence of
$B(p_\pm(t),t)=\frac{\mathrm{d}}{\mathrm{d}t}p_{\pm}(t)$, the chain rule
in form of $(\partial_t + \frac{\mathrm{d}}{\mathrm{d}t}p_{\pm}(t) \cdot\nabla)\zeta_{p_{\pm}}=0$,
the estimate~\eqref{eq:auxUpperBoundCutoffs7}, and
finally the estimate $|B - B(p_\pm(t),t)| \leq C\dist(\cdot,p_\pm(t))$.
The bound~\eqref{eq:auxTimeEvolCutoffs4} follows similarly
thanks to the estimates~\eqref{eq_calib_intpol_blowup1} and~\eqref{eq_calib_intpol_advect}.
Furthermore, one derives~\eqref{eq:auxTimeEvolCutoffs2} 
by means of the definition~\eqref{eq:auxBoundaryCutoff}, the 
estimates~\eqref{eq:auxUpperBoundCutoffs4} and $|B(x,t) - B(P^{\partial\Omega}(x),t)|
\leq C\dist(x,\partial\Omega)$, and finally the fact that $\partial_t\zeta_{\partial\Omega}=0$
as well as $(B(P^{\partial\Omega}(x),t) \cdot\nabla)\zeta_{\partial\Omega}(x,t) = 0$.
The latter more precisely follows from $\nabla s_{\partial\Omega}=\mathrm{n}_{\partial\Omega}$
in form of $\nabla\zeta_{\partial\Omega}\cdot\tau_{\partial\Omega}=0$ and the boundary condition
$B|_{\partial\Omega} \cdot \mathrm{n}_{\partial\Omega} = 0$ (cf.\ Proof of 
Proposition~\ref{prop:gluingConstruction}, Step~2: Proof of~\eqref{eq_boundaryCondVelocity}).
It remains to establish the estimate~\eqref{eq:auxTimeEvolCutoffs1}.
This in turn follows from~\eqref{eq:bulkInterface1}, $B|_I = \eta_I B^I + \eta_{p_+} B^{p_+}
+ \eta_{p_-} B^{p_-}$ due to~\eqref{eq:supportBoundaryCutoff} and~\eqref{def:globalVel},
as well as the estimates~\eqref{eq:auxUpperBoundCutoffs2}, \eqref{eq:upperBoundsCutoffs},
and~\eqref{eq:localComp3} resulting in $|(B {-} B_I) \cdot \nabla\zeta_I| \leq C\dist^2(\cdot,I)$.
\end{proof}

\subsection{Construction of the transported weight $\vartheta$}\label{sec_calib_theta}
The last missing ingredient for the proof of Theorem~\ref{theo_existenceBoundaryAdaptedCalibration}
consists of the following result.

\begin{lemma}
\label{lem:existenceTransportedWeight}
Let the setting as described at the beginning of Subsection~\ref{sec_calib_global}
be in place. For a given set of admissible localization functions in the
sense of Definition~\ref{def:admissibleLocFunctions}, let~$B$ denote
the associated velocity field defined by~\eqref{def:globalVel}. There
then exists a map $\vartheta\colon\overline{\Omega}{\times}[0,T]\to[-1,1]$
which satisfies the corresponding requirements~\eqref{eq_regularityWeight}
and~\eqref{eq_weightNegativeInterior}--\eqref{eq_weightEvol} 
of a boundary adapted gradient flow calibration.
\end{lemma}

\begin{proof}[Proof of Theorem~\ref{theo_existenceBoundaryAdaptedCalibration}]
This now follows immediately from Proposition~\ref{prop:gluingConstruction},
Proposition~\ref{prop:existenceAdmissibleLocFunctions} and Lemma~\ref{lem:existenceTransportedWeight}.
Recall also in this context that the supplemental 
conditions~\eqref{eq:GradVelXiXi}--\eqref{eq:skewSymmetryGradVelInt}
are taken care of by Proposition~\ref{prop:gluingConstruction}.
\end{proof}

\begin{proof}[Proof of Lemma~\ref{lem:existenceTransportedWeight}]
We first provide a construction of the transported weight~$\vartheta$.
In a second step,  we establish the desired properties.

Let us start by fixing some useful notation. For the two localization
scales~$\bar r$ and~$\bar \delta$ defined by~\eqref{eq:choiceMinLocScale}
and~\eqref{eq:2ndMinLocScale}, respectively, define an associated neighborhood
of the network $I \cup (\partial\Omega {\times} [0,T])$
by means of 
\begin{align}
\mathscr{U}_{\bar r,\bar \delta}(t)
:= \overline{\mathrm{im}(X_{I}^{\bar r,\bar\delta}(\cdot,t,\cdot))}
\cup \overline{\mathrm{im}(X_{\partial\Omega}^{\bar r,\bar\delta})}
\cup \bigcup_{p \in \{p_\pm\}} B_{\bar r}(p(t)),
\quad t \in [0,T].
\end{align}
For each $t \in [0,T]$, we also introduce for convenience
the notation $\mathscr{A}_+(t) := \mathscr{A}(t)$ 
and $\mathscr{A}_-(t) := \Omega\setminus\overline{\mathscr{A}(t)}$.

Choose next a (up to the sign) smooth truncation of the identity $\bar\vartheta\colon\Rd[]\to[-1,1]$
in the sense that $\bar\vartheta(s) = -s$ for $s \in [-1/2,1/2]$, $\bar\vartheta'(s)<0$ for $s \in (-1,1)$,
$\bar\vartheta(s) = 1$ for $s \leq -1$ and $\bar\vartheta(s) = -1$ for $s \geq 1$.
For a given $\delta \in (0,\bar\delta]$ which we fix later, we next define two auxiliary maps
\begin{align}
\label{def:auxWeightInterface}
\vartheta_I(x,t) &:= \bar\vartheta\Big(\frac{s_I(x,t)}{\delta\bar r}\Big),
&& (x,t) \in \overline{\mathrm{im}(X_I)},
\\ \label{def:auxWeightBoundary}
\vartheta_{\partial\Omega}(x,t) &:= \bar\vartheta\Big(\frac{s_{\partial\Omega}(x)}{\delta\bar r}\Big),
&& (x,t) \in \overline{\mathrm{im}(X_{\partial\Omega})} {\times} [0,T].
\end{align}
We have everything set up to proceed with an adequate definition of the weight~$\vartheta$.
Fix $t \in [0,T]$. Away from the two contact points, we set
\begin{align}
\label{def:weightBulk1}
\vartheta(\cdot,t) := 
\begin{cases}
\pm 1 & \text{in } \mathscr{A}_{\pm}(t) \setminus \mathscr{U}_{\bar r,\bar \delta}(t),
\\
\pm \vartheta_{\partial\Omega}(\cdot,t) & \text{in } 
\overline{\mathrm{im}(X_{\partial\Omega}^{\bar r,\bar\delta})} 
\setminus \bigcup_{p \in \{p_\pm\}} B_{\bar r}(p(t)),
\\
\vartheta_I(\cdot,t) & \text{in }
\overline{\mathrm{im}(X_{I}^{\bar r,\bar\delta}(\cdot,t,\cdot))}
\setminus \bigcup_{p \in \{p_\pm\}} B_{\bar r}(p(t)),
\end{cases}
\end{align}
whereas we define in the vicinity of the two contact points
\begin{align}
\label{def:weightBulk2}
\vartheta(\cdot,t) := 
\begin{cases}
\vartheta_I(\cdot,t) & \text{in }
B_{\bar r}(p_{\pm}(t)) \cap W^{p_\pm}_I(t),
\\
\pm \vartheta_{\partial\Omega}(\cdot,t) & \text{in }
B_{\bar r}(p_{\pm}(t)) \cap W^{p_\pm}_{\partial\Omega}(t),
\\
\big(\lambda_{\pm}^{p_{\pm}}\vartheta_I \pm 
(1 {-} \lambda_{\pm}^{p_{\pm}})\vartheta_{\partial\Omega}\big)(\cdot,t)
& \text{in } B_{\bar r}(p_{\pm}(t)) \cap W^{p_\pm}_{\pm}(t).
\end{cases}
\end{align}

In order to guarantee the required regularity~\eqref{eq_regularityWeight} 
for the piecewise definitions~\eqref{def:weightBulk1}--\eqref{def:weightBulk2},
it simply suffices to choose~$\delta \in (0,\bar{\delta}]$
small enough and to recall the regularity assertions
from Remark~\ref{th_strongsol_notation_tub}, Remark~\ref{th_strongsol_notation_bdry}
and Lemma~\ref{th_calib_intpol}. The desired sign 
conditions~\eqref{eq_weightNegativeInterior}--\eqref{eq_weightZeroInterface}
also follow immediately from an inspection of the definitions~\eqref{def:weightBulk1}--\eqref{def:weightBulk2}.

For a proof of the coercivity estimate~\eqref{eq_weightCoercivity}, note that
by the properties of~$\bar\vartheta$ and the 
definitions~\eqref{def:auxWeightInterface}--\eqref{def:auxWeightBoundary}
there exists $C > 0$ such that for all $t \in [0,T]$
\begin{align}
\label{eq:lowerBoundAuxWeight1}
\dist(\cdot,I(t)) &\leq C|\vartheta_I(\cdot,t)|
&& \text{on } \overline{\mathrm{im}(X_I)(\cdot,t,\cdot)},
\\ \label{eq:lowerBoundAuxWeight2}
\dist(\cdot,\partial\Omega) &\leq C|\vartheta_{\partial\Omega}(\cdot,t)|
&& \text{on } \overline{\mathrm{im}(X_{\partial\Omega})}.
\end{align}
In view of the estimates~\eqref{eq:auxCompDistances1}--\eqref{eq:auxCompDistances5},
the required bound~\eqref{eq_weightCoercivity} therefore holds
as a consequence of the definitions~\eqref{def:weightBulk1}--\eqref{def:weightBulk2}.

For a proof of the estimate~\eqref{eq_weightEvol}, we first claim
that there exists $C > 0$ such that for all $t \in [0,T]$ it holds
\begin{align}
\label{eq:auxTimeEvolWeight1}
|(\partial_t + B\cdot\nabla)\vartheta_I|(\cdot,t) 
&\leq C\dist(\cdot,I(t)) 
&& \text{on } \overline{\mathrm{im}(X_I)(\cdot,t,\cdot)} \cap \overline{\Omega},
\\ \label{eq:auxTimeEvolWeight2}
|(\partial_t + B\cdot\nabla)\vartheta_{\partial\Omega}|(\cdot,t) 
&\leq C\dist(\cdot,\partial\Omega)
&& \text{on } \overline{\mathrm{im}(X_{\partial\Omega})} \cap \overline{\Omega},
\\ \label{eq:auxEstimateWeights}
|\vartheta_I - \vartheta_{\partial\Omega}|(\cdot,t) 
&\leq C\dist(\cdot,p_{\pm}(t))
&& \text{on } B_{\bar r}(p_{\pm}(t)) \cap W^{p_{\pm}}_{\pm}(t).
\end{align}
Since~\eqref{eq:auxEstimateWeights} is obvious, let us concentrate on
the proof of~\eqref{eq:auxTimeEvolWeight1} and~\eqref{eq:auxTimeEvolWeight2}.	
These two, however, can be derived along the lines of the argument
in favor of the two estimates~\eqref{eq:auxTimeEvolCutoffs1}
and~\eqref{eq:auxTimeEvolCutoffs2}, respectively.
Being equipped with the auxiliary estimates~\eqref{eq:auxTimeEvolWeight1}--\eqref{eq:auxEstimateWeights},
the desired bound~\eqref{eq_weightEvol} now follows from making use of
the definitions~\eqref{def:weightBulk1}--\eqref{def:weightBulk2},
the estimates~\eqref{eq:auxCompDistances1}--\eqref{eq:auxCompDistances5},
the estimate~\eqref{eq:auxTimeEvolCutoffs4}, and the already
established bound~\eqref{eq_weightCoercivity}.
\end{proof}

\appendix
\section{Weak solutions to the Allen--Cahn problem~\eqref{eq_AllenCahn}--\eqref{eq_initialData}}
\label{sec_appA}

\begin{proof}[Proof of Lemma~\ref{lem_existenceWeakSolutionAllenCahn}]

\textit{Step 1: Implicit time discretization.} Let $T>0$ be fixed, $N\in\N$ and $\tau=\tau(N):=\frac{T}{N}$. We define $ u_N^0:= u_{\eps,0}\in H^1(\Omega)$ and construct inductively for $k=1,...,N$: if $ u_N^{k-1}\in H^1(\Omega)$ is known, then let $ u_N^k$ be a minimizer of
\begin{align}\label{eq_min_energy}
E_k:H^1(\Omega)\rightarrow[0,\infty]: u\mapsto E_\varepsilon[ u]
+\frac{1}{2\tau}\| u- u_N^{k-1}\|_{L^2(\Omega)}^2.
\end{align}
Clearly, $E_k$ is non-trivial due to the assumptions on $W,\sigma$. The existence of a minimizer can be shown via the direct method, cf.~Step 2 below. Due to \eqref{eq_double_well2} it follows that $ u_N^k\in H^1(\Omega)\cap L^p(\Omega)$.

Due to the assumptions on $W$ and $\sigma$ one can proceed similar to Garcke \cite{Garcke2003}, Lemma 3.5, to obtain the associated Euler-Lagrange equation: for all test functions $\xi\in H^1(\Omega)\cap L^\infty(\Omega)$ it holds
\[
\varepsilon \int_\Omega \nabla u_N^k \cdot\nabla\xi 
+\int_\Omega \frac{ u_N^k- u_N^{k-1}}{\tau} \xi 
+\int_\Omega \frac{1}{\varepsilon}W'( u_N^k) \xi
+ \int_{\partial \Omega} \sigma'(\textup{tr}  u_N^k)\textup{tr}\xi\,d\Hc^{d-1}=0.
\]

We consider the piecewise constant extension $ u_N(t):= u_N^k$ on $((k-1)\tau,k\tau]$ for $k=0,...,N$ and the piecewise linear extension $\overline{ u}_N(t):=\lambda  u_N^{k-1}+ (1-\lambda)  u_N^k$ for $t=\lambda(k-1)\tau+(1-\lambda)k\tau$, where $\lambda\in[0,1], k=0,...,N$. 

\textit{Step 2: Existence of minimizers.} To this end, we consider a minimizing sequence $( u_n)_{n\in\N}$ for $E_k$ in $H^1(\Omega)$. The functional $E_k$ is coercive. More precisely, it holds
\[
E_k( u)\geq 
\frac{\varepsilon}{2}\|\nabla u\|_{L^2(\Omega)}^2+\frac{1}{4\tau} \| u\|_{L^2(\Omega)}^2
-C(\| u_N^k\|_{L^2(\Omega)}),
\]
where we used $W,\sigma\geq 0$ and Young's inequality. Hence $( u_n)_{n\in\N}$ is a bounded sequence in $H^1(\Omega)$ and there is a weakly convergent subsequence (for simplicity denoted with the same index) $ u_n \rightharpoonup\widetilde{ u}$ for $n\rightarrow\infty$ in $H^1(\Omega)$ for some $\widetilde{u}\in H^1(\Omega)$. The terms in $E_k$ without the $W$ and $\sigma$-contributions are convex and continuous, hence also weakly lower semi-continuous. Furthermore, because of the compact embedding $H^1(\Omega)\hookrightarrow\hookrightarrow L^2(\Omega)$ as well as the compactness of the trace operator $\textup{tr}:H^1(\Omega)\rightarrow L^2(\partial \Omega)$, we obtain after sub-sequence extractions that for $n\rightarrow\infty$
\begin{alignat*}{2}
& u_n\rightarrow\widetilde{ u}\text{ in }L^2(\Omega), 
&\quad & u_n\rightarrow\widetilde{ u}\text{ a.e. in }\Omega,\\
&\textup{tr} u_n\rightarrow\textup{tr}\widetilde{ u}\text{ in }L^2(\partial \Omega), 
&\quad& \textup{tr} u_n\rightarrow\textup{tr}\widetilde{ u}\text{ a.e. in }\partial \Omega.
\end{alignat*}
Finally, the Fatou Lemma yields that $\widetilde{ u}$ is a minimizer of $E_k$.

\textit{Step 3: Uniform Estimates.} Inserting $ u_N^k$ and $ u_N^{k-1}$ in $E_k$ from \eqref{eq_min_energy} yields 
\[
E_\varepsilon[ u_N^k]+\frac{1}{2\tau} \| u_N^k- u_N^{k-1}\|_{L^2(\Omega)}^2 
\leq E_\varepsilon[ u_N^{k-1}].
\]
Because of $ u_N^{k-1},  u_N^k\in H^1(\Omega)\cap L^p(\Omega)$ the terms stemming from $E_\varepsilon$ are finite. Therefore one can apply a telescope sum argument which implies
\begin{align*}
\int_\Omega \frac{\varepsilon}{2}|\nabla u_N^k|^2+\int_\Omega \frac{1}{\varepsilon}W( u_N^k)+\int_{\partial \Omega} \sigma( u_N^k)\,\Hc^{d-1}+\sum_{l=1}^k\frac{1}{2\tau}\| u_N^l- u_N^{l-1}\|_{L^2(\Omega)}^2\\
\leq \int_\Omega \frac{\varepsilon}{2} |\nabla  u_{\eps,0}|^2
+\int_\Omega \frac{1}{\varepsilon}W( u_{\eps,0}) 
+\int_{\partial \Omega} \sigma(\textup{tr}  u_{\eps,0})\,d\Hc^{d-1}.
\end{align*}
Therefore \eqref{eq_double_well1} yields that the $ u_N^k$ are uniformly bounded in $H^1(\Omega)\cap L^p(\Omega)$ independently of $k=0,...,N$ and $N\in\N$. Hence it follows that
\begin{align*}
&( u_N)_{N\in\N}\text{ is bounded in }L^\infty(0,T;H^1(\Omega)\cap L^p(\Omega)),\\
&(\overline{ u}_N)_{N\in\N}\text{ is bounded in }H^1(0,T,L^2(\Omega))\cap L^\infty(0,T;H^1(\Omega)\cap L^p(\Omega)).
\end{align*}

\textit{Step 4: Convergence.} There is a sub-sequence (not re-labelled) and a $\hat{ u}$ such that 
\[
 u_N\rightharpoonup^\ast \hat{ u} \quad\text{ in } L^\infty(0,T;H^1(\Omega)\cap L^p(\Omega)).
\]
Moreover, due to the Aubin-Lions-Lemma, cf.~Simon \cite{Simon1986}, Corollary 5 , the space $L^\infty(0,T,H^1(\Omega))\cap H^1(0,T,L^2(\Omega))$ is compactly embedded into $C([0,T],H^s(\Omega))$ for all $s\in[0,1)$, where we choose a fixed $s\in(\frac{1}{2},1)$. Therefore up to a sub-sequence for some $\overline{ u}$ it holds
\[
\overline{ u}_N\rightarrow\overline{ u}\quad\text{ in }C([0,T],H^s(\Omega)).
\]  
With the estimate
\[
\|\overline{ u}_N(t)- u_N(t)\|_{L^2(\Omega)}\leq\| u_N^k- u_N^{k-1}\|_{L^2(\Omega)}\leq C\sqrt{\tau}\quad\text{ for all }t\in[(k-1)\tau,k\tau]
\]
for some $C>0$ independent of $k,N$ and using $\tau=\frac{T}{N}$ we obtain
\[
 u_N \rightarrow \overline{ u}\quad\text{ in }L^\infty(0,T,L^2(\Omega))
\]
and $\hat{ u}=\overline{ u}$. Using $\overline{\phi}\in L^\infty(0,T,H^1(\Omega))$ and an interpolation estimate we get
\[
 u_N \rightarrow \overline{ u}\quad\text{ in }L^\infty(0,T,H^s(\Omega)). 
\]
Moreover, it holds $\overline{ u}\in C^{\frac{1}{2}}([0,T],L^2(\Omega))$ since $(\overline{ u}_N)_{N\in\N}$ is bounded in this space and an interpolation estimate yields $\overline{ u}_N\rightarrow 
\overline{ u}$ in $C^\alpha([0,T],L^2(\Omega))$ for all $\alpha\in(0,\frac{1}{2})$.
Furthermore, it holds
\[
\overline{ u}_N\rightharpoonup\overline{ u}\quad\text{ in }L^2(0,T,H^1(\Omega))\cap H^1(0,T,L^2(\Omega)),
\]
where the weak limit equals $\overline{ u}$ due to the compactness into $L^2(0,T,L^2(\Omega))$. Due to all these convergence properties and the continuity of the trace operator from $H^s(\Omega)$ to $L^2(\partial \Omega)$, we obtain after sub-sequence extraction that 
\[
 u_N,\overline{ u}_N\rightarrow\overline{ u}\quad\text{ a.e. in }\Omega\times(0,T),\quad \textup{tr} u_N, \textup{tr}\overline{ u}_N \rightarrow \textup{tr}\overline{ u}\quad\text{ a.e. in }\partial \Omega\times(0,T).
\]

\textit{Step 5: Weak formulation.} Using the above convergence properties one can pass to the limit in the Euler-Lagrange equation. This yields \eqref{eq_evolEquationWeakSolution}.

\textit{Step 6: Uniqueness and bound in Lemma \ref{lem_existenceWeakSolutionAllenCahn}.} Using a Gronwall-argument and the splitting \eqref{eq_double_well3} of $W$, one can prove uniqueness of weak solutions. Now assume that the initial phase field additionally satisfies $ u_{\varepsilon,0}\in[-1,1]$ a.e.~in $\Omega$. Then in the above construction of a weak solution via the implicit time discretization one can choose the minimizers $ u_N^k$ in such a way that $ u_N^k\in[-1,1]$ a.e.~in $\Omega$ for $k=1,...,N$, $N\in\N$. This follows via mathematical induction over $k$ since the energy $E_k( u)$ is non-increasing when truncating the values of $ u$ at $[-1,1]$ provided that this holds for $ u_N^{k-1}$. Then the obtained weak solution also has the desired property. 
\end{proof}

\begin{proof}[Proof of Lemma~\ref{lem_higherRegularityWeakSolutionAllenCahn}]
We split the proof into three steps. In principle, all of these steps
are based on standard arguments. However, due to the nonlinear Robin
boundary condition~\eqref{eq_RobinBoundaryCondition}, we decided to 
present some level of detail.

\textit{Step 1: Proof of the properties~\eqref{eq_weakSolutionsAreStrongSolutions}
and~\eqref{eq_weakFormRobinCondition}.}
Since the initial phase field satisfies 
$ u_{\eps,0} \in [-1,1]$ almost everywhere in~$\Omega$,
it follows from~\eqref{eq_comparisonPrincipleWeakSolution} in Lemma \ref{lem_existenceWeakSolutionAllenCahn} and the boundedness of~$W'$ on~$[-1,1]$ that
$W'( u_\eps)\in L^\infty(\Omega{\times}(0,T))$.
Testing~\eqref{eq_evolEquationWeakSolution} with
test functions which are compactly supported in~$\Omega{\times}(0,T)$
thus entails together with the regularity in time~\eqref{eq_regularityWeakSolution}
of~$ u_\eps$ that~$\Delta u_\eps$, as a distribution on~$\Omega{\times}(0,T)$,
is represented by an $L^2$-function on~$\Omega{\times}(0,T)$, namely 
$\partial_t u_\eps + \frac{1}{\eps^2}W'( u_\eps)$,
which in turn proves~\eqref{eq_weakSolutionsAreStrongSolutions}. Then \eqref{eq_weakFormRobinCondition} directly follows by testing~\eqref{eq_evolEquationWeakSolution} with
$\zeta\in C^\infty_{\mathrm{cpt}}((0,T);C^\infty(\overline{\Omega}))$.

\textit{Step 2: Proof of $\nabla\partial_t u_\eps\in L^2_{\mathrm{loc}}(0,T;L^2(\Omega))$.}
Let $0<s<t<T$, and let $\eta\in C^\infty_{\mathrm{cpt}}((0,T);[0,1])$ be
such that $\eta|_{[s,t]}\equiv 1$. Denote with $D^h_t f$ the difference quotient in the time variable for $h>0$ and some function $f$. We test~\eqref{eq_evolEquationWeakSolution} with $D^{-h}_t(\eta D^h_t u_\eps)$ for $|h|\ll_{s,t} 1$, which is an admissible test function after approximation. Then by approaching the characteristic function $\chi_{[s,t]}$ with~$\eta$ we obtain
\begin{align*}
&\int_{s}^{t} \int_{\Omega} |D^h_t\nabla u_\eps|^2 \dx \dt
+ \int_{s}^{t} \int_{\Omega} \partial_t\frac{1}{2}|D^h_t u_\eps|^2 \dx \dt 
\\&
= -\int_{s}^{t} \int_{\Omega} \frac{1}{\eps^2}D^h_t\big(W'( u_\eps)\big) D^h_t u_\eps \dx \dt
-\int_{s}^{t} \int_{\partial \Omega} \frac{1}{\eps}D^h_t\big(\sigma'( u_\eps)\big) D^h_t u_\eps \dH \dt
\end{align*}
for all $|h|\ll_{s,t} 1$. By a Lipschitz estimate and standard Sobolev theory for
difference quotients we have
\begin{align*}
\bigg|\int_{s}^{t} \int_{\Omega} \frac{1}{\eps^2}D^h_t\big(W'( u_\eps)\big) D^h_t u_\eps \dx \dt\bigg|
\leq C(\eps,\|W''\|_{L^\infty([-1,1])}) \int_0^T\int_{\Omega} |\partial_t u_\eps|^2 \dx \dt
\end{align*} 
for all $|h|\ll_{s,t} 1$. For an estimate of the boundary integral, we argue as follows.
Since the initial phase field satisfies 
$ u_{\eps,0} \in [-1,1]$ almost everywhere in~$\Omega$,
it follows from~\eqref{eq_comparisonPrincipleWeakSolution}
that we may replace~$\sigma$ by any $C^2$-density $\widetilde\sigma\colon\Rd[]\to\Rd[]$
which coincides with~$\sigma$ on~$[-1,1]$. Fix one such~$\widetilde\sigma$. Then analogous as before we have
$\big|D^h_t\big(\sigma'( u_\eps)\big)\big|
\leq C\|\widetilde\sigma''\|_{L^\infty([-1,1])} \big|D^h_t u_\eps\big|$. Using the trace (interpolation) inequality and Young's inequality as well as standard Sobolev theory for difference quotients we finally obtain the bound
\begin{align*}
\bigg|\int_{s}^{t} \int_{\partial \Omega} \frac{1}{\eps}D^h_t\big(\sigma'( u_\eps)\big) D^h_t u_\eps \dH \dt\bigg|
&\leq C(\delta,\eps,\|\widetilde\sigma''\|_{L^\infty([-1,1])})
\int_0^T\int_{\Omega} |\partial_t u_\eps|^2 \dx \dt
\\&~~~
+\delta \int_{s}^{t} \int_{\Omega} |D^h_t\nabla u_\eps|^2 \dx \dt
\end{align*}
for all~$\delta\in (0,1)$ and all $|h|\ll_{s,t} 1$. Hence, an absorption argument
together with the fundamental theorem of calculus (the latter facilitated by a 
standard mollification argument in the time variable) entails based on the
previous four displays that
\begin{align*}
&\int_{s}^{t} \int_{\Omega} |D^h_t\nabla u_\eps|^2 \dx \dt
\\&
\leq \int_{\Omega}|D^h_t u_\eps(\cdot,s)|^2 \dx
+ C(\eps,\|W''\|_{L^\infty([-1,1])},\|\widetilde\sigma''\|_{L^\infty([-1,1])})
\int_0^T\int_{\Omega} |\partial_t u_\eps|^2 \dx \dt
\end{align*}
for all~$0<s<t<T$ and all $|h|\ll_{s,t} 1$. In particular, since for almost every $s\in (0,T)$
it holds~$\int_{\Omega}|\partial_t u_\eps(\cdot,s)|^2\dx < \infty$, it follows that
for almost every $s\in (0,T)$ and all~$t\in (s,T)$ it holds 
$\int_{s}^{t} \int_{\Omega} |D^h_t\nabla u_\eps|^2 \dx \dt \lesssim 1$ 
uniformly over all $|h|\ll_{s,t} 1$. This in turn implies the claim
by standard Sobolev theory for difference quotients.

\textit{Step 3: Proof of $ u_\eps\in L^2(0,T;H^2(\Omega))$.}
We only provide details for the local estimate for tangential derivatives of $\nabla u_\varepsilon$
at a boundary point $x_0\in\partial \Omega$ after locally flattening the
boundary~$\partial \Omega$ around~$x_0$. With respect to the latter---up 
to a rotation and translation---we may assume that
$x_0=0$ and that there exists a radius $r>0$ as well as a 
$C^2$-map $g \colon B_r(0) \cap \Rd[d{-}1] \to \Rd[]$ such that $g(0)=0$ and
\begin{align*}
\Omega \cap B_r(0) &= \{x=(x',x_d) \in B_r(0) \colon x_d > g(x')\},
\\
\partial \Omega \cap B_r(0) &= \{x=(x',x_d) \in B_r(0) \colon x_d = g(x')\}.
\end{align*}
Defining the map $\Psi\colon B_r(0)\to\R^d:(x',x_d)\mapsto(x',x_d{-}g(x'))$, which is a $C^2$-diffeomorphism onto its image,
and the coefficient field $A:=\nabla\Psi^{-1}(\nabla\Psi^{-1})^\mathsf{T}$, 
we have that $\mathrm{det}\nabla\Psi = 1$ and that the operator $-\nabla(a\nabla\,\cdot)$ is
uniformly elliptic and bounded. Choosing $r'\in (0,1)$ small enough such that
$B_{r'}(0)\subset\subset\mathrm{im}\,\Psi$, we then obtain from~\eqref{eq_evolEquationWeakSolution}
and a change of variables that
\begin{align}
\nonumber
&\int_{0}^{T} \int_{B^+_{r'}(0)} \zeta\partial_t\widetilde u_\eps \dx \dt
+ \int_{0}^{T} \int_{B^+_{r'}(0)} \nabla\zeta\cdot A\nabla\widetilde u_\eps \dx \dt
\\& \label{eq_weakFormFlatBoundary}
= - \int_{0}^{T}\int_{B^+_{r'}(0)} \zeta \frac{1}{\eps^2} W'(\widetilde  u_\eps) \dx \dt
\\&~~~ \nonumber
- \int_{0}^{T} \int_{B_{r'}(0) \cap \{x_d=0\}} \zeta\sqrt{1{+}|\nabla_{x'}g(x')|^2} 
\frac{1}{\eps}(\sigma'\circ\widetilde u_\eps)\big(x',g(x')\big)  \dH \dt
\end{align}
for all~$\zeta\in C^\infty_{\mathrm{cpt}}(B_{r'}(0))$,
where we have also defined $\widetilde u_\eps:= u_\eps\circ\Psi^{-1}$
as well as $B^+_{r'}(0)=B_{r'}(0)\cap\{(x',x_d)\in\Rd \colon x_d > 0\}$.

Let~$\eta\in C^\infty_{\mathrm{cpt}}(\frac{1}{2}B_{r'}(0);[0,1])$. We denote by $D^h_x f$ for $h>0$ 
and some $f$ the difference quotient
in the spatial variables with respect to an arbitrary, but fixed, tangential direction. 
Testing~\eqref{eq_weakFormFlatBoundary}
with the (after approximation) admissible test function 
$D^{-h}_x(\eta^2D^h_x\widetilde u_\eps)$ for $|h| < \frac{1}{2}$ 
 we obtain
together with the fundamental theorem of calculus (which is facilitated by a standard
mollification argument in the time variable) and the uniform ellipticity of~$A$
\begin{align*}
& \int_{0}^{T} \int_{B^+_{r'}(0)} \eta^2 |D^h_x\nabla\widetilde u_\eps|^2 \dx \dt
\lesssim \int_{0}^{T} \int_{B^+_{r'}(0)} \eta^2 D^h_x\nabla\widetilde u_\eps
\cdot A D^h_x\nabla\widetilde u_\eps \dx \dt
\\&
\leq \int_{B^+_{r'}(0)} \eta^2|D^h_x\widetilde u_{\eps}(\cdot,0)|^2 \dx
- \int_{0}^{T} \int_{B^+_{r'}(0)} \eta^2 D^h_x\widetilde u_\eps
\frac{1}{\eps^2}D^h_x\big(W'(\widetilde u_\eps)\big) \dx \dt
\\&~~~
- \int_{0}^{T} \int_{B_{r'}(0) \cap \{x_d=0\}} \eta^2 D^h_x\widetilde u_\eps
\sqrt{1{+}|\nabla_{x'}g(x')|^2} \frac{1}{\eps}D^h_x\big((\sigma'\circ\widetilde u_\eps)(x',g(x'))\big)  \dH \dt
\\&~~~
- \int_{0}^{T} \int_{B^+_{r'}(0)} D^h_x\widetilde u_\eps 2\eta\nabla\eta
\cdot A D^h_x\nabla\widetilde u_\eps \dx \dt
\\&~~~
- \int_{0}^{T} \int_{B^+_{r'}(0)} \eta^2D^h_x\nabla\widetilde u_\eps
\cdot \big\{D^h_x\big(A\nabla\widetilde u_\eps\big) - A D^h_x\nabla\widetilde u_\eps\big\} \dx \dt
\\&~~~
- \int_{0}^{T} \int_{B^+_{r'}(0)} D^h_x\widetilde u_\eps 2\eta\nabla\eta
\cdot\big\{D^h_x\big(A\nabla\widetilde u_\eps\big) - A D^h_x\nabla\widetilde u_\eps\big\} \dx \dt
\\&~~~
- \int_{0}^{T} \int_{B^+_{r'}(0)} \eta^2D^h_x\widetilde u_\eps \cdot
\Big\{D^h_x\Big(\sqrt{1{+}|\nabla_{x'}g(x')|^2} 
\frac{1}{\eps}(\sigma'\circ\widetilde u_\eps)\big(x',g(x')\big)\Big)
\\&~~~~~~~~~~~~~~~~~~~~~~~~~~~~~~~~~~~~~
- \sqrt{1{+}|\nabla_{x'}g(x')|^2} \frac{1}{\eps}D^h_x\big((\sigma'\circ\widetilde u_\eps)(x',g(x'))\big) \Big\} \dH \dt
\end{align*}
for all $\eta\in C^\infty_{\mathrm{cpt}}(\frac{1}{2}B_{r'}(0);[0,1])$ and all $|h| < \frac{1}{2}$.

The terms on the right hand side without the first one can be estimated similarly as in \textit{Step~2} of this proof by
\begin{align*}
\delta \int_{0}^{T} \int_{B^+_{r'}(0)} \eta^2 |D^h_x\nabla\widetilde u_\eps|^2 \dx \dt
+ C(\delta)\int_0^T\int_{\Psi(\Omega\cap B_r(0))} |\widetilde u_\eps|^2 + |\nabla\widetilde u_\eps|^2 \dx \dt
\end{align*}
for all $\delta\in (0,1)$, $\eta\in C^\infty_{\mathrm{cpt}}(\frac{1}{2}B_{r'}(0);[0,1])$ and $|h| < \frac{1}{2}$, where the first three of these six terms can be estimated without the $|\widetilde u_\eps|^2$-term on the right hand side. Altogether, by an absorption argument and by fixing
$\eta\in C^\infty_{\mathrm{cpt}}(\frac{1}{2}B_{r'}(0);[0,1])$ such that
$\eta|_{\frac{1}{4}B_{r'}(0)} \equiv 1$, we obtain
\begin{align*}
&\int_{0}^{T} \int_{\frac{1}{4}B^+_{r'}(0)} |D^h_x\nabla\widetilde u_\eps|^2 \dx \dt
\\&
\leq C \int_{\Omega} | u_\eps(\cdot,0)|^2 + |\nabla u_\eps(\cdot,0)|^2 \dx
+ C \int_0^T\int_{\Omega} | u_\eps|^2 + |\nabla u_\eps|^2 \dx \dt
\end{align*}
uniformly over all $|h| < \frac{1}{2}$. This in turn establishes 
the desired local estimate for tangential derivatives 
at a boundary point $x_0\in\partial \Omega$ after locally flattening the
boundary~$\partial \Omega$ around~$x_0$. From here onwards, one may
proceed by standard arguments to deduce $ u_\eps\in L^2(0,T;H^2(\Omega))$.
\end{proof}

\begin{proof}[Proof of Lemma~\ref{lem_energyDissipationEquality}]
We proceed in two steps.

\textit{Step 1: Proof of~\eqref{eq_energyDissipationWeakSolution} under an additional assumption.}
In this step, we establish~\eqref{eq_energyDissipationWeakSolution}
assuming momentarily that the energy functional~$E_\eps[ u_\eps]$ is continuous on~$[0,T]$.
This fact will then be checked in a second step. Under this additional assumption
it clearly suffices to prove that for all~$0<s<T'<T$ it holds
\begin{align}
\label{eq_auxEnergyDissip}
E_\eps[ u_\eps(\cdot,T')] 
+ \int_{s}^{T'} \int_{\Omega} \eps\big|\partial_t  u_\eps\big|^2 \dx \dt
= E_\eps[ u_\eps(\cdot,s)].
\end{align}
Let $0<s<T'<T$, and let $\eta\in C^\infty_{\mathrm{cpt}}((0,T);[0,1])$
such that $\eta|_{[s,T']}\equiv 1$. Testing~\eqref{eq_weakFormFlatBoundary}
with the thanks to Lemma~\ref{lem_higherRegularityWeakSolutionAllenCahn} 
admissible test function $\eps\eta\partial_t u_\eps$
and by approaching the characteristic function $\chi_{[s,T']}$ with $\eta$ shows
\begin{align*}
&\int_{s}^{T'} \int_{\Omega} \eps\nabla u_\eps\cdot\partial_t\nabla u_\eps \dx \dt
+ \int_{s}^{T'} \int_{\Omega} \frac{1}{\eps} W'( u_\eps)\partial_t u_\eps \dx \dt
\\&
+ \int_{s}^{T'} \int_{\partial \Omega} \sigma'( u_\eps)\partial_t u_\eps \dH \dt
= -\int_{s}^{T'} \int_{\Omega} \eps|\partial_t u_\eps|^2 \dx \dt.
\end{align*}
By a standard mollification argument, the chain rule, and the fundamental theorem of calculus,
we thus obtain from the previous display the desired identity~\eqref{eq_auxEnergyDissip}.

\textit{Step 2: Proof of~$E_\eps[ u_\eps]\in C([0,T])$.}
Recalling that $ u_\eps\in C([0,T];L^2(\Omega))$,
it suffices to prove that the Dirichlet energy is continuous on~$[0,T]$.
Indeed, continuity for the other two energy contributions then
follows from the trace (interpolation) inequality and a Lipschitz estimate. Hence note that $ u_\eps\in H^1(0,T;L^2(\Omega))\cap L^2(0,T;H^2(\Omega))$ due to~\eqref{eq_regularityWeakSolution} and~\eqref{eq_higherRegularityPhaseField}. Interpolation yields
$ u_\eps \in C([0,T];H^1(\Omega))$ which concludes the claim.
\end{proof}

\section{Construction of well-prepared initial data}
\label{sec_appB}

\begin{proof}[Proof of Lemma~\ref{lem_wellPreparedInitialData}]
We split the proof into four steps. 

\textit{Step 1: Construction of auxiliary signed distance to initial bulk interface.}
As the $C^2$-interface $\overline{\partial^*\mathscr{A}(0) \cap \Omega}$ intersects
the $C^{2}$-domain boundary~$\partial \Omega$ non-tangentially
at two distinct points~$c_{\pm}(0)\in\partial \Omega$, we may choose
two localization scales~$r,\delta\in (0,1)$ being sufficiently small
such that the following properties hold true:

First, we require as usual that (with $\mathrm{n}(\cdot,0) := \mathrm{n}_{\partial^*\mathscr{A}(0)\cap\Omega}$)
\begin{align*}
\Psi\colon (\partial^*\mathscr{A}(0) {\cap} \Omega) \times (-r,r) \to \Rd[2],
\quad (x,s) \mapsto x + s\mathrm{n}(x,0)
\end{align*}
defines a $C^1$-diffeo onto its image~$\mathrm{im}\,\Psi$ such that
$\Psi\in C^1(\overline{\partial^*\mathscr{A}(0){\cap}\Omega}{\times}[-r,r])$ and 
$\Psi^{-1} \in C^1(\overline{\mathrm{im}\,\Psi})$.
Furthermore, denote by $L_{\pm}(0)$ the tangent line to $\overline{\partial^*\mathscr{A}(0){\cap}\Omega}$ 
at~$c_{\pm}(0)\in\partial \Omega$, respectively, and let $\vec{\tau}_{\pm}(0)\in L_{\pm}(0)$
the associated unit tangent to $\overline{\partial^*\mathscr{A}(0) \cap \Omega}$ 
at~$c_{\pm}(0)\in\partial \Omega$ pointing outside of~$\Omega$ (i.e., 
$c_{\pm}(0) {+} \ell\vec{\tau}_{\pm}(0)\in\Rd[2]\setminus\overline{\Omega}$ for all $0<\ell<r$ for $r$ small).
Denoting by~$\mathbb{H}_{\pm}(0)$ the open half-space given by
$\{x\in\Rd[2]\colon (x{-}c_{\pm}(0))\cdot\tau_{\pm}(0)>0\}$, we
next require that $r$ is small such that 
$\overline{B_{r}(y_{\pm})}\cap\overline{\partial^*\mathscr{A}(0){\cap}\Omega}=\{c_{\pm}(0)\}$
for all $y_{\pm}\in\partial B_{r}(c_{\pm}(0))\cap\overline{\mathbb{H}_{\pm}(0)}$.
By this choice of the scale~$r\in (0,1)$, the set
\begin{align}
\label{eq_extendedBulkInterface}
\widetilde I(0) := \overline{\partial^*\mathscr{A}(0){\cap}\Omega} \cup
\bigcup_{c_{\pm}(0)} \Big(\big(c_{\pm}(0){+}L_{\pm}(0)\big) 
\cap \mathbb{H}_{\pm}(0) \cap \overline{B_{\frac{r}{2}}(c_{\pm}(0))}\Big)
\end{align}
is an embedded, compact and orientable $C^1$-manifold with 
boundary~$\{c_{\pm}(0) {+} \frac{r}{2}\tau_{\pm}(0)\}$
extending the bulk interface~$\partial^*\mathscr{A}(0){\cap}\Omega$. We write~$\widetilde{\mathrm{n}}(\cdot,0)$
for the associated continuous unit normal vector field coinciding with~$\mathrm{n}(\cdot,0)$
along~$\partial^*\mathscr{A}(0){\cap}\Omega$. The second localization scale~$\delta\in (0,1)$
is now chosen sufficiently small such that
\begin{align}
\label{eq_extendedCoordinateChange}
\widetilde\Psi\colon \widetilde I(0) \times [-\delta r,\delta r] \to \Rd[2],
\quad (\widetilde x,\widetilde s) \mapsto \widetilde x + \widetilde s\,\widetilde{\mathrm{n}}(\widetilde x,0)
\end{align}
defines a homeomorphism onto its image~$\mathrm{im}\,\widetilde\Psi$, and such that
$\Omega\setminus\mathrm{im}\,\widetilde\Psi$ decomposes into two non-empty and disjoint connected
components~$\Omega_{\pm}(\widetilde\Psi)$ such that the set $\partial \Omega_{\pm}(\widetilde\Psi) \cap \Omega$
is given by $\widetilde\Psi(\widetilde I(0){\times}\{\pm \delta r\}) \cap \Omega$. 

With two such localization scales~$r,\delta\in (0,1)$ in place, we remark that
the projection onto the second coordinate of the inverse~$\widetilde\Psi^{-1}$
defines a $C^1$-function $\widetilde s$ which inside~$\Omega$ equals the signed
distance to~$\widetilde I(0)$.
Hence, by a slight abuse of notation we may extend $\tilde{s}$ to a $1$-Lipschitz continuous function
on~$\overline{\Omega}$ by means of
\begin{align}
\label{eq_auxSignedDistanceBulkInterface}
\widetilde s(x) :=
\begin{cases}
\pm\dist(x,\widetilde I(0)) & x \in \Omega_{\pm}(\widetilde\Psi), 
\\
\widetilde s(x) & x\in \Omega \cap \mathrm{im}\,\widetilde\Psi,
\end{cases}
\end{align}
which serves as a suitable extension of the signed distance function to the
initial bulk interface~$\partial^*\mathscr{A}(0) \cap \Omega$ (by which we again understand the
projection onto the second coordinate of the inverse~$\Psi^{-1}$).

\textit{Step 2: Definition of initial phase field~$ u_{\eps,0}$.}
Let~$\theta_0\colon\Rd[]\to (-1,1)$ denote the optimal transition
profile associated with the double-well potential~$W$, and fix a scale~$\eps\in (0,1)$.
Recalling the definition~\eqref{eq_auxSignedDistanceBulkInterface},
we then introduce an initial phase field by means of 
\begin{align}
\label{eq_initialPhaseField}
 u_{\eps,0}(x) := \theta_0\Big(\frac{\widetilde s(x)}{\eps}\Big),
\quad x \in \Omega.
\end{align}

\textit{Step 3: Properties of~$ u_{\eps,0}$ and optimal estimates for bulk energy contributions.}
That~$E_{\eps}[ u_{\eps,0}]<\infty$ and~\eqref{eq_boundInitialPhaseField}
hold true follows directly from the definitions~\eqref{eq_auxSignedDistanceBulkInterface}
and~\eqref{eq_initialPhaseField}. In terms of the required estimates, we claim that
\begin{align}
\label{eq:initialDataWellPreparednessBulk1}
\int_{\Omega} \frac{\eps}{2}|\nabla u_{\eps,0}|^2 
+ \frac{1}{\eps} W( u_{\eps,0}) - \nabla(\psi\circ u_{\eps,0})\cdot\xi(\cdot,0) \dx
&\lesssim \eps^2,
\\
\label{eq:initialDataWellPreparednessBulk2}
E_{\mathrm{bulk}}[ u_{\eps,0}|\mathscr{A}(0)] &\lesssim \eps^2.
\end{align}
In particular, in case of the specific choice~\eqref{eq_choiceBoundaryEnergyDensity}
for the boundary energy density, these two bounds immediately imply~\eqref{eq_wellPreparedInitialRelativeEnergy}
with optimal rate~$\eps^2$ since the boundary term in the definition~\eqref{eq_relEnergy}
of the relative energy simply vanishes in the special case~\eqref{eq_choiceBoundaryEnergyDensity}. 

For a proof of~\eqref{eq:initialDataWellPreparednessBulk1},
we split our task into two contributions by decomposing 
$\Omega = (\Omega \cap \{|\widetilde s\,| \geq \delta r\}) \cup (\Omega \cap \{|\widetilde s\,| < \delta r\})$.
By $|\nabla(\psi\circ u_{\eps,0})\cdot\xi(\cdot,0)| 
\leq \sqrt{2W( u_{\eps,0})}|\nabla u_{\eps,0}|$, the generalized chain rule for Lipschitz functions, 
$|\nabla\widetilde s\,|\leq 1$, and Young's inequality, we have
\begin{align*}
&\int_{\Omega {\cap} \{|\widetilde s\,| \,\geq \delta r\}} \frac{\eps}{2}|\nabla u_{\eps,0}|^2 
+ \frac{1}{\eps} W( u_{\eps,0}) - \nabla(\psi\circ u_{\eps,0})\cdot\xi(\cdot,0) \dx
\\&
\leq \frac{2}{\eps} \int_{\Omega {\cap} \{|\widetilde s\,| \,\geq \delta r\}} 
\Big|\theta_0'\Big(\frac{\widetilde s(x)}{\eps}\Big)\Big|^2 + W\Big(
\theta_0\Big(\frac{\widetilde s(x)}{\eps}\Big)\Big) \dx,
\end{align*}
which thanks to $\theta_0'(r)=\sqrt{2W(\theta_0(r))}$
for all~$r\in\Rd[]$ and the exponential decay of~$|\theta_0'|$ upgrades to 
\begin{align}
\label{eq_auxEstimate1}
\int_{\Omega {\cap} \{|\widetilde s\,| \,\geq \delta r\}} \frac{\eps}{2}|\nabla u_{\eps,0}|^2 
+ \frac{1}{\eps} W( u_{\eps,0}) - \nabla(\psi\circ u_{\eps,0})\cdot\xi(\cdot,0) \dx
\lesssim \eps^2.
\end{align}
For an estimate of the contribution from $\Omega \cap \{|\widetilde s\,| < \delta r\}$,
we note that $\nabla u_{\eps,0}(x)=\frac{1}{\varepsilon}\theta_0'(\frac{\widetilde s(x)}{\eps})\widetilde{\mathrm{n}}(P_{\widetilde I(0)}(x),0)$
and thus $\nabla(\psi\circ u_{\eps,0})(x) = \frac{1}{\varepsilon}|\theta_0'(\frac{\widetilde s(x)}{\eps})|^2\widetilde{\mathrm{n}}(P_{\widetilde I(0)}(x),0)$
for all $x\in \Omega \cap \{|\widetilde s\,| < \delta r\} \subset\mathrm{im}\,\widetilde\Psi$,
where the map~$P_{\widetilde I(0)}$ denotes the projection onto the nearest point on~$\widetilde I(0)$.
In particular,
\begin{align}
\nonumber
&\int_{\Omega {\cap} \{|\widetilde s\,| \,< \delta r\}} \frac{\eps}{2}|\nabla u_{\eps,0}|^2 
+ \frac{1}{\eps} W( u_{\eps,0}) - \nabla(\psi\circ u_{\eps,0})\cdot\xi(\cdot,0) \dx
\\& \label{eq_auxEstimate2}
= -\int_{\Omega {\cap} \{|\widetilde s\,| \,< \delta r\}}
\frac{1}{\varepsilon}\Big|\theta_0'\Big(\frac{\widetilde s(x)}{\eps}\Big)\Big|^2
\widetilde{\mathrm{n}}(P_{\widetilde I(0)}(x),0)\cdot\big(\xi(\cdot,0)
- \widetilde{\mathrm{n}}(P_{\widetilde I(0)}(\cdot),0)\big) \dx.
\end{align}
We claim that 
\begin{align}
\label{eq_auxEstimate3}
\big|\widetilde{\mathrm{n}}(P_{\widetilde I(0)}(x),0)\cdot\big(\xi(x,0)
- \widetilde{\mathrm{n}}(P_{\widetilde I(0)}(x),0)\big)\big|
\lesssim \widetilde s^{\,2}(x)
\end{align}
for all $x \in \Omega \cap \{|\widetilde s\,| < \delta r\}$.
Once the estimate~\eqref{eq_auxEstimate3} is established, it follows
in combination with~\eqref{eq_auxEstimate2} and the exponential decay of~$|\theta_0'|$ 
(together with a transformation argument) that
\begin{align}
\nonumber
&\int_{\Omega {\cap} \{|\widetilde s\,| \,< \delta r\}} \frac{\eps}{2}|\nabla u_{\eps,0}|^2 
+ \frac{1}{\eps} W( u_{\eps,0}) - \nabla(\psi\circ u_{\eps,0})\cdot\xi(\cdot,0) \dx
\\& \label{eq_auxEstimate4}
\lesssim \eps\int_{\Omega {\cap} \{|\widetilde s\,| \,< \delta r\}}
\Big|\theta_0'\Big(\frac{\widetilde s(x)}{\eps}\Big)\Big|^2
\frac{\widetilde s^{\,2}(x)}{\eps^2} \dx \lesssim  \eps^2.
\end{align}
Obviously, the estimates~\eqref{eq_auxEstimate1} and~\eqref{eq_auxEstimate4}
then imply the desired bound~\eqref{eq:initialDataWellPreparednessBulk1}, 
so that it remains to verify~\eqref{eq_auxEstimate3}. 

To this end, we observe first that for all 
$x \in \Omega \cap \{|\widetilde s\,| < \delta r\}\subset\mathrm{im}\,\widetilde\Psi$
it holds $\widetilde{\mathrm{n}}(P_{\widetilde I(0)}(x),0) = 
\mathrm{n}(P_{\overline{\partial^*\mathscr{A}(0){\cap} \Omega}}(x),0)$
due to the choice of~$r\in (0,1)$ and the definition of the extended
interface~$\widetilde I(0)$. Hence,
\begin{align*}
&\widetilde{\mathrm{n}}(P_{\widetilde I(0)}(x),0)\cdot\big(\xi(x,0)
- \widetilde{\mathrm{n}}(P_{\widetilde I(0)}(x),0)\big)
\\&
=\mathrm{n}(P_{\overline{\partial^*\mathscr{A}(0){\cap} \Omega}}(x),0)\cdot\big(\xi(x,0)
- \mathrm{n}(P_{\overline{\partial^*\mathscr{A}(0){\cap} \Omega}}(x),0)\big)
\end{align*}
for all $x \in \Omega \cap \{|\widetilde s\,| < \delta r\}$. In particular, 
a Taylor expansion argument based on the conditions~\eqref{eq_consistencyProperty} 
and the regularity~\eqref{eq_regularityXi} entails
\begin{align*}
\big|\widetilde{\mathrm{n}}(P_{\widetilde I(0)}(x),0)\cdot\big(\xi(x,0)
- \widetilde{\mathrm{n}}(P_{\widetilde I(0)}(x),0)\big)\big|
\lesssim 
\dist^2(x,\overline{\partial^*\mathscr{A}(0){\cap} \Omega})
\end{align*}
for all $x \in \Omega \cap \{|\widetilde s\,| < \delta r\}$.
The previous display can be post-processed to~\eqref{eq_auxEstimate3}
since $\dist(\cdot,\overline{\partial^*\mathscr{A}(0){\cap} \Omega}) 
\lesssim |\widetilde s\,|$ in $\Omega \cap \{|\widetilde s\,| < \delta r\}$.
Indeed, the latter claim is trivially true in the 
image $\widetilde\Psi((\overline{\partial^*\mathscr{A}(0){\cap} \Omega}){\times}(-\delta r,\delta r))$
as $\dist(\cdot,\overline{\partial^*\mathscr{A}(0){\cap} \Omega}) = |\widetilde s\,|$ on this set. For the remaining points 
$x\in(\Omega {\cap} \{|\widetilde s\,| {<} \delta r\}) \setminus 
\widetilde\Psi((\overline{\partial^*\mathscr{A}(0){\cap} \Omega}){\times}(-\delta r,\delta r)) \subset \mathrm{im}\,\widetilde\Psi$,
the claim follows from recognizing that for such points 
$P_{\overline{\partial^*\mathscr{A}(0){\cap} \Omega}}(x) \in \{c_{\pm}(0)\}$
and that the angle formed by the vectors $x - P_{\overline{\partial^*\mathscr{A}(0){\cap} \Omega}}(x)$
and $P_{\widetilde I(0)}(x) - P_{\overline{\partial^*\mathscr{A}(0){\cap} \Omega}}(x)$ 
is bounded away from zero uniformly (which
in turn holds true since the bulk interface intersects the domain boundary non-tangentially).

We next turn to the proof of the estimate~\eqref{eq:initialDataWellPreparednessBulk2}.
Recalling~\eqref{eq_bulkError}, we start by plugging in definitions in form of 
\begin{align*}
E_{\mathrm{bulk}}[ u_{\eps,0}|\mathscr{A}(0)]
&= \int_{\mathscr{A}(0)} |\vartheta(\cdot,0)|
\bigg|\int_{\theta_0(\frac{\widetilde s(x)}{\eps})}^{1}
\sqrt{2W(\ell)}\,\mathrm{d}\ell\bigg| \dx
\\&~~~
+ \int_{\Omega\setminus\mathscr{A}(0)} |\vartheta(\cdot,0)|
\bigg|\int_{-1}^{\theta_0(\frac{\widetilde s(x)}{\eps})}
\sqrt{2W(\ell)}\,\mathrm{d}\ell\bigg| \dx,
\end{align*}
so that one obtains the preliminary estimate
\begin{align*}
E_{\mathrm{bulk}}[ u_{\eps,0}|\mathscr{A}(0)]
&\lesssim \int_{\mathscr{A}(0)} |\vartheta(\cdot,0)|
\Big|\theta_0\Big(\frac{\widetilde s(x)}{\eps}\Big) - 1\Big| \dx
\\&~~~
+ \int_{\Omega\setminus\mathscr{A}(0)} |\vartheta(\cdot,0)|
\Big|\theta_0\Big(\frac{\widetilde s(x)}{\eps}\Big) - ({-}1)\Big| \dx.
\end{align*}
Both terms on the right hand side of the previous display can again be treated by decomposing 
$\Omega = (\Omega \cap \{|\widetilde s\,| \geq \delta r\}) \cup (\Omega \cap \{|\widetilde s\,| < \delta r\})$.
Throughout $\Omega \cap \{|\widetilde s\,| \geq \delta r\}$, one then simply capitalizes on the
fact that the optimal profile~$\theta_0(r)$ converges exponentially fast to~$\pm 1$ as $r\to\pm\infty$ 
and that $\chi=0,1$ correlates with the sign of~$\tilde{s}$.
Throughout $\Omega \cap \{|\widetilde s\,| < \delta r\}$, one in addition makes use of the
Lipschitz estimate $|\vartheta(\cdot,0)| \lesssim \dist(\cdot,\overline{\partial^*\mathscr{A}(0){\cap}\Omega})
\lesssim |\widetilde s\,|$ (recall for the first inequality that $\vartheta(\cdot,0)=0$
along $\overline{\partial^*\mathscr{A}(0){\cap}\Omega}$). In summary, 
one obtains~\eqref{eq:initialDataWellPreparednessBulk2}.

\textit{Step 4: Estimate for boundary energy contribution.}
We claim that 
\begin{align}
\label{eq:initialDataWellPreparednessBoundary}
0 \leq \int_{\partial \Omega} \sigma( u_{\eps,0}) - 
\psi( u_{\eps,0})\cos\alpha \dH \lesssim \eps.
\end{align}
Note that together with the estimates from the previous step, we in particular 
obtain the asserted bound~\eqref{eq_wellPreparedInitialRelativeEnergy}
once~\eqref{eq:initialDataWellPreparednessBoundary} is proven.

Due to the definition~\eqref{eq_auxIndicatorPhaseField} and the 
compatibility conditions between~$\sigma$ and~$\psi$ 
at the endpoints~$\pm 1$ from~\eqref{eq_propBoundaryEnergyDensity2},
it follows again from the exponentially fast convergence $\theta_0(r) \to \pm 1$ 
as $r\to\pm\infty$ that the contribution coming from the integration over the set $\partial\Omega \cap
\{|\widetilde s| \geq \delta r\}$ is of higher order compared 
with the claim~\eqref{eq:initialDataWellPreparednessBoundary}.
On $\partial\Omega \cap \{0 \leq |\widetilde s| < \delta r\}$ we can additionally 
use an integral transformation and a scaling argument to 
obtain~\eqref{eq:initialDataWellPreparednessBoundary}.
\end{proof}

\section*{Acknowledgments}
This project has received funding from the European Research Council 
(ERC) under the European Union's Horizon 2020 research and innovation 
programme (grant agreement No 948819)
\begin{tabular}{@{}c@{}}\includegraphics[width=8ex]{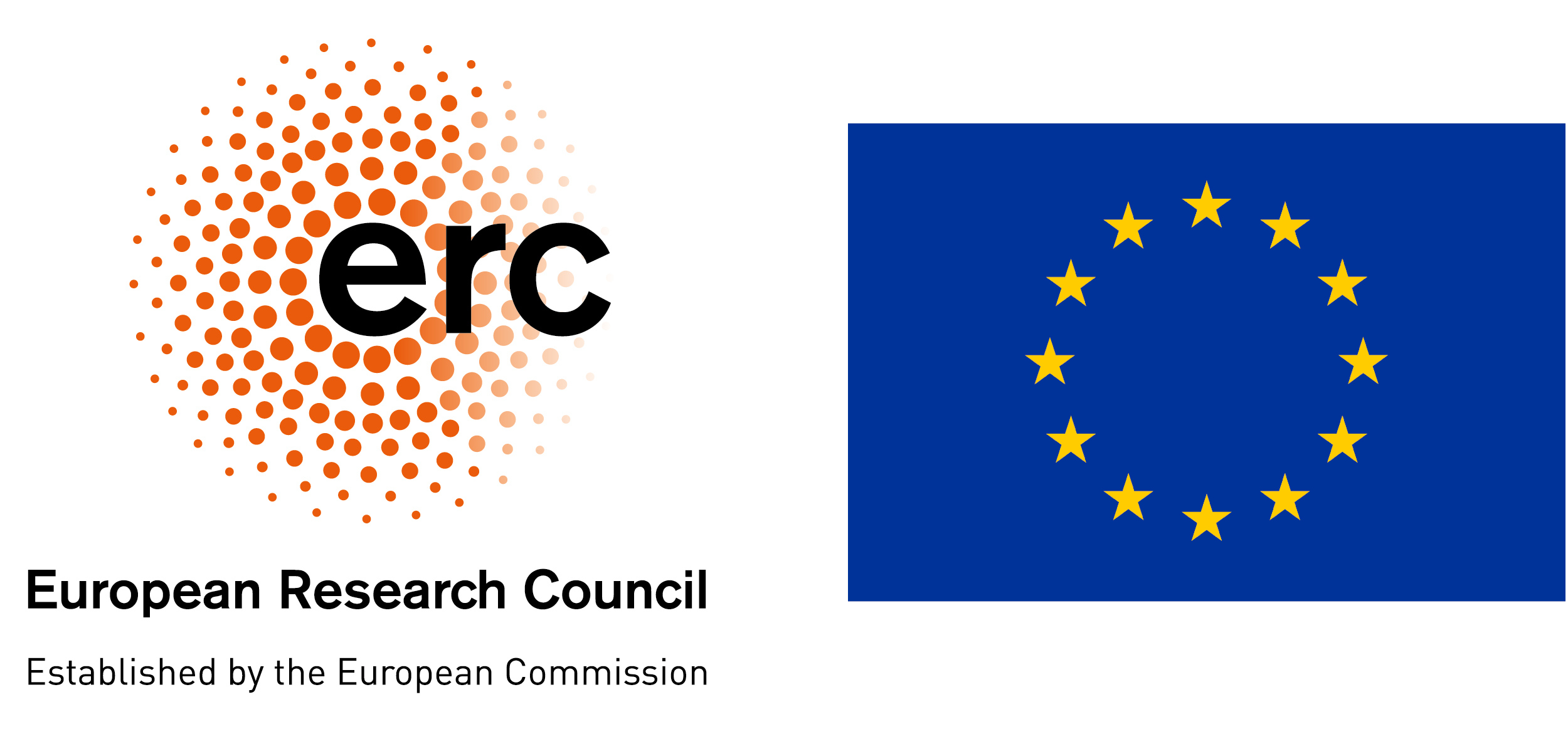}\end{tabular}, and 
from the Deutsche Forschungsgemeinschaft (DFG, German Research Foundation) under 
Germany's Excellence Strategy -- EXC-2047/1 -- 390685813.

\bibliographystyle{abbrv}
\bibliography{Convergence_rates_fixed_contact_angle}

\end{document}